\documentclass[a4paper,11pt]{article}
\usepackage{tikz}
\usepackage[czech,english]{babel}
\usepackage[utf8]{inputenc}
\usepackage{
   amsmath
	,amsfonts
	,amssymb
	,amsthm
	,enumerate
	,xypic
	,todonotes
	,hyperref
	,cite
	,datetime
	,mathrsfs}

\usepackage[all,cmtip]{xy}
\usepackage[left=3cm, right=3cm, top=3cm, bottom=4.5cm]{geometry}
\usepackage{turnstile}

\newenvironment{smallarray}[1]
 {\null\,\vcenter\bgroup\scriptsize
  \arraycolsep=.13885em
  \hbox\bgroup$\array{@{}#1@{}}}
 {\endarray$\egroup\egroup\,\null}

\newcommand{\keywords}[1]{\noindent\textbf{Key words and phrases: } #1}
\newcommand{\MSC}[1]{\noindent\textbf{2010 Mathematics Subject Classification: } #1}

\numberwithin{equation}{section}

\newcommand{\adjunct}[2]{\nsststile{#2}{#1}}

\newcommand{\Nearrow}{\rotatebox[origin=c]{45}{$\Rightarrow$}}
\newcommand{\Nwarrow}{\rotatebox[origin=c]{135}{$\Rightarrow$}}
\newcommand{\Searrow}{\rotatebox[origin=c]{-45}{$\Rightarrow$}}
\newcommand{\Swarrow}{\rotatebox[origin=c]{225}{$\Rightarrow$}}

\DeclareMathOperator{\id}{id}

\renewcommand{\textbf}[1]{\text{\fontseries{b}\selectfont{\upshape #1}}}

\def\Cate#1{\textbf{#1}}
	
	\def\S{\Cate{S}}
	\def\cat{\Cate{Cat}}
	\def\Cat{\Cate{CAT}}
\def\LAdj{\mathrm{LAdj}}
\def\RAdj{\mathrm{RAdj}}
\def\dia{\mathrm{dia}}

\def\op{{\mathrm{op}}}
\renewcommand\hom{\mathrm{Hom}}
\def\ho{\mathrm{Ho}}
\def\Ch{\mathrm{Ch}}

\def\hocolim{\mathrm{Hocolim}}

\def\holim{\mathrm{Holim}}

\def\Ob{\mathrm{Ob}}
\def\Der{\mathbf {D}}

\def\ker{\mathrm{Ker}}
\def\Ker{\mathrm{Ker}}

\def\pt{\mathrm{pt}}
\def\colim{\mathrm{colim}}
\def\Ab{\mathrm{Ab}}

\def\Hom{\mathrm{Hom}}

\def\pr{\mathrm{pr}}
\def\Set{\mathrm{Set}}
\def\sSet{\mathrm{sSet}}
\def\Top{\mathrm{Top}}
\def\Sp{\mathrm{Sp}}
\def\cD{\mathcal{D}}

\def\Z{\mathbb{Z}}
\def\N{\mathbb{N}}

\def\H{\mathcal H}
\def\U{\mathcal U}
\def\V{\mathcal V}

\def\D{\mathbb D}

\def\class#1{\mathcal{#1}}

	\def\B{\class{B}}
	
	\def\C{\class{C}}
	
	\def\W{\mathcal W}
	
	\def\Q{\mathcal Q}
	\def\G{\mathcal G}
\renewcommand\t{\mathbf{t}}

\def\add{\mathrm{add}}
\def\Add{\mathrm{Add}}

\def\bbone{\mathbf{1}}
\def\bbtwo{\mathbf{2}}

\def\mod#1{\mathrm{Mod}\text{-}#1}

\theoremstyle{plain}
  \newtheorem{maintheorem}{Theorem}
  \newtheorem{maincorollary}[maintheorem]{Corollary}
  
\theoremstyle{definition}
	\newtheorem{mainquestion}{Question}
    
\theoremstyle{plain}
	\newtheorem{theorem}{Theorem}[section]
	\newtheorem{lemma}[theorem]{Lemma}
	\newtheorem{proposition}[theorem]{Proposition}
	\newtheorem{corollary}[theorem]{Corollary}
\theoremstyle{remark}
	\newtheorem{remark}[theorem]{Remark}
\theoremstyle{definition}
	\newtheorem{example}[theorem]{Example}

	\newtheorem{definition}[theorem]{Definition}
	\newtheorem{notation}[theorem]{Notation}

\title{$t$-Structures on stable derivators and Grothendieck hearts}

\author{Manuel Saor\'{\i}n\footnote{The first named author was supported by the research projects from
the Ministerio de Econom\'{\i}a y Competitividad of Spain (MTM2016-77445-P) and the Fundaci\'on `S\'eneca' of
Murcia (19880/GERM/15), both with a part of FEDER funds.} \and Jan \v{S}\v{t}ov\'{\i}\v{c}ek\footnote{The second named author was supported by Neuron Fund for Support of Science.} \and Simone Virili\footnote{The third named author was supported by the Ministerio de Econom\'{\i}a y Competitividad of
Spain via a grant `Juan de la Cierva-formaci\'on'. He was also supported by the
Fundaci\'on `S\'eneca' of Murcia (19880/GERM/15) with a part of FEDER funds.}}

\begin{document}

\maketitle

\begin{abstract}
We prove that, given any strong and stable derivator and a $t$-structure in its base triangulated category $\cal D$, the $t$-structure canonically lifts to all the (coherent) diagram categories and each incoherent diagram in the heart uniquely lifts to a coherent one.
We use this to show that the $t$-structure being compactly generated implies that the coaisle is closed under directed homotopy colimits which, in turn, implies that the heart is an (Ab.$5$) Abelian category. If, moreover, $\cal D$ is a well-generated algebraic or topological triangulated category, then the heart of any accessibly embedded (in particular, compactly generated) $t$-structure has a generator.
As a consequence, it follows that the heart of any compactly generated $t$-structure of a well-generated algebraic or topological triangulated category is a Grothendieck Abelian category.
\end{abstract}

\keywords $t$-structures, heart, derivator, Grothendieck Abelian category.\\
\MSC  18E30, 18E40, 18E15.

\setcounter{tocdepth}{1}
\tableofcontents

\addcontentsline{toc}{section}{Introduction}
\section*{Introduction}

$t$-Structures in triangulated categories were introduced by Beilinson, Bernstein and Deligne \cite{BBD} in their study of perverse
sheaves on an algebraic or analytic variety. A $t$-structure in a triangulated category $\mathcal{D}$  is a pair of full
subcategories satisfying a suitable set of axioms (see the precise definition
in Subsection~\ref{tria_and_t_subs}) which guarantees that their intersection is an
Abelian category $\mathcal{H}$, called the heart of the $t$-structure. One then naturally defines a cohomological functor
\[
{H}\colon\mathcal{D}\longrightarrow\mathcal{H},
\] 
which allows to develop an intrinsic (co)homology theory,
where the homology ``spaces'' are objects of $\mathcal{D}$
itself.
$t$-Structures have been used in many
branches of mathematics, with special impact in algebraic geometry, algebraic topology
and representation theory of algebras. 

\medskip
Given a $t$-structure in a triangulated category $\cal D$, and considering the induced Abelian category $\H$, a natural problem consists in finding necessary and sufficient conditions on
the $t$-structure and on the ambient category for the heart to be a ``nice'' Abelian
category. When our triangulated category has (co)products, the category $\H$ is known to be (co)complete (see \cite[Prop.\,3.2]{parra2015direct}) and, using the classical hierarchy of Abelian categories due to Grothendieck \cite{Grothendieck}, the natural question  
 is the following: 

\begin{mainquestion} \label{mainques.when-grothendieck}
When is the heart $\H$ a Grothendieck Abelian category?
\end{mainquestion}

As one might expect, the real issue is to prove that $\H$ has exact directed colimits. In this respect, we encounter a phenomenon which seems invisible to the triangulated category $\cal D$ alone, namely directed homotopy colimits and the question on whether $\H$ is closed under these. To work with homotopy colimits, we need a certain enhancement of $\cal D$ and we choose Grothendieck derivators, as this is in some sense the minimal homotopy theoretic framework where a well-behaved calculus of homotopy (co)limits is available. This said, we are immediately led to the second main problem of the paper:

\begin{mainquestion} \label{mainques.t-str-and-derivators}
How do $t$-structures interact with strong and stable Grothendieck derivators?
\end{mainquestion}

The study of Question~\ref{mainques.when-grothendieck} has a long tradition in algebra. In its initial steps, the focus was almost exclusively put on the case of the so-called Happel-Reiten-Smal{\o} \mbox{$t$-structures} introduced in \cite{HRS} and treated in much greater generality in~\cite[Sec.\,5.4]{bondal-vdb}.
These are $t$-structures on a derived category $\Der(\cal G)$ of an Abelian category $\cal G$ induced by a torsion pair in $\cal G$.
The study of conditions for the heart of the Happel-Reiten-Smal{\o} $t$-structure in $\Der(\mathcal{G})$, for a Grothendieck or module category $\mathcal{G}$, to be again a Grothendieck or a module category, has received a lot of attention in recent years (see \cite{HKM, CGM, CMT, MT, parra2015direct, parra2016addendum} and \cite{StCo}).
Let us remark that the first named author with C.\,Parra \cite{parra2015direct, parra2016addendum} gave a complete answer to Question~\ref{mainques.when-grothendieck} in this particular case: the heart of the Happel-Reiten-Smal{\o} $t$-structure $\t_\tau$, associated to a torsion pair $\tau=(\mathcal T,\mathcal F)$ in a Grothendieck Abelian category $\mathcal{G}$, is again a Grothendieck Abelian category if, and only if, the torsion-free class $\mathcal{F}$ is closed under taking directed colimits in $\mathcal{G}$.  

\smallskip
When more general $t$-structures are considered, the answers to Question~\ref{mainques.when-grothendieck} are more scarce and they are often inspired in one or another way by tilting theory. In a sense, the classical derived Morita Theorems of Rickard~\cite{rickard1991derived} (for the bounded setting) and Keller~\cite{kel94-dg} (for the unbounded case) can be seen as the first examples where an answer to the problem is given. Namely, if $A$ and $B$ are ordinary algebras
 and ${}_BT_A$ is a two-sided tilting complex (see \cite{rickard1991derived}), then the triangulated equivalence $-\otimes_B^\mathbb{L}T\colon\Der(B)\tilde{\longrightarrow}\Der(A)$ takes the canonical $t$-structure $(\Der^{\leq 0}(B),\Der^{\geq 0}(B))$ to the pair $(T^{\perp_{>0}},T^{\perp_{<0}})$, which is then a $t$-structure in $\Der(A)$, whose heart is equivalent to $\mod B$. This includes the case of a classical ($n$-)tilting module in the sense of \cite{miyashita-tilt}. The dual of a (not necessarily classical) ($n$-)tilting $A$-module is that of a (big) ($n$-)cotilting $A$-module $Q$, in which case the second named author proved that $(_{}^{\perp_{<0}}Q,_{}^{\perp_{>0}}Q)$ is a $t$-structure in $\Der(A)$ whose heart is a Grothendieck Abelian category (see \cite[Thm.\,6.2]{stovicek2014derived}). These two  results have recently been extended to include all silting sets of compact objects and all pure-injective cosilting sets in a compactly generated triangulated category (see \cite[Prop.\,4.2]{NSZ}, also for the used terminology). Results saying that cosilting $t$-structures have Grothendieck hearts under appropriate assumptions also include~\cite[Thm.\,3.6]{hugel2016torsion} and~\cite[Prop.\,3.10]{MV-silting}, whereas conditions under which the $t$-structure $(T^{\perp_{>0}},T^{\perp_{<0}})$ obtained from a non-classical tilting module $T$ has a Grothendieck heart were given in~\cite{bazzoni2016t}.
A common feature of the results in \cite{stovicek2014derived,NSZ,hugel2016torsion,MV-silting} is that the heart is proved to be a Grothendieck Abelian category rather indirectly, using the pure-injectivity of certain cotilting or cosilting objects and ideas from model theory of modules.

\smallskip
Last but not least, there is a family of results which provides evidence that $t$-structures generated by a set of compact objects should, under all reasonable circumstances, have Grothendieck hearts.
Briefly summarizing, \cite[Thm.\,4.10]{parra2017hearts} establishes this result for any compactly generated $t$-structure in the derived category $\Der(R)$ of a commutative ring $R$ which is given by a left bounded filtration by supports, \cite[Coro.\,4.10]{hugel2016torsion} gives the same result for any non-degenerate compactly generated $t$-structure in an algebraic compactly generated triangulated category,
and finally Bondarko establishes such a result in~\cite[Thm.\,5.4.2]{bondarko} for practically all triangulated categories which one encounters in practice.

\medskip
Our approach here is rather different from the ones above and, in some sense, much more direct. It is more in the spirit of Lurie's~\cite[Rem.\,1.3.5.23]{Lurie_higher_algebra}, where a criterion for a $t$-structure to have a Grothendieck heart is given in the language of $\infty$-categories. Our aim is to reach the corresponding criterion for exactness of directed colimits in the heart faster and in a way hopefully more accessible to the readers concerned with the representation theory of associative algebras and related fields.

\smallskip
To outline our strategy, consider a $t$-structure $\t = (\U,\Sigma\V)$ on a triangulated category $\cal D$ with coproducts and let $\H=\U\cap\Sigma\V$ be the heart of $\t$. Then we have the following well-known chain of implications:

\begin{center}
\begin{tabular}{ccccc}
$\t$ is compactly &$\implies$& $\V$ is closed &$\implies$& $\H$ has exact \\
generated (see p.\pageref{compactly_gen_t_st}) && under coproducts  && coproducts.
\end{tabular}
\end{center}
However, $\V$ being closed under coproducts is \emph{not} enough to have exact directed colimits in $\H$. This follows from the main results of \cite{parra2015direct, parra2016addendum} (see also Example~\ref{ex.HRS t-structure smash but non-htpy} below).
What we need instead is a stronger condition: we assume that $\V$ is closed under directed homotopy colimits. For instance, in the case of the derived category $\cal D = \Der(\G)$ of a Grothendieck Abelian category $\G$, the homotopy colimit is the total left derived functor
\[ \hocolim_I = \mathbf{L}\colim_I\colon \Der(\G^I) \longrightarrow \Der(\G) \]
of the usual colimit functor $\colim_I\colon \G^I \to \G$. The problem is, of course, that albeit there always exists a canonical comparison functor $\Der(\G^I) \to \Der(\G)^I$, it is typically far from being an equivalence.
At this point, the language of derivators naturally enters the scene, as the categories of the form $\Der(\G^I)$ can be naturally assembled to form a derivator (for all the undefined terminology we refer to Section \ref{prel_on_ders_sec}). To see this, one should recall that, when $\mathcal{M}$ is a cofibrantly generated model category with $\mathcal{W}$ its class of weak equivalences (e.g., the usual injective model structure of $\Ch(\G)$, with $\W$ the class of quasi-isomorphisms), the functor category $\mathcal{M}^I$ admits a cofibrantly generated model structure, with weak equivalences calculated level-wise, for each small category $I$, and the assignment $I\mapsto\mathbb{D}(I):=\text{Ho}(\mathcal{M}^I)$ gives a well-defined derivator (see\cite{Cis03}), i.e.\ a $2$-functor
\[
\D\colon \cat^{\op}\longrightarrow \Cat
\]
which satisfies certain axioms, where $\cat$ is the $2$-category of small categories and $\Cat$ is the 2-``category" of all categories. Furthermore, $\D$ is strong and stable provided $\mathcal{M}$ is a stable model category. The axioms in particular imply that the natural range of this $2$-functor is naturally the 2-``category'' of all triangulated categories, that is, $\mathbb{D}(I)$ is a triangulated category for each $I\in \cat$. A prototypical example is precisely the assignment $\D\colon I \mapsto \Der(\G^I)$ for a Grothendieck Abelian category $\G$.

\medskip
The main results of the paper will be proven in general for such a strong and stable derivator $\D$. Note also that, denoting by $\bbone$ the one-point category with the identity morphism only, and letting $\mathcal{D}:=\D(\bbone)$ the base category of $\D$, one usually views $\D$ as an enhancement of the triangulated category $\cal D$, which is in some sense the minimalistic enhancement which allows for a well-behaved calculus of homotopy (co)limits or, more generally, homotopy Kan extensions (see \cite{Moritz}). More precisely, a homotopy colimit functor $\hocolim_I\colon \D(I)\to\D(\bbone)$ is simply a left adjoint to $\D(\pt_I)\colon \D(\bbone)\to\D(I)$, where $\pt_I$ is the unique functor $I\to\bbone$ in $\cat$. The existence of the adjoint is ensured by the axioms of a derivator and this notion of homotopy colimit is consistent with our previous definition of $\hocolim_I$ as the total left derived functor of $\colim_I$.

\smallskip
The advantage of derivators is that we now can give  a precise meaning to what it means that $\V$ is closed under directed homotopy colimits but, on the other hand, we now fully hit Question~\ref{mainques.t-str-and-derivators}. We started with a $t$-structure $\t=(\U,\Sigma\V)$ in the base category $\D(\bbone)$, and it is a natural question whether this $t$-structure lifts to the other triangulated categories $\D(I)$, with $I\in\cat$. Even when it does, what we are really concerned with is the relation between the heart in $\D(I)$ and the $I$-shaped diagrams in the heart of $\D(\bbone)$. Luckily, both these problems have a very natural solution that is condensed in the following theorem (for the proof see Section~\ref{sec_tstr}):

\begin{maintheorem}\label{mainthm.lift_tstructure_general_thm}
Let $\D\colon \cat^{\op}\to \Cat$ be a strong and stable derivator, and $\t=(\U,\Sigma \V)$ a $t$-structure in $\D(\bbone)$. If we let 
\[\begin{split}
&\U_I:=\{X\in \D(I):X_i\in \U,\ \forall i\in I\},\\
&\V_I:=\{Y\in \D(I):Y_i\in \V,\ \forall i\in I\},
\end{split}\] 
then $\t_I:=(\U_I,\Sigma \V_I)$ is a $t$-structure in $\D(I)$. Furthermore, the diagram functor 
\[
\dia_I\colon \D(I)\longrightarrow \D(\bbone)^I
\] 
induces an equivalence $\H_{I}\cong \H^I$ between the heart $\H_{I}$ of $\t_I$ and the category $\H^I$ of diagrams of shape $I$ in the heart of $\t$.
\end{maintheorem}

Now we can state our main answer to the part of Question~\ref{mainques.when-grothendieck} which is concerned with the exactness of directed colimits in the heart of a $t$-structure. This result will be proved in Section~\ref{sec_dirlim_heart}.

\begin{maintheorem} \label{mainthm.htpy smashing has AB5 heart}
Let $\D\colon \cat^{\op}\to \Cat$ be a strong stable derivator and let $\t=(\U,\Sigma\V)$ be a $t$-structure in $\D(\bbone)$. Then we have the implications:

\begin{center}
\begin{tabular}{ccccc}
$\t$ is compactly &$\implies$& $\V$ is closed under &$\implies$& $\H$ has exact \\
generated && directed homotopy colimits && directed colimits.
\end{tabular}
\end{center}
\end{maintheorem}

What remains is to give a criterion for the heart to have a generator. As it turns out, this is, unlike the (Ab.$5$) condition, a problem of mostly a technical nature. For most triangulated categories (or derivators) and $t$-structures arising in practice, the answer is affirmative. In the next theorem we give a general criterion in the setting of what one may call ``accessible stable derivators'', that is, the ones associated with a stable combinatorial model category. This will be treated in Section~\ref{sec_generators}.

\begin{maintheorem} \label{mainthm.Grothendieck heart}
Let $(\C,\W,\cal B,\cal F)$ be a stable combinatorial model category, $\cal D=\ho(\C)$ the triangulated homotopy category and $\t=(\U,\Sigma\V)$ a $\lambda$-accessibly embedded $t$-structure in $\cal D$ for some infinite regular cardinal $\lambda$ (e.g., $\cal D$ a well-generated algebraic or topological triangulated category, with a $t$-structure generated by a small set of objects).
\\
Then the heart $\H=\U\cap\Sigma \V$ of $\t$ has a generator. If, in particular, $\t$ is homotopically smashing (equivalently, $\lambda=\aleph_0$), then $\H$ is a Grothendieck Abelian category.
\end{maintheorem}

As an immediate consequence, we obtain the following corollary which provides an alternative to~\cite[Thm.\,5.4.2]{bondarko} in showing that the heart of a compactly generated $t$-structure is, in practice, always a Grothendieck Abelian category (see also Remark \ref{comparison_rem}).

\begin{maincorollary} \label{maincor.Grothendieck heart}
Let $\cal D=\ho(\C)$, where $\C$ is a combinatorial stable model category. If $\t=(\mathcal{U},\Sigma\mathcal{V})$ is a compactly generated $t$-structure in $\cal D$, then the heart $\H=\mathcal{U}\cap\Sigma\mathcal{V}$ is a Grothendieck Abelian category.
\end{maincorollary}

\bigskip\noindent
{\bf Acknowledgement.} 
It is a pleasure for us to thank Fritz H\"ormann and Moritz Groth for helpful discussions and suggestions. We are also grateful to Rosie Laking and Gustavo Jasso for  pointing out some analogies with results of Jacob Lurie in the $\infty$-categorical setting, and to Mikhail Bondarko for telling us about his results in \cite{bondarko}. Finally, we are indebted to an anonymous referee whose numerous remarks and suggestions have helped to greatly improve the exposition and clarity of our results.

\section{Preliminaries and notation}\label{sec_prelim}

In this first section we fix most of the notations and conventions about general category theory that will be needed in the rest of the paper. This includes basic facts about additive and Abelian categories, torsion pairs, triangulated categories, $t$-structures, categories with weak equivalences and model categories. All the results included in this section are known, so the proofs are omitted in favor of suitable references to the existing literature.

\subsection{Conventions and basic results in category theory}
Given a category $\C$ and two objects $x$ and $y$ in $\C$, we denote by $\C(x,y):=\hom_\C(x,y)$ the $\hom$-set of all morphism $x\to y$ in $\C$. Throughout the paper, all the subcategories will be assumed to be full, so we just say ``subcategory'' to mean ``full subcategory''. Similarly, we  generally drop the distinction between a subcategory $\S\subseteq \C$ and a subclass $\S\subseteq \Ob(\C)$, as the two things univocally determine each other. 

We denote by $\Ab$ the category of Abelian groups.

\medskip\noindent
\textbf{Ordinals.}
Any ordinal $\lambda$ can be viewed as a category in the following way: the objects of $\lambda$ are the ordinals $\alpha<\lambda$ and, given $\alpha$, $\beta<\lambda$, the $\hom$-set $\lambda(\alpha,\beta)$ is a point if $\alpha\leq \beta$, while it is empty otherwise. Following this convention,
\begin{itemize}
\item $\bbone=\{0\}$ is the category with one object and no non-identity morphisms;
\item $\bbtwo=\{0\to 1\}$ is the category with one non-identity morphism;
\item in general, $\mathbf{n}=\{0\to 1\to\ldots\to (n-1)\}$, for any $n\in\N_{>0}$.
\end{itemize}

\medskip\noindent
\textbf{Functor categories, limits and colimits.}
A category $I$ is said to be ({\em skeletally}) {\em small} when (the isomorphism classes of) its objects form a set. If $\C$ and $I$
are an arbitrary and a small category, respectively, a functor $I\to \C$ is said to be a {\em diagram} in $\C$ of shape $I$. The category of diagrams in $\C$ of shape $I$, and natural transformations between them, is denoted by $\C^I$.
A diagram $X$ of shape $I$, is also denoted as $(X_i)_{i\in I}$, where $X_i
:= X(i)$ for each $i$ in $I$. When every diagram of shape $I$ has a limit (resp., colimit), we say
that $\C$ has all $I$-limits (resp., $I$-colimits). In this case, $\lim_I \colon \C^I\to \C$ (resp., $\colim_I \colon \C^I\to \C$) will denote
the ($I$-)limit (resp., ($I$-)colimit) functor, which is the right (resp., left) adjoint to the constant diagram functor
$\kappa_I\colon \C\to \C^I$. A particular case, very important for us,  comes when $I$ is a directed set, viewed as
a small category in the usual way. The corresponding colimit
functor is the ($I$-)directed colimit functor 
$\varinjlim_I\colon \C^I\to \C$. The $I$-diagrams in $\C$ are usually called \emph{directed
systems}  of shape $I$.

The category $\C$ is said to be
{\em complete} (resp., {\em cocomplete, bicomplete}) when all $I$-limits (resp., $I$-colimits, both) exist in $\C$, for any small category $I$.

\medskip\noindent
\textbf{$2$-Categories of categories.}
We denote by $\cat$ the $2$-category of small categories and by $\cat^{\op}$ the $2$-category obtained by reversing the direction of the functors in $\cat$ (but letting the direction of natural transformations unchanged). Similarly, we  denote by $\Cat$ the \mbox{$2$-``category"} of all categories. This, when taken literally, may originate some set-theoretical problems that, for our constructions, can be safely ignored: see the discussion after \cite[Def.\,1.1]{Moritz}.

Given two natural transformations $\alpha\colon F\Rightarrow G\colon \C\to \cD$ and $\beta\colon G\Rightarrow H\colon \C\to \cD$, we denote by $\beta\circ \alpha\colon F\Rightarrow H$ their vertical composition, that is, $(\beta\circ\alpha)_C:=\beta_C\circ \alpha_C$ for each $C\in \C$. Furthermore, given two natural transformations $\alpha\colon F_1\Rightarrow F_2\colon \C\to \cD$ and $\beta\colon G_1\Rightarrow G_2\colon \cD\to \mathcal{E}$, we denote by $\beta\circledast\alpha\colon G_1\circ F_1\Rightarrow G_2\circ F_2$ their horizontal composition, that is, $(\beta \circledast\alpha)_C:=\beta_{F_2(C)}\circ G_1(\alpha_C)=G_2(\alpha_C)\circ \beta_{F_1(C)}$, for each $C\in \C$. With a slight abuse of notation, we also let $\beta\circledast F_1:=\beta \circledast \id_{F_1}$ and $G_1\circledast \alpha:=\id_{G_1}\circledast\,\alpha$.

\medskip\noindent
\textbf{Categories of adjoint functors and mates.}
Let $\C$ and $\cD$ be two categories. Given a pair of adjoint functors $L:\C\rightleftarrows \cD:R$, we use the following compact notation 
\[
L\adjunct{\eta}{\epsilon}R
\] 
to mean that $L$ is left adjoint to $R$, with unit $\eta\colon \id_\C\Rightarrow R\circ L$ and counit $\epsilon\colon L\circ R\Rightarrow \id_{\cD}$. In particular, the following relations hold
\begin{equation}\label{adj_tria_eq}
\id_{L}=(\epsilon\circledast L)\circ( L\circledast\eta) \quad\text{and}\quad \id_{R}=(R\circledast\epsilon) \circ (\eta \circledast R)
\end{equation}
Furthermore, we let $\LAdj(\C,\cD)$ (resp., $\RAdj(\C,\cD)$) be the (full) subcategory of $\Cat(\C,\cD)$ spanned by the left (resp., right) adjoint functors.
\begin{lemma}[{\cite[Sec.\,2]{kelly1974review} or \cite[Lem.\,1.14]{Moritz}}]\label{lem.adjoints}
Given two categories $\C$ and $\cD$, there is an equivalence of categories 
\[
\Phi\colon \LAdj(\C,\cD)\longrightarrow \RAdj(\cD,\C)^\op,
\] 
which is constructed as follows:
\begin{itemize}
\item for each left adjoint $L$ fix a right adjoint $\Phi(L)$ of $L$ (that is, $L\dashv\Phi(L)$);
\item given two left adjoints $L$ and $L'$, where $L\adjunct{\eta}{\epsilon}\Phi(L)$ and $L'\adjunct{\eta'}{\epsilon'}\Phi(L')$, and a natural transformation $\alpha\colon L\Rightarrow L'$ one defines $\Phi(\alpha)\colon \Phi(L')\Rightarrow \Phi(L)$ as follows:
\[
\Phi(\alpha):=(\Phi L\circledast \epsilon')\circ (\Phi L\circledast\alpha \circledast\Phi L')\circ (\eta\circledast\Phi L').
\]
\end{itemize}
Furthermore, the bijection $\Phi_{L,L'}\colon \LAdj(\C,\cD)(L,L')\to \RAdj(\cD,\C)(\Phi(L'),\Phi(L))$ respects compositions and identities. In particular, a given $\alpha\colon L\Rightarrow L'$ is a natural isomorphism if and only if $\Phi(\alpha)$ is a natural isomorphism.
\end{lemma}
Given a natural transformation between left adjoints $\alpha\colon L\Rightarrow L'$, its image $\Phi(\alpha)$ in $\RAdj(\C,\cD)$ via the equivalence described above is said to be the {\em total mate} of, or the natural transformation {\em conjugated} to, $\alpha$.
Consider now a square in $\Cat$, that commutes up to the natural transformation $\alpha$:
\[
\xymatrix@C=10pt@R=6pt{
\C\ar[dd]_L\ar[rr]^F&&\C'\ar[dd]^{L'}\ar@{}[ddll]|-{\Swarrow\alpha}\\
\\
\cD\ar[rr]_G&&\cD'
}
\]
If there are adjunctions $L\adjunct{\eta}{\epsilon}R\colon \C\to \cD$, and  $L'\adjunct{\eta'}{\epsilon'}R'\colon \C'\to \cD'$, then we can consider the following pasting diagram:
\[
\xymatrix@C=10pt@R=6pt{
\C\ar@{<-}[dd]_R\ar[rr]^F&&\C'\ar@{<-}[dd]^{R'}\ar@{}[ddll]|-{\alpha_*\Searrow}\\
\\
\cD\ar[rr]_G&&\cD'
}\qquad
\xymatrix@C=11pt@R=2pt{
&\cD\ar@/_10pt/@{=}[ddrr]\ar[rr]^R&&\C\ar[dd]_L\ar[rr]^F&&\C'\ar[dd]^{L'}\ar@/_-10pt/@{=}[ddrr]&&\\
:=&&\ \ \ \ \ \ \ \ \Swarrow\text{\scriptsize$\epsilon$}&&\Swarrow\text{\scriptsize$\alpha$}&&\text{\scriptsize$\eta'$}\Swarrow\ \ \ \ \ \ \ \ &\\
&&&\cD\ar[rr]_G&&\cD'\ar[rr]_{R'}&&\C'
}
\]
that is, $\alpha_*:=(R'G\circledast\epsilon)\circ (R'\circledast\alpha\circledast R)\circ (\eta'\circledast FR)$. Dually, given  adjunctions $F_!\adjunct{\eta^F}{\epsilon^F}F\colon \C'\to \C$, and  $G_!\adjunct{\eta^G}{\epsilon^G}G\colon \cD'\to \cD$, we can consider the following pasting:
\[
\xymatrix@C=10pt@R=6pt{
\C\ar@{<-}[dd]_{F_!}\ar[rr]^{L}&&\cD\ar@{<-}[dd]^{G_!}\ar@{}[ddll]|-{\alpha_!\Nwarrow}\\
&&&&:=\\
\C'\ar[rr]_{L'}&&\cD'
}\qquad
\xymatrix@C=11pt@R=2pt{
\C'\ar@/_10pt/@{=}[ddrr]\ar[rr]^{F_!}&&\C\ar[dd]_{F}\ar[rr]^L&&\cD\ar[dd]^{G}\ar@/_-10pt/@{=}[ddrr]&&\\
&\ \ \ \ \ \ \ \ \Nearrow\text{\scriptsize$\eta^F$}&&\Nearrow\text{\scriptsize$\alpha$}&&\text{\scriptsize$\epsilon^G$}\Nearrow\ \ \ \ \ \ \ \ &\\
&&\C'\ar[rr]_{L'}&&\cD'\ar[rr]_{G_!}&&\cD
}
\]
that is, $\alpha_!:=(\epsilon^G\circledast LF_!)\circ (G_!\circledast \alpha \circledast F_!)\circ (G_!L'\circledast\eta^F)$. In fact, there is a close relation between $\alpha_!$ and $\alpha_*$:

\begin{lemma}[{\cite[Sec.\,2]{kelly1974review}}]\label{mates_are_conjugated_lemma}
Let $\C$, $\C'$, $\cD$ and $\cD'$ be categories, consider the adjunctions:
\[
L\adjunct{\eta}{\epsilon}R\colon \C\to \cD, \qquad  L'\adjunct{\eta'}{\epsilon'}R'\colon \C'\to \cD', \qquad F_!\adjunct{\eta^F}{\epsilon^F}F\colon \C'\to \C, \qquad G_!\adjunct{\eta^G}{\epsilon^G}G\colon \cD'\to \cD.
\]
and a natural transformation $\alpha\colon L'\circ F\Rightarrow G\circ L$. Then $\alpha_*\colon  F\circ R\Rightarrow R'\circ G$ is the natural transformation conjugated to $\alpha_!\colon G_!\circ L'\Rightarrow L\circ F_!$. In particular, $\alpha_*$ is a natural isomorphism if, and only if, $\alpha_!$ is a natural isomorphism.
Furthermore, the operations $(-)_*$ and $(-)_!$ are each other's inverse, in the sense that 
\[
(\alpha_*)_!=\alpha=(\alpha_!)_*.
\]
\end{lemma}

\subsection{Additive categories, Abelian categories and torsion pairs}

\medskip\noindent
\textbf{Additive categories.}
Recall that a category is said to be additive if it is an $\Ab$-enriched category with a $0$-object and finite products and coproducts. An important property of additive categories is that, in this context, finite products and coproducts coincide. More precisely, given an additive category $\C$ and two objects $A_1$ and $A_2$, any coproduct $(A, (\epsilon_i\colon A_i\to A)_{i=1,2})$ (resp., any product $(A, (\pi_i\colon A\to A_i)_{i=1,2})$) can be completed to a biproduct, that is, a triple $(A, (\epsilon_i\colon A_i\to A)_{i=1,2}, (\pi_i\colon A\to A_i)_{i=1,2})$ such that the following identities hold:
\[
\pi_1\epsilon_1=\id_{A_1},\quad \pi_2\epsilon_2=\id_{A_2},\quad \epsilon_1\pi_1+\epsilon_2\pi_2=\id_A,\quad \pi_2\epsilon_1=0=\pi_1\epsilon_2.
\]
This allows one to interpret (column-finite) matrices as morphisms between (infinite) coproducts.

\begin{definition}
Let $\C$ be an additive category and consider the following sequence in $\C$:
\[
\xymatrix{
A\ar[r]^\alpha& B\ar[r]^\beta& C.
}
\]
We say that the above sequence is {\em split-exact} if there exist $\alpha'\colon B\to A$ and $\beta'\colon C\to B$ such that the triple $(B,(\alpha,\beta'),(\alpha',\beta))$ is a biproduct of $A$ and $C$.
\end{definition}

\begin{remark} 
Note that a sequence $A\stackrel{\alpha}{\longrightarrow}B\stackrel{\beta}{\longrightarrow}C$ in an additive category $\mathcal{C}$ is split-exact if, and only if, either one of the following equivalent conditions holds:
\begin{enumerate}[\rm (1)]
\item it is a {\em kernel-cokernel sequence} (i.e., $\alpha$ is the kernel of $\beta$, and $\beta$ is the cokernel of $\alpha$) such that $\alpha$ is a section (resp., $\beta$ is a retraction);
\item there are morphisms $\alpha'\colon B\rightarrow A$ and $\beta'\colon C\rightarrow B$ such that $\alpha'\alpha =\id_A$, $\beta\beta'=\id_C$ and $\id_B=\alpha\alpha'+\beta'\beta$ (i.e., the conditions $\beta\alpha=0=\alpha'\beta'$ are redundant). Indeed, notice for example that $\beta\alpha=\beta\alpha\alpha'\alpha=\beta(\id_B-\beta'\beta)\alpha=(\id_C-\beta\beta')\beta\alpha=0$.
\end{enumerate}
\end{remark}

When $\alpha\colon A\to A'$ and $\beta\colon B'\to B$ are isomorphisms, the following are trivial examples of split-exact sequences:
\[
\xymatrix@C=33pt{
A\ar[r]^\alpha& A'\ar[r]^0& 0&0\ar[r]^0& B'\ar[r]^{\beta}& B.
}
\]

\begin{example}\label{countable_splitexact_ex}
Let $\C$ be an additive category. 
\begin{enumerate}[\rm (1)]
\item Consider the following sequence in $\C$:
\[
\xymatrix@C=55pt{
A\ar[r]^-{\scriptsize\alpha:=\left[\begin{matrix}\alpha_1\\ \alpha_2\end{matrix}\right]}& A'\sqcup B'\ar[r]^-{\scriptsize\beta:=\left[\begin{matrix}\beta_1& \beta_2\end{matrix}\right]}&B.
}
\]
This sequence is split-exact if $\beta \circ\alpha=0$ and both $\alpha_1\colon A\to A'$ and $\beta_2\colon B'\to B$ are isomorphisms as, in this case, it is enough to choose $\alpha':=[\begin{smallarray}{cc}\alpha_1^{-1}&0\end{smallarray}]$ and $\beta':=\left[\begin{smallarray}{c}0 \\\beta_2^{-1}\end{smallarray}\right]$. Similarly, the sequence is split-exact if $\beta \circ\alpha=0$ and both $\alpha_2$ and $\beta_1$ are isomorphisms.
\item If $X\in\C$ if an object such that the countable coproduct $X^{(\N)}$ of copies of $X$ exists in $\C$, then the following sequence is split-exact:
\[
\xymatrix@C=130pt{
X^{(\N)}\ar[r]^{\alpha=
\left[\begin{smallarray}{cccc}
1&&&\\
-1&1&&{\text{\Large $0$}}\\
&-1&1&\\
{\text{\Large $0$}}&&\ddots&\ddots
\end{smallarray}\right]
}&X^{(\N)}\ar[r]^{
\beta=\left[\begin{smallarray}{cccc}
1&1&1&\cdots
\end{smallarray}\right]}&X,
}
\] 
where $\alpha':=\left[\begin{smallarray}{ccccc}
0&-1&-1&-1&\cdots\\
&0&-1&-1&\cdots\\
&&0&-1&\cdots\\
{\text{\Large $0$}}&&&\ddots&\ddots
\end{smallarray}\right]
$ and $\beta':=\left[\begin{smallarray}{c}1 \\0 \\0 \\\vdots\end{smallarray}\right]$.
\end{enumerate}
\end{example}

\medskip\noindent
\textbf{Subcategories of additive categories.}
Given an additive category $\C$ and a (always full) subcategory $\S \subseteq  \C$, we shall denote by $\add_\C(\S)$ (resp., $\Add_\C(\S)$), or simply $\add(\S)$ (resp., $\Add(\S)$) if no confusion is possible, the subcategory of $\C$ spanned by the direct summands of finite (resp., small) coproducts of objects in $\S$.

We use the notation $\varinjlim \S=\S$ to mean that $\S$ is closed under taking directed colimits, that is, given $F\colon I\to \C$ for $I$ a directed set, such that $F(i)\in \S$ for all $i\in I$, whenever the colimit $\varinjlim_IF$ exists in the ambient category $\C$, it is an object of $\S$.

A family of objects $\S$ is said to be a {\em set of generators} when the functor
\[
\xymatrix{
\coprod_{S\in \S}\C(S,-) \colon \C\longrightarrow \Ab
}
\] 
is faithful. An object $G$ is a {\em generator} of $\C$ when $\{G\}$ is a set of
generators.

\medskip\noindent
\textbf{(Ab.$5$) and Grothendieck (Abelian) categories.}
Let $\C$ be an Abelian category. 
Recall from  \cite{Grothendieck} that  $\C$ is called (Ab.$5$)  when it is (Ab.$3$) (=cocomplete)  and the directed colimit functor $\varinjlim_I\colon \C^I\to \C$
 is exact, for any directed set $I$.
An (Ab.$5$) Abelian category $\G$ having a set of generators (equivalently, a generator), is said to be a {\em Grothendieck Abelian category}. Such a category always has enough injectives, and any of its objects has an injective envelope
(see \cite{Grothendieck}). Moreover, it is always a complete (and cocomplete) category (see \cite[Coro.\,X.4.4]{S75}).

\medskip\noindent
\textbf{Finitely presented objects.}
When $\G$ is a cocomplete additive category, an object $X$ of $\G$ is called {\em finitely presented} when $\G(X, -) \colon \G \to \Ab$ preserves directed colimits, that is, for any directed set $I$ and any diagram $(Y_i)_{i\in I}$ in $\G^{I}$, the following canonical map is an isomorphism
\[
\colim_I\, \G(X,Y_i)\longrightarrow \G(X,\colim_I\,Y_i)
\]
where the first colimit is taken in $\Ab$ and the second in $\G$. When $\G$ is a Grothendieck Abelian category with a set of finitely presented generators which, in this setting, is equivalent to say that each object of $\G$ is a directed colimit of finitely
presented objects, we say that $\G$ is {\em locally finitely presented}.

\medskip\noindent
\textbf{Torsion pairs.}
A {\em torsion pair} in an Abelian category $\C$ is a pair $\tau = (\mathcal T , \mathcal F)$ of subcategories satisfying the following
two conditions:
\begin{enumerate}[(TP.1)]
\item $\C(T, F) = 0$, for all $T\in\mathcal T$ and $F\in\mathcal F$;
\item for any object $X$ of $\C$ there is a short exact sequence $\varepsilon_X\colon 0 \to T_X\to X\to F_X \to 0$, where $T_X\in \mathcal T$ and
$F_X \in\mathcal F$.
\end{enumerate}
In such case,
the short exact sequence $\varepsilon_X$ from (TP.2) is
uniquely determined, up to a unique isomorphism, and the assignment
$X\mapsto T_X$ (resp. $X\mapsto F_X$) underlies a functor $\C\to \mathcal T$ (resp.,  $\C\to \mathcal F$) which is right
(resp., left) adjoint to the inclusion functor $\mathcal T\to \C$ (resp., $\mathcal F\to \C$). We say that $\tau$ is of {\em finite type} provided $\varinjlim\mathcal{F}=\mathcal{F}$.

\subsection{Triangulated categories and $t$-structures}\label{tria_and_t_subs}

\medskip\noindent
\textbf{Triangulated categories.}
We refer  to \cite{Neeman} for the precise definition of triangulated category. In particular, given a triangulated category $\mathcal D$, we  always denote by $\Sigma\colon\mathcal D\to \mathcal D$ the {\em suspension functor}, and we denote ({\em distinguished}) {\em triangles} either by $X \to Y\to  Z\overset{+}\to$ or  $X\to Y \to Z\to \Sigma X$.
Unlike the terminology used in the abstract setting of additive categories, in the
context of triangulated categories a weaker version of the notion of ``set of generators'' is commonly used. Namely,
a set $\S \subseteq \mathcal D$ is called a {\em set of generators} of $\mathcal D$ if an object $X$ of $\mathcal D$ is zero whenever $\mathcal D(\Sigma^kS, X) = 0$,
for all $S \in\S$ and $k \in\Z$. In case $\mathcal D$ has coproducts, we  say that an object $X$ is {\em compact} when the functor $\mathcal D(X,-) \colon \mathcal D\to \Ab$ preserves coproducts. We  say that $\mathcal D$ is {\em compactly generated} when it has a set of compact generators. 

Given a set $\mathcal X$ of objects in $\mathcal D$ and a subset $I\subseteq \Z$, we let
\begin{align*}
\mathcal{X}^{\perp_{I}}&:=\{Y\in\mathcal D:\mathcal D(X,\Sigma^iY)=0\text{, for all }X\in\mathcal{X}\text{ and }i\in I\}\\
{}^{\perp_{I}}\mathcal{X}&:=\{Z\in\mathcal D:\mathcal D(Z,\Sigma^iX)=0\text{, for all }X\in\mathcal{X}\text{ and }i\in I\}.
\end{align*}
If $I=\{i\}$ for some $i\in \Z$, then we let $\mathcal{X}^{\perp_{i}}:=\mathcal{X}^{\perp_{I}}$ and ${}^{\perp_{i}}\mathcal{X}:={}^{\perp_{I}}\mathcal{X}$. If $i=0$, we even let $\mathcal{X}^{\perp_{}}:=\mathcal{X}^{\perp_{0}}$ and ${}^{\perp_{}}\mathcal{X}:={}^{\perp_{0}}\mathcal{X}$. In particular, $\S$ is a set of generators if, and only if $\S^{\perp_\Z}=0$.

\medskip\noindent
\textbf{Cohomological functors.}
Given a triangulated category $\mathcal D$ and an Abelian category $\C$, an additive functor $H^0 \colon \mathcal D\to \C$ is said to be a {\em cohomological functor} if, for any given triangle $X\to Y\to Z\to \Sigma X$, the sequence $H^0(X)\to H^0(Y)\to H^0(Z)$ is exact in $\C$. In particular, one obtains a long exact sequence as follows:
\[
\ldots \to H^{n-1}(Z)\to H^{n}(X)\to H^{n}(Y)\to H^{n}(Z)\to H^{n+1}(X)\to \ldots
\]
where $H^n := H^0 \circ \Sigma^{n}$, for any $n \in \Z$.

\medskip\noindent
\textbf{$t$-Structures.}
A {\em $t$-structure} in $\mathcal D$ is a pair $\t=(\U, \W)$ of subcategories satisfying the following axioms:
\begin{enumerate}[($t$-S.1)]
\item $\mathcal D(U, \Sigma^{-1}W) = 0$, for all $U \in \U$ and $W \in \W$;
\item $\Sigma \U \subseteq \U$;
\item for each $X \in \mathcal D$, there are $U_X \in \U$, $V_X \in \Sigma^{-1}\W$ and a triangle 
\[
U_X \to X \to V_X\to \Sigma U_X\qquad\text{in $\cD$.}
\] 
\end{enumerate}
One can see that, in such case, both classes are closed under taking direct summands in $\mathcal D$, that $\W = \Sigma(\U^{\perp})$ and $\U = {}^{\perp}(\Sigma^{-1}\W) = {}^{\perp}(\U^{\perp})$. For this reason, we  write a $t$-structure as $\t=(\U, \Sigma(\U^{\perp}))$ or $\t=(\U, \Sigma\V)$, meaning that $\V:=\U^{\perp}$. We will call $\U$ and $\U^{\perp}$ the {\em aisle} and the {\em co-aisle} of the $t$-structure, respectively.
The triangle of the above axiom ($t$-S.3) is
uniquely determined by $X$, up to a unique isomorphism, and defines
functors $\tau_{\U}\colon \mathcal D\to \U$ and $\tau^{\U^{\perp}}\colon \mathcal D\to \U^{\perp}$ which are right and left adjoints to the respective inclusions. We call them the {\em left} and {\em right truncation functors} with respect to the given $t$-structure $\t$. Furthermore, the above triangle will be referred to as the {\em truncation triangle} of $X$ with respect to $\t$. 

\smallskip
We say that a  $t$-structure $(\U,\Sigma\V)$ is generated by a set $\S$, when $\Sigma\V=\S^{\perp_{<0}}$ (equivalently, $\V =\S^{\perp_{\leq 0}}$). 
 When $\mathcal D$ has coproducts, we say that the $t$-structure is {\em compactly generated} when  it is generated by a set 
 $\S$ consisting of compact objects in $\mathcal{D}$; in this case, we say that $\S$ is a {\em set of compact generators}\label{compactly_gen_t_st} of the aisle $\U$ or of the $t$-structure.

\medskip\noindent
\textbf{Hearts.}
Given a $t$-structure $\t=(\U, \Sigma\V)$ in a triangulated category $\cD$, the subcategory $\mathcal H := \U \cap \Sigma \V\subseteq \cD$ is called the {\em heart} of the $t$-structure and it is an Abelian category, where the short exact sequences ``are'' the triangles of $\mathcal D$ with the three terms in $\H$. Moreover, with an obvious abuse of notation, the assignments $X\mapsto \tau_U \circ \tau^{\Sigma(\U^{\perp})}(X)$ and $X\mapsto\tau^{\Sigma(\U^{\perp})}\circ\tau_U (X)$ define two naturally isomorphic cohomological functors $H^0_\t\colon \mathcal D\to \H$ (see \cite{BBD}).

\begin{example} \label{expl.tilted t-str}
Let $\cal D$ be a triangulated category together with a $t$-structure $\t=(\U,\Sigma\V)$ and heart $\H:=\U\cap \Sigma\V$. Given a torsion pair $\tau=(\mathcal T,\mathcal F)$ in $\H$ we can define a new $t$-structure $\t_\tau=(\U_\tau,\Sigma\V_\tau)$ on $\cal D$, called the {\em Happel-Reiten-Smal{\o} tilt of $\t$ with respect to $\tau$} (see \cite{HRS}), where 
\[
\U_\tau:=\Sigma \U*\mathcal T,\quad \text{and}\qquad \V_\tau:=\mathcal F*\V,
\]
with the convention that, given two classes $\mathcal X,\ \mathcal Y\subseteq \cal D$, $Z\in \mathcal X*\mathcal Y$ if and only if there exists a triangle $X\to Z\to Y\to\Sigma X$ in $\mathcal D$, with $X\in\mathcal X$ and $Y\in\mathcal Y$.
\end{example}

\subsection{Categories with weak equivalences}

\noindent
\textbf{Categories with weak equivalences.} Let $\C$ be a category and let $\W$ be a collection of morphisms containing all the isomorphisms in $\C$. The pair $(\C,\W)$ is said to be a {\em category with weak equivalences} (or a {\em relative category}) if, given two composable morphisms $\phi$ and $\psi$, whenever two elements of $\{\phi,\psi,\psi\phi\}$ belong in $\W$ so does the third. The elements of $\W$ are called {\em weak equivalences}.

Given a category with weak equivalences $(\C,\W)$, its {\em universal localization} is a pair $(\C[\W^{-1}],F)$ consisting of a category $\C[\W^{-1}]$ and a  functor $F\colon\C\to \C[\W^{-1}]$ such that $F(\phi)$ is an isomorphism for all $\phi\in\W$, and it is universal with respect to these properties, that is, if $G\colon\C\to\cD$ is a functor such that $G(\phi)$ is an isomorphism for all $\phi\in\W$, then there exists a unique functor $G'\colon \C[\W^{-1}]\to \cD$ such that $G'F=G$ (see \cite{GZ}).

\medskip\noindent
\textbf{Derived functors.}
Let $(\C,\W)$, $(\C',\W')$ be categories with weak equivalences and suppose that their universal localizations exist. A functor $\mathbf{L}G\colon \C[\W^{-1}]\to \C'[\W'^{-1}]$ together with a natural transformation  $\alpha\colon \mathbf{L}G \circ F\Rightarrow F'\circ G$ is called the {\em total left derived functor} of \mbox{$G\colon \C\to \C'$} if the pair $(\mathbf{L}G,\alpha)$ is terminal between all pairs $(H,\beta)$ such that \mbox{$H\colon \C[\W^{-1}]\to \C'[\W'^{-1}]$} and $\beta\colon H \circ F\Rightarrow F'\circ G$. That is, given any $(H,\beta)$, there is a unique natural transformation $\gamma\colon H \Rightarrow \mathbf{L}G$ such that $\beta = \alpha \circ \gamma F$.
The notion of {\em total right derived functor} $\mathbf{R}G$ of $G$ is defined dually.

\medskip\noindent
\textbf{Model categories.} A \emph{model structure} on a bicomplete category $\mathcal{C}$ is a triple $(\W,\cal B,\cal F)$ of classes of morphisms, closed under retracts, called respectively the \emph{weak equivalences}, \emph{cofibrations}, and \emph{fibrations}, such that $(\C,\W)$ is a category with weak equivalence and satisfying a series of axioms, for which we refer to \cite{Hovey_libro, Dwyer}. 
A model category is then a bicomplete category $\C$ equipped with a model structure; i.e.\ a quadruple $(\C; \W, \cal B, \cal F)$.

The mere existence of a model structure for a category with weak equivalences allows one to give an explicit construction of the universal localization $\C[\W^{-1}]$, which is traditionally called the \emph{homotopy category} of $\C$ in this context, and denoted by $\ho(\C)$.
Even better, model structures allow to construct and compute derived functors as well. To this end, an adjunction $(F,G)\colon \C \rightleftarrows \C'$ between two model categories $\C$, $\C'$ with model structures $(\W,\cal B,\cal F)$ and $(\W',\cal B',\cal F')$, respectively, is called a \emph{Quillen adjunction} if it satisfies one of the following equivalent conditions (see~\cite[Lemma 1.3.4]{Hovey_libro}):
\begin{itemize}
\item the left adjoint $F\colon \C \to \C'$ preserves cofibrations and trivial cofibrations;
\item the right adjoint $G\colon \C' \to \C$ preserves fibrations and trivial fibrations.
\end{itemize}

Given a Quillen adjunction, the total derived functors $(\mathbf{L}F,\mathbf{R}G)$ exist and form an adjunction between $\ho(\C)$ and $\ho(\C')$. Moreover, $\mathbf{L}F(X)$ for an object $X$ of $\C$ can be computed by applying $F$ to a cofibrant replacement of $X$ and dually for $\mathbf{R}G$.

In the context of algebraic examples, model structures on Abelian (even Grothendieck) categories play a prominent role. In particular, many of our examples will arise from the so-called Abelian model structures (see \cite{hovey, becker2014models, gillespie2011model}).
The following example allows one to encode the machinery of classical homological algebra in the framework of model categories. 

\begin{example}\label{inj_mod_struc}
Given a Grothendieck Abelian category $\G$, we will denote by $\Ch(\G)$, $\mathcal K(\G)$ and $\Der(\G)$ the category of (unbounded) cochain complexes
of objects of $\G$, the homotopy category of $\G$ and the derived category of $\G$, respectively (see \cite{verdier1977categories, keller1998introduction}). 
Recall that $\Ch(\G)$ is a bicomplete category. With the class $\W$ of quasi-isomorphisms in $\Ch(\G)$, the pair $(\Ch(\G),\W)$ is a category with weak equivalences. Furthermore, taking $\mathcal F$ be the class of all the epimorphisms with dg-injective kernels and let $\mathcal B$ be the class of monomorphisms, then $\Ch(\G)$ with $(\W,\cal B,\cal F)$ is a model category (see for example \cite[Thm.\,2.3.13]{Hovey_libro}, \cite{hovey} or \cite{Gil07} for details and a proof). The homotopy category in this case is $\Der (\G)$. 
\end{example}

\section{Preliminaries on derivators}\label{prel_on_ders_sec}

In this section we fix some notational conventions and we recall the basic definitions about prederivators, derivators, morphisms and $2$-morphisms between them. Furthermore, we collect some useful facts about pointed, additive, strong and stable derivators. All these results are probably known to experts but, in several cases, we have not been able to find suitable references in the existing literature. In those cases we have included a short proof.

\subsection{(Pre)Derivators}\label{subs_pred}

\noindent
\textbf{Prederivators.} 
A {\em pre-derivator} is a strict $2$-functor
\[
\D\colon \cat^{\op}\to \Cat.
\]
All along this paper, we will use the following notational conventions:
\begin{itemize}
\item the letter $\D$ will denote a (pre)derivator;
\item  for any natural transformation $\alpha\colon u\Rightarrow v\colon J\to I$ in $\cat$, we will use the notation $\alpha^*:=\D(\alpha)\colon u^*\Rightarrow v^*\colon \D(I)\to \D(J)$. Furthermore, we denote respectively by $u_!$ and $u_*$ the left and the right adjoint to $u^*$ (whenever they exist), and call them respectively the {\em left} and {\em right homotopy Kan extension of $u$};
\item the letters $U$, $V$, $W$, $X$, $Y$, $Z$,  will be used either for objects in the base $\D(\bbone)$ or for (incoherent) diagrams in $\D(\bbone)$, that is, functors $I\to \D(\bbone)$, for some small category~$I$;
\item the letters $\mathscr U$, $\mathscr V$, $\mathscr W$, $\mathscr X$, $\mathscr Y$, $\mathscr Z$,  will be used for objects in some image $\D(I)$ of the derivator, for $I$ a category (possibly) different from $\bbone$. Such objects will be usually referred to as {\em coherent diagrams of shape $I$};
\item given $I\in \cat$, consider the unique functor $\pt_I\colon I\to \bbone$. We usually denote by $\hocolim_I:=(\pt_I)_!\colon \D(I)\to \D(\bbone)$ and $\holim_I:=(\pt_I)_*\colon \D(I)\to \D(\bbone)$ respectively the left and right homotopy Kan extensions of $\pt_I$; these functors are called respectively {\em homotopy colimit} and {\em homotopy limit}. 
\end{itemize}

For a given object $i\in I$, we also denote by $i$ the inclusion $i\colon\bbone\to I$ such that $0\mapsto i$. We obtain an evaluation functor $i^*\colon \D(I)\to \D(\bbone)$. For an object $\mathscr X\in \D(I)$, we let $\mathscr X_i:=i^*\mathscr X$. Similarly, for a morphism $\alpha\colon i\to j$ in $I$, one can interpret $\alpha$ as a natural transformation from $i\colon \bbone \to I$ to $j\colon \bbone \to I$. In this way, to any morphism $\alpha$ in $I$, we can associate $\alpha^*\colon i^*\Rightarrow j^*$. For an object $\mathscr X\in \D(I)$, we let $\mathscr X_\alpha:=\alpha^*_{\mathscr X}\colon \mathscr X_i\to \mathscr X_j$.
For any $I$ in $\cat$, we denote by 
\[
\dia_I\colon \D(I)\longrightarrow \D(\bbone)^{I}
\]
the {\em diagram functor}, such that, given $\mathscr X\in \D(I)$, $\dia_I(\mathscr X)\colon I\to \D(\bbone)$ is defined by $\dia_I(\mathscr X)(i\overset{\alpha}{\to} j)=(\mathscr X_i\overset{\mathscr X_\alpha}{\to} \mathscr X_j)$. We will refer to $\dia_I(\mathscr X)$ as the {\em underlying} (incoherent) {\em diagram} of the coherent diagram $\mathscr X$.

\begin{example}\label{description_constant_diagram}
Let $\D\colon \cat^{\op}\to \Cat$ be a prederivator. Given $I\in \cat$, consider the unique functor $\pt_I\colon I\to \bbone$, let $X\in \D(\bbone)$ and consider $\mathscr X:=\pt_I^*X\in \D(I)$. Then the underlying diagram $\dia_I(\mathscr X)\in \D(\bbone)^I$ is constant, that is, $\mathscr X_i= X$ for all $i\in I$ and, for all $\alpha\colon i\to j$ in $I$, the map $\mathscr X_\alpha\colon \mathscr X_i\to \mathscr X_j$ is the identity of $X$. The objects isomorphic to one of the form $\pt_I^*X\in \D(I)$ for some $X\in \D(\bbone)$ are called  {\em constant (coherent) diagrams}. 
\end{example}


\medskip\noindent
\textbf{The $2$-category of prederivators.} 
Prederivators can be organized into a $2$-category, as is explained in \cite[Sec.\,2]{Moritz}. Given prederivators $\D$ and $\D'$, a morphism of prederivators $F\colon\D\to\D'$ consists of functors $F_I\colon\D(I)\to\D'(I)$, one for each $I\in\cat$, and natural equivalences $\gamma^F_u\colon u^*\circ F_I\cong F_J\circ u^*$, one for each functor $u\colon J\to I$, subject to coherence relations (as in \cite{Moritz}, we will be sloppy and use the same symbol for $\gamma^F_u$ and its inverse). Given morphisms of prederivators $F,\, G\colon\D\to\D'$, a $2$-morphism (also called a natural transformation) $\alpha\colon F\Rightarrow G$ consists of a family of natural transformations $\alpha_I\colon F_I\Rightarrow G_I$, one for each $I\in\cat$, which are in a standard way compatible with the natural equivalences $\gamma^F_u$ and $\gamma^G_u$ as above for all functors $u$ in $\cat$. More details can be found in our references.

\begin{example}\label{expl.Yoneda_prederivator}
Any ordinary category $\C$ gives rise to a prederivator $\mathbb{Y}_\C\colon \cat^{\op}\to \Cat$ by the rule $\mathbb{Y}_\C(I)=\C^I$. Such prederivators are called \emph{represented}, \cite[Ex.\,1.2]{Moritz}, or \emph{Yoneda} prederivators. The construction shows that $\Cat$ embeds into the $2$-category of prederivators (up to usual set-theoretical issues, see \cite[Sec.\,2]{Moritz} for details).
\end{example}

\medskip\noindent
\textbf{Derivators.} 
A pre-derivator $\D$ is a {\em derivator} if it satisfies the following four axioms:
\begin{enumerate}[(Der.$1$)]
\item if $\coprod_{i\in I} J_i$ is a disjoint union of small categories, then the canonical functor $\D(\coprod_I J_i) \to \prod_I\D(J_i)$ is an equivalence of categories;
\item for each $I\in\cat$ and  $f\colon \mathscr X \to \mathscr Y$ a morphism in $\D(I)$, $f$ is an isomorphism if, and only if, $i^*(f)\colon i^*(\mathscr{X}) \to i^*(\mathscr{Y})$ is an isomorphism in $\D(\bbone)$ for each $i\in I$;
\item for each functor $u\colon I \to J$ in $\cat$, the functor $u^*$ has both a left adjoint $u_!$ and a right adjoint $u_*$ (i.e.\  homotopy Kan extensions are required to exist).
\end{enumerate}

Before stating the last axiom (Der.$4$), let us introduce the following notation. Suppose we have a natural transformation: $\alpha\colon v\circ u'\Rightarrow u\circ w$ in $\cat$:
\begin{equation}\label{gen_square_in_small_cats}
\xymatrix@C=10pt@R=6pt{
I\ar@{<-}[dd]_-u\ar@{<-}[rr]^-v&&I'\ar@{<-}[dd]^-{u'}\ar@{}[ddll]|-{\Swarrow\alpha}\\
\\
J\ar@{<-}[rr]_-w&&J',
}
\end{equation}
and let $\D\colon \cat^{\op}\to \Cat$ be a prederivator that satisfies (Der.3). Applying $\D$ to the above square, we get a natural transformation $\alpha^*\colon (u')^*\circ v^*\Rightarrow w^*\circ u^*$
\[
\xymatrix@C=10pt@R=6pt{
\D(I)\ar[dd]_-{u^*}\ar[rr]^-{v^*}&&\D(I')\ar[dd]^-{(u')^*}\ar@{}[ddll]|-{\Swarrow\alpha^*}\\
\\
\D(J)\ar[rr]_-{w^*}&&\D(J').
}
\]
In this setting, we define $\alpha_!:=(\alpha^*)_!\colon w_!\circ(u')^*\Rightarrow u^*\circ v_!$ and $\alpha_*:=(\alpha^*)_*\colon v^*\circ u_*\Rightarrow u'_*\circ w^*$, where $(-)_!$ and $(-)_*$ are defined right before Lemma \ref{mates_are_conjugated_lemma}. Hence, we deduce by that lemma that $\alpha_!$ is an isomorphism if and only if $\alpha_*$ is, and, furthermore, 
\begin{equation}\label{mates_and_derivators_eq}
(\alpha_!)_*=\alpha^*=(\alpha_*)_!.
\end{equation}
In order to properly state (Der.$4$), we need to start with a special case of the square \eqref{gen_square_in_small_cats}. Indeed, let $u\colon J\to I$ be a functor in $\cat$ and let $i\in I$. We define the {\em slice category} $u/i$ (resp., $i/u$) whose objects are pairs $(j,a\colon u(j)\to i)$ (resp., $(j,a\colon i\to u(j))$) where $j\in J$ and $a$ is a morphism in $I$. Furthermore, a morphism $f\colon (j,a)\to (j',a')$ in $u/i$ (resp., in $i/u$) is a morphism $f\colon j\to j'$ in $J$ such that $a=a'\circ u(f)$ (resp., $u(f)\circ a=a'$). Then, ``forgetting the second component of objects" gives well-defined functors
\[
\pr_i\colon u/i\longrightarrow J \qquad\text{and}\qquad \pr^i\colon i/u\longrightarrow J.
\]
We then get the following two squares in $\cat$ that commute up to a natural transformation:
\[
\xymatrix@C=18pt@R=8pt{
u/i\ar[dd]_-{\pt_{u/i}}\ar[rr]^-{\pr_i}&&J\ar[dd]^-{u}\ar@{}[ddll]|-{\Swarrow\gamma}\\
\\
\bbone\ar[rr]_-{i}&&I
}
\qquad\qquad\text{and}\qquad\qquad
\xymatrix@C=18pt@R=8pt{
i/u\ar[dd]_-{\pt_{u/i}}\ar[rr]^-{\pr^i}&&J\ar[dd]^-{u}\ar@{}[ddll]|-{\Nearrow\delta}\\
\\
\bbone\ar[rr]_-{i}&&I,
}
\]
where $\gamma_{(j,a)}:=a\colon u(j)\to i$ for each $(j,a)\in u/i$ and  $\delta_{(j,b)}:=b\colon i\to u(j)$ for each $(j,b)\in i/u$. We finally have all the ingredients to properly state the axiom (Der.$4$): 
\begin{enumerate}[(Der.$4$)]
\item the homotopy Kan extensions can be computed pointwise, that is, given $u\colon J\to I$ in $\cat$ and $i\in I$, the following natural transformations are invertible
\[
\xymatrix@C=18pt@R=8pt{
\D(u/i)\ar[dd]_-{(\pt_{u/i})_!}\ar@{<-}[rr]^-{\pr^*_i}&&\D(J)\ar[dd]^-{u_!}\ar@{}[ddll]|-{\Searrow\gamma_!}\\
\\
\D(\bbone)\ar@{<-}[rr]_-{i^*}&&\D(I)
}
\qquad\qquad\textrm{and}\qquad\qquad
\xymatrix@C=18pt@R=8pt{
\D(i/u)\ar[dd]_-{(\pt_{i/u})_*}\ar@{<-}[rr]^-{(\pr^i)^*}&&\D(J)\ar[dd]^-{u_*}\ar@{}[ddll]|-{\Nwarrow\delta_*}\\
\\
\D(\bbone)\ar@{<-}[rr]_-{i^*}&&\D(I).
}
\]
\end{enumerate}

Finally, let us note that $1$- and $2$-cells in the $2$-category of derivators are defined exactly as for prederivators. In other words, derivators form a full sub-$2$-category of the $2$-category of prederivators.

\begin{remark}\label{mates_rem_der4}
Let $\D$ be a prederivator that satisfies (Der.3), $u\colon J\to I$ in $\cat$ and $i\in I$. Consider the natural transformation $\gamma\colon u\circ\pr_i\Rightarrow i\circ\pt_{u/i}$ described in the above discussion. Now remember that, by \eqref{mates_and_derivators_eq}, we have that $(\gamma_!)_*=\gamma^*$. Expanding this equality we get:
\begin{align*}
\gamma^*=(\gamma_!)_*&=(\pt_{u/i}^*i^*\circledast \epsilon_u)\circ(\pt_{u/i}^*\circledast \gamma_!\circledast u^*)\circ(\eta_{u/i}\circledast \pr_i^*u^*)
\end{align*}
where $\epsilon_u\colon u_!u^*\Rightarrow \id_{\D(I)}$ and $\eta_{u/i}\colon \id_{\D(u/i)}\Rightarrow \pt_{u/i}^*(\pt_{u/i})_!$ are the counit and unit, respectively, of the corresponding adjunctions. A similar observation can be done regarding the natural transformation $\delta$.
\end{remark}

\begin{example}\label{expl.Yoneda_derivator}
The Yoneda prederivator (Example~\ref{expl.Yoneda_prederivator}) is a derivator if and only if $\C$ is a complete and cocomplete category.
\end{example}

\begin{example}\label{expl.shift_of_a_prederivator}
Given a derivator $\D$ and a small category $A$, we can define a prederivator $\D^A\colon\cat^\op\to\Cat$ by $\D^A(I):=\D(A\times I)$ (and in the obvious way on functors and natural transformations). Then $\D^A$ is in fact a derivator by~\cite[Theorem 1.25]{Moritz} and it is called the \emph{$A$-shift} of $\D$. Moreover, given a functor $u\colon I\to J$ in $\cat$, one obtains a morphism of derivators $u^*\colon\D^J\to\D^I$ (see~\cite[Example 2.1(2)]{Moritz}).
\end{example}

\begin{example}\label{coh_incoh_rest_ext_example}
Let $\D\colon \cat^{\op}\to \Cat$ be a derivator, $I\in \cat$  and $j\in I$. Then, by\mbox{\cite[Lem.\,11.1]{KN}}, the following triangle commutes up to a canonical isomorphism:
\begin{equation}\label{commutative_triangle_incoherent_extension}
\xymatrix@R=5pt{
\D(I)\ar[rrrr]^{\dia_I}&&&&\D(\bbone)^I\\
&{\Large \Nwarrow}\cong&\\
\D(\bbone)\ar@/_1.0pc/[uurrrr]_{-\otimes j}\ar[uu]^{j_!}
}
\end{equation}
where $-\otimes j$ is the left adjoint to the obvious ``evaluation at $j$" functor $(-)_{\restriction{j}}\colon \D(\bbone)^I\to \D(\bbone)$. Furthermore, given $X\in \D(\bbone)$, we have that $(X\otimes j)(k)\cong \coprod_{I(j,k)}X$, for all $k\in I$.
\end{example}

The following computations will be useful later on:

\begin{lemma}\label{easy_relations_for_u_and_k}
Let $\D\colon \cat^{\op}\to \Cat$ be a derivator, $I$ a small category, $J\subseteq I$ a (full) subcategory and $u\colon J\to I$ the inclusion functor. Then, the following assertions hold true:
\begin{enumerate}[\rm (1)]
\item letting $u_!\adjunct{\eta_u}{\epsilon_u}u^*\colon \D(J)\rightleftarrows \D(I)$ be the induced adjunction, the following natural transformation is invertible: $u^*\circledast\epsilon_u=(\eta_u\circledast u^*)^{-1}\colon u^*u_!u^*\tilde\Longrightarrow u^*$;
\item if $k\in J\subseteq I$, so that $k=u\circ k\colon \bbone\to I$, the following natural transformation is invertible: $k^*u^*\circledast\epsilon_u=k^*\circledast\epsilon_u\colon k^*u_!u^*\tilde\Longrightarrow k^*$.
\end{enumerate}
\end{lemma}
\begin{proof}
(1). By \cite[Prop.\,1.20]{Moritz}, as $u$ is fully faithful, also $u_!$ is fully faithful and, equivalently, $\eta_u$ is a natural isomorphism.  Then, by the triangle unit-counit identities \eqref{adj_tria_eq} we obtain that $u^*\circledast\epsilon_u=(\eta_u\circledast u^*)^{-1}\colon u^*u_!u^*\tilde\Longrightarrow u^*$ is a natural isomorphism.

\smallskip\noindent
(2) follows from (1) and the fact that $\D$ is a strict $2$-functor.
\end{proof}

\begin{remark} \label{rem.cosieve}
An important special case of a full subcategory is a \emph{cosieve}:
it is a full subcategory $J\subseteq I$ such that whenever $j\in J$ and $j\to k$ is a morphism in $I$, then $k\in J$. In that case, we also know what $k^*u_!\mathscr X$ looks like for $k\in I\setminus J$. Namely, $k^*u_!\mathscr X$ is, by~\cite[Prop.\ 1.23]{Moritz}, the initial object of $\D(\bbone)$ for every $\mathscr X$. Thus, in the case of pointed derivators (see Subsection~\ref{subsec.pointed-derivator} below), $u_!$ is called the \emph{left extension by zero functor}.
\end{remark}


\medskip\noindent
\textbf{Preservation of homotopy Kan extensions.} 
Given a morphism of derivators 
\[
F=(F_I,\gamma^F_u)\colon\D\longrightarrow\D'
\]
 and $u\colon J\to I$, we have the following natural transformations:
 \[
 (\gamma_u^F)_!\colon u_!F_J\Longrightarrow F_Iu_!\qquad\text{and}\qquad  (\gamma_u^F)_*\colon F_Iu_*\Longrightarrow u_*F_J.
 \]
 If either of these is a natural isomorphism, we say that $F$ preserves homotopy left or right Kan extensions along $u$, respectively. Recall from \cite[Def.\,3.15]{Moritz} that $F$ is \emph{left exact} if it preserves homotopy pullbacks and final objects and \emph{right exact} if it preserves homotopy pushouts and initial objects. 

\smallskip
The following statement is well-known in ordinary category theory. Here we extend it to arbitrary morphisms of derivators:

\begin{lemma} \label{lem.referee2.3}
A morphism of derivators which preserves directed homotopy colimits and finite
coproducts also preserves arbitrary coproducts.
\end{lemma}
\begin{proof}
Let $G\colon\D\to\D'$ be a morphism of derivators that preserves directed homotopy colimits and finite
coproducts, and let $I$ be a set, viewed as a discrete category. We consider the directed set $P:=\mathcal P^{<\omega}(I)$ of the finite parts of $I$, ordered by inclusion, and consider it as a posetal category. There is an obvious inclusion functor $u\colon  I\to P$, such that $i\mapsto\{i\}$. Furthermore,
\[
\left(\xymatrix@C=40pt{ I\ar[r]^-u&P\ar[r]^-{\pt_{P}} &\bbone}\right) = \left(\xymatrix{I\ar[r]^-{\pt_I}& \bbone}\right).
\]
By the compatibility of mates with pasting (see \cite[Lem.\,1.14]{Moritz}) to prove that $G$ commutes with coproducts indexed by $I$, that is, with homotopy left Kan extensions along $\pt_I$ (use (Der.1) for this equivalence), it is enough to verify that $G$ commutes, separately, with homotopy left Kan extensions along $u$ and along $\pt_{P}$, the latter being trivial as $P$ is directed. It remains to prove that $(\gamma_u^G)_!\colon u_!\circ G_I\Longrightarrow G_{P}\circ u_!$ is invertible. By (Der.2), $(\gamma_u^G)_!$ is an isomorphism if, and only if, $F^*((\gamma_u^G)_!)$ is an isomorphism for each $F\in P$. Fix then an $F\in P$ and consider the following natural isomorphisms
\[
F^*\circ  u_!\circ G_I\cong \hocolim_{u/F}\circ \pr_F^*\circ G_I\cong \hocolim_{u/F}\circ G_{u/F}\circ \pr_F^*,
\]
where the first isomorphism comes from the axiom (Der.4), while the second is induced by the natural isomorphism $\gamma^G_{\pr_F}$. Similarly we have the following two isomorphisms, the first one induced by $\gamma^G_{F}$ and the second one by (Der.4): 
\[
F^*\circ G_{P}\circ u_!=G_{\bbone}\circ F^*\circ u_!\cong G_{\bbone}\circ \hocolim_{u/F}\circ \pr_F^*.
\]
Hence, to prove that $F^*((\gamma_u^G)_!)$ is an isomorphism, it is enough to prove that 
\[
(\gamma_{\pt_{u/F}}^G)_!\colon \hocolim_{u/F}\circ G_{u/F}\Longrightarrow G_{\bbone}\circ \hocolim_{u/F}
\] 
is invertible. On the other hand, $u/F\cong F$ is a finite discrete category, so that $\hocolim_{u/F}$ is a finite coproduct, and $G$ commutes with finite coproducts by hypothesis.
%
\end{proof}

\medskip\noindent
\textbf{Homotopical epimorphism.} 
Following~\cite[Sec.\,6]{GS-trees}, we call a functor $v\colon J\to I$ between small categories a {\em homotopical epimorphism} if, for any derivator $\D$, the induced functor $v^*\colon \D(I)\to \D(J)$ is fully faithful. As this holds for any derivator, in particular it holds for all the shifts $\D^K$ (see Example~\ref{expl.shift_of_a_prederivator}) of a given derivator $\D$, showing that $v\colon J\to I$ is a  homotopical epimorphism if, and only if, for any derivator $\D$, it induces a fully faithful morphism of derivators $v^*\colon \D^I\to \D^J$. 

\begin{lemma}\label{homo_epi_implies_lax_epi_lem}
If $v\colon J\to I$ is a homotopical epimorphism between small categories, then it is also a lax epimorphism in the sense of \cite{{adamek2001functors}}, that is, for any category $\C$, the functor $(-)\circ v\colon \C^I\to \C^J$ is fully faithful. 
\end{lemma}
\begin{proof}
By the equivalence ``(1)$\Leftrightarrow$(2)'' of \cite[Thm.\,1.1]{adamek2001functors}, it is enough to check that $(-)\circ v\colon\Set^I\to \Set^J$ is fully faithful. This last condition follows by the definition of homotopical epimorphism applied to the derivator $\mathbb{Y}_\Set$ (see Example \ref{expl.Yoneda_derivator}).
\end{proof}

The following is a useful criterion to identify homotopical epimorphisms:

\begin{lemma}[{\cite[Prop.\,8.2]{GS2017}}]\label{criterion_homo_epi}
Let $u\colon A\to I$ be an essentially surjective functor in $\cat$, let $\D$ be a derivator, and let $\mathbb E\subseteq \D^A$ be a sub-prederivator of $\D^{A}$
(in the sense that $\mathbb{E}(J)$ is a subcategory of $\D^A(J)$ for each $J\in\cat$)
that satisfies the following two conditions:
\begin{enumerate}[\textup(1\textup)]
\renewcommand{\theenumi}{\textup{\arabic{enumi}}}
\item $\mathrm{Im}(u^*\colon \D^I\to \D^A)\subseteq \mathbb E$;
\item for any $J\in \cat$ and ${\mathscr X }\in \mathbb E(J)$, the component of the counit $\epsilon_{\mathscr X}\colon u^*u_*{\mathscr X}\to {\mathscr X}$ of the adjunction $u^*\dashv u_*$ is an isomorphism.
\end{enumerate}
Then, $u^*\colon \D^I\to \D^A$ is fully faithful and  $\mathrm{Im}(u^*)= \mathbb E$. In particular, $\mathbb E$ is a derivator. 
\end{lemma}

\subsection{Pointed and additive derivators} \label{subsec.pointed-derivator}
We refer to \cite{Moritz} for a detailed discussion, as well as for the precise definitions of {\em pointed} derivators (i.e., derivators $\D$ for which $\D(\bbone)$ has a zero object). Similarly, we refer to \cite{groth2012monoidal} for a discussion of {\em additive} derivators (i.e., derivators $\D$ for which $\D(\bbone)$ is additive). In fact, for a pointed (resp., additive) derivator, the condition imposed on $\D(\bbone)$ implies that the categories $\D(I)$ (for each $I$ in $\cat$) are automatically pointed (resp., additive). 

If $\D$ and $\D'$ are pointed (resp., additive) derivators, a morphism $F\colon\D\to\D'$ is called {\em pointed} (resp., {\em additive}) if  $F_I$ is a pointed (resp., additive) functor, for each $I$ in $\cat$. Note that, for each $u\colon J\to I$ in $\cat$, the three morphisms $u^*\colon \D^I\to \D^J$ and $u_!,\, u_*\colon \D^J\to \D^I$ are always pointed (resp., additive). 

\begin{definition}
Let $\D\colon \cat^{\op}\to \Cat$ be an additive derivator and $I\in \cat$. A sequence
\begin{equation}\label{pses_eq_def}
\xymatrix{
{\mathscr X}\ar[r]^\phi&{\mathscr Y}\ar[r]^\psi&{\mathscr Z}
}\qquad\text{in $\D(I)$,}
\end{equation}
is said to be {\em pointwise split-exact} if, for each $i\in I$, the sequence ${\mathscr X}_i\to {\mathscr Y}_i\to {\mathscr Z}_i$, obtained by applying $i^*$ to  \eqref{pses_eq_def}, is split-exact in the additive category $\D(\bbone)$.
\end{definition}

\begin{lemma}
Let $\mathbb{D}\colon\cat^{\op}\to\Cat$ be an additive derivator. For each $I\in\Cat$, the functor $\dia_I\colon\mathbb{D}(I)\to\mathbb{D}(\mathbf{1})^I$ takes pointwise split-exact sequences to kernel-cokernel sequences.
\end{lemma}
\begin{proof} 
Let $\mathscr{X}\stackrel{\phi}{\longrightarrow}\mathscr{Y}\stackrel{\psi}{\longrightarrow}\mathscr{Z}$ be a pointwise split-exact sequence in $\mathbb{D}(I)$. It follows that 
  $\mathscr{X}_i=(\text{dia}_I\mathscr{X})_i\stackrel{\phi_i}{\longrightarrow}\mathscr{Y}_i=(\text{dia}_I\mathscr{Y})_i\stackrel{\psi_i}{\longrightarrow}\mathscr{Z}_i=(\text{dia}_I\mathscr{Z})_i$ is a split-exact sequence, whence a kernel-cokernel one, in $\mathbb{D}(\mathbf{1})$,  for all $i\in I$. Now use the fact a morphism in $\mathbb{D}(\mathbf{1})^I$ has a co/kernel if, and only if, so do all of its components and, in this case, the co/kernel is computed componentwise.
 \end{proof}

\subsection{Strong and stable derivators}

We refer to \cite{Moritz} for a detailed discussion, as well as for the precise definitions of {\em stable} derivators (they are the pointed derivators in which a coherent commutative square is cartesian if, and only if, it is cocartesian), and of {\em strong} derivators (the ones in which the partial diagram functors $\D(\mathbf{2}\times I) \to \D(I)^\mathbf{2}$ are full and essentially surjective for each $I\in\cat$). 
The key fact, which can be found in~\cite[Thm.\,4.16 and Coro.\,4.19]{Moritz}, is that, given a strong and stable derivator $\D$, each $\D(I)$ is in fact a triangulated category.

Furthermore, if $\D$ and $\D'$ are stable derivators, left and right exactness of a morphism $F\colon\D\to\D'$ are equivalent and left (=right) exact morphisms are simply called \emph{exact} in this context. If $F\colon\D\to\D'$ is an exact morphism of strong and stable derivators, all $F_I\colon\D(I)\to\D'(I)$ are naturally triangulated functors by~\cite[Prop.\,4.18]{Moritz}. In particular, all the morphisms of the form $u^*$ and the Kan extensions $u_!$, $u_*$, for some $u\colon J\to I$ in $\cat$, are naturally triangulated functors. 

\begin{lemma}\label{split_induces_triangle_lemma}
Let $\D\colon \cat^{\op}\to \Cat$ be a strong and stable derivator, $I\in \cat$ and let
\[
\xymatrix{
{\mathscr X}\ar[r]^\phi&{\mathscr Y}\ar[r]^\psi&{\mathscr Z}
}
\]
be a pointwise split-exact sequence in $\D(I)$. Then, this sequence can be completed to a distinguished triangle in $\D(I)$. 
\end{lemma}
\begin{proof}
Consider the morphism $\phi\colon {\mathscr X}\to{\mathscr Y}$ and complete it to a triangle in $\D(I)$:
\[
\xymatrix{
{\mathscr X}\ar[r]^\phi&{\mathscr Y}\ar[r]^{\phi'}&{\mathscr C}\ar[r]^{\phi''}&\Sigma \mathscr X.
}
\]
Since $\psi\phi=0$ by hypothesis, there is a morphism $f\colon{\mathscr C}\to {\mathscr Z}$ such that $f\phi'=\psi$; let us verify that $f$ is an isomorphism. By (Der.2), it is enough to verify that $f_i:=i^*(f)$ is an isomorphism, for each $i$ in $I$. But this is clearly true since $\phi_i$ is a split monomorphism (by hypothesis) and both $\phi'_i$ and $\psi_i$ are cokernels of $\phi_i$ (the first by \cite[Lem.\,1.2.6]{Neeman} and the second by hypothesis). Hence, by the universality of cokernels, $f_i$ is the unique morphism such that $f_i\phi_i'=\psi_i$ and it is an isomorphism. 
\end{proof}

In the next example, we mention some classes of strong and stable derivators that will appear frequently in the rest of the paper:

\begin{example} \label{expl.derived categories}
Let $(\C,\mathcal W,\B,\mathcal F)$ be a model category. For any small category $I$, let $\W_I$ be the class of morphisms in $\C^I$ which belong pointwise to $\W$. By \cite[Thm.\,1]{Cis03}, the universal localization $\C^I[\W_I^{-1}]$ can always be constructed and, furthermore, the assignment $I\mapsto \C^{I}[\W_I^{-1}]$ underlies a derivator $\D_{(\C,\W)}\colon \cat^{\op}\to \Cat$. Furthermore, $\D_{(\C,\W)}$ is always strong and it is pointed (resp., stable) if $\C$ has the same property in the sense of model categories. For such derivators, homotopy co/limits and, more generally, homotopy Kan extensions, are just the total derived functors of the usual co/limit and Kan extension functors. 

Given a Grothendieck Abelian category $\G$, 
we refer to the strong and stable derivator arising as above from the injective model structure on $\Ch(\G)$ 
as the {\em canonical derivator enhancing the derived category} $\Der(\G)$. 
\end{example}

\section{Categories of finite and of countable length}

In this section we introduce the concept of category of finite length and that of category of countable length. Then, given such a category $I$, we study canonical presentations for objects in $\D(I)$ (first in the finite and then in the countable length case), where $\D$ is allowed to be an arbitrary additive derivator. Finally, we verify that every small category is the homotopically surjective image of a suitable category of countable length. In this way the results about standard presentations for coherent diagrams over categories of countable length can be applied to arbitrary shapes.
%
%

\subsection{Categories of finite length}

\begin{definition}
A small category $I$ is said to be of {\em finite length} if there is $n\in\N_{>0}$ and
\[
d\colon I\to \mathbf{n}^{\op},
\]
 a functor such that, for any pair of objects $i,\, j\in I$, if there is a non-identity morphism in $I(i,j)$, then $d(i)>d(j)$. If $I$ is a category of finite length, the smallest $n\in\N_{>0}$ for which there exists a functor $d\colon I\to \mathbf{n}^{\op}$ as above is called the {\em length} of $I$, in symbols $\ell(I)=n$. 
\end{definition}

Note that, for a category of finite length $I$, we have that $I(i,i)=\{\id_i\}$ for any $i\in I$. Furthermore, given $i\neq j\in I$, at most one of $I(i,j)$ and $I(j,i)$ is not empty. Moreover, the length of $I$ is exactly the length of a maximal path in $I$.

\begin{definition}
Given a category $I$ of finite length, we  say that an object $i$ in $I$ is {\em minimal} if there is no non-identity morphism starting at $i$. 
\end{definition}

Throughout this subsection, we fix the following notation:

\begin{notation}\label{notation_fl}
We let $\D\colon \cat^\op\to \Cat$ be an additive derivator, $I$ a category of finite length and $u\colon J\to I$ the inclusion of the (full) subcategory $J$ of $I$ of all the non-minimal objects. We also denote by $I\setminus J\subseteq I$ the subcategory of minimal objects. We denote by $u_!\adjunct{\eta_u}{\epsilon_u}u^*$ the induced adjunction. Furthermore, for each $i\in I\setminus J$ we have the following adjunctions:
\[
i_!\adjunct{\eta_i}{\epsilon_i}i^*\qquad\text{and}\qquad \hocolim_{u/i}\adjunct{\eta_{u/i}}{\epsilon_{u/i}}\pt_{u/i}^*.
\] 
Finally, we denote by $\gamma_!\colon (\pt_{u/i})_!\pr_i^*\Rightarrow i^*u_!$ the canonical natural transformation, which is invertible by (Der.4).
\end{notation}

The following lemma is a trivial fact about categories of finite length that will be extremely important when trying to prove things by induction:

\begin{lemma}
For each $i\in I\setminus J$, we have that $\ell(I)>\ell(J)\geq \ell(u/i)$.
\end{lemma}
\begin{proof}
The fact that $\ell(I)>\ell(J)$ follows since any maximal path in $I$ ends in a minimal object that, by definition, does not belong in $J$. Furthermore, a maximal path in $u/i$ is something of the form 
\[
(u(j_1)\overset{\phi_1}{\longrightarrow}i)\overset{\psi_1}{\longrightarrow}(u(j_2)\overset{\phi_2}{\longrightarrow}i)\overset{\psi_2}{\longrightarrow}\ldots\overset{\psi_{k-1}}{\longrightarrow}(u(j_k)\overset{\phi_k}{\longrightarrow}i)
\]
where $\psi_s\colon j_s\to j_{s+1}$ and $\phi_{s+1}\psi_s=\phi_s$, for $s=1,\ldots,k-1$. Then, 
$$j_1\overset{\psi_1}{\longrightarrow}j_2\overset{\psi_2}{\longrightarrow}\ldots \overset{\psi_{k-1}}{\longrightarrow}j_k$$
is a path in $J$, so that $\ell(u/i)\leq \ell(J)$.
\end{proof}

We can now introduce the standard presentation for coherent diagrams of finite length:

\begin{lemma}\label{lifting_inductive_step}
Given ${\mathscr X}$ in $\D(I)$, there is a pointwise split-exact sequence:
\[
\xymatrix@C=50pt@R=9pt{
\coprod_{I\setminus J}i_!i^*u_!u^*{\mathscr X}\ar[rr]^-{
\footnotesize
\left[\begin{smallarray}{ccc}
\ddots &&{\text{\Large $0$}}\\
&i_!i^*\epsilon_u&\\
{\text{\Large $0$}}&&\ddots\\
\hline
\cdots&-\epsilon_i&\cdots
\end{smallarray}\right]
}%
&& \coprod_{I\setminus J}i_!{\mathscr X}_i       \sqcup    u_!u^*{\mathscr X}   \ar[rr]^-{
\footnotesize
\left[\begin{array}{ccc|c}
\cdots &\epsilon_i&\cdots&\epsilon_u
\end{array}\right]
}&&{\mathscr X}.
}
\]
\end{lemma}
\begin{proof}
We have to verify that, for each $k\in I$, when we apply the evaluation functor $k^*$ to the sequence in the statement, we get a split-exact sequence in $\D(\bbone)$. We distinguish two cases based on whether $k$ is minimal or not.
 Indeed, if $k\in J$ then $k^*i_!=0$ for all $i\in I\setminus J$ since then $\{i\}\subseteq I$ is a cosieve (cf.\ Remark~\ref{rem.cosieve}), and $k^*\epsilon_u$ is an isomorphism by Lemma~\ref{easy_relations_for_u_and_k}, so we get the following split-exact sequence:
\[
\xymatrix@C=30pt@R=9pt{
0\ar[rrr]^-{0}&&& k^*u_!u^*{\mathscr X}\ar[rrr]^-{k^*\epsilon_u}&&& k^*{\mathscr X}.
}
\]
Suppose now that $k\in I\setminus J$; using the natural isomorphism $k^*\epsilon_k\colon k^*k_!k^*\ \tilde\Longrightarrow \ k^*$ (see Lemma \ref{easy_relations_for_u_and_k} recalling that the functor $k\colon \bbone \to I$ is fully faithful), we get the following split-exact sequence:
\[
\xymatrix@C=29pt@R=9pt{
k^*k_!k^*u_!u^*{\mathscr X}\ar[rrr]^-{
\footnotesize
\left[\begin{array}{c}
k^*k_!k^*\epsilon_u\\
\\
-k^*\epsilon_k
\end{array}\right]
}&&&k^*k_!k^*{\mathscr X}\sqcup  k^*u_!u^*{\mathscr X}\ar[rrr]^-{
\footnotesize
\left[\begin{array}{cc}
k^*\epsilon_k &k^*\epsilon_u
\end{array}\right]
}
&&& k^*{\mathscr X},
}
\]
as this is an occurrence of the second part of Example \ref{countable_splitexact_ex}(1).
\end{proof}

The following two technical propositions will be essential in the following section. Note that, even if we state them in the setting of this subsection, they hold for any small category $I$, any full subcategory $J$ and any $i\in I\setminus J$.

\begin{proposition}\label{morph_fin_length_prop_1}
Given ${\mathscr X},{\mathscr Y}\in\D(I)$ and $i\in I\setminus J$, the following diagram commutes:
\[
\xymatrix@C=150pt@R=21pt{
\D(I)(i_!i^*{\mathscr X},{\mathscr Y})\ar[r]^-{(-)\circ i_!i^*(\epsilon_u)}\ar[ddd]^-{\rotatebox[origin=c]{90}{$\sim$}}_-{i^*(-)\circ\eta_i}&\D(I)(i_!i^*u_!u^*{\mathscr X},{\mathscr Y})\ar[d]_-{\rotatebox[origin=c]{90}{$\sim$}}^-{i^*(-)\circ\eta_i}\\
&\D(\bbone)(i^*u_!u^*{\mathscr X},i^*{\mathscr Y})\ar[d]_-{\rotatebox[origin=c]{90}{$\sim$}}^-{(-)\circ\gamma_!}\\
&\D(\bbone)(\hocolim_{u/i}\pr_i^*u^*{\mathscr X},i^*{\mathscr Y})\ar[d]_-{\rotatebox[origin=c]{90}{$\sim$}}^-{\pt_{u/i}^*(-)\circ\eta_{u/i}}\\
\D(\bbone)(\mathscr X_i,\mathscr Y_i)\ar[r]_{\pt^*_{u/i}(-)\circ\gamma^*}&\D(u/i)(\pr_i^*u^*(\mathscr X),\pt_{u/i}^*\mathscr Y_i)
}
\]
where the columns are natural isomorphisms. Furthermore, for each $f\colon \mathscr X_i\to \mathscr Y_i$ the corresponding map $\pt^*_{u/i}(f)\circ\gamma^*\colon \pr_i^*u^*(\mathscr X)\to \pt_{u/i}^*\mathscr Y_i$ acts as follows: for each object $(j,a\colon j\to i)$ in $u/i$,
\[
(j,a)^*(\pt^*_{u/i}(f)\circ\gamma^*)=f\circ \mathscr X_a\colon \mathscr X_j\to \mathscr Y_i.
\]
\end{proposition}
\begin{proof}
The left column and the first morphism in the right column are induced by the adjunction $i_!\dashv i^*$ so they are clearly isomorphisms. The second and the third morphism in the right column are induced, respectively, by the natural transformation $\gamma_!$, which is invertible by (Der.4), and by the adjunction $\hocolim_{u/i}\dashv\pt_{u/i}^*$, so it is invertible too. Finally, to see that the square commutes, consider the following equalities:
\begin{align*}
\pt^{*}_{u/i}(i^*(-)\circ\eta_i)\circ \gamma^*&=\pt^*_{u/i}(i^*(-)\circ\eta_i\circ i^*(\epsilon_u)\circ\gamma_!)\circ \eta_{u/i}\\
&=\pt^*_{u/i}(i^*(-)\circ i^*i_!i^*(\epsilon_u)\circ\eta_i\circ\gamma_!)\circ \eta_{u/i},
\end{align*}
where the first equality holds since $\gamma^*=(\gamma_!)_*$ (see Remark \ref{mates_rem_der4}) and the second one by the naturality of $\eta_i\colon \id_{\D(\bbone)}\Rightarrow i^*i_!$.
\end{proof}

\begin{proposition}\label{morph_fin_length_prop_2}
Given ${\mathscr X},{\mathscr Y}\in\D(I)$ and $i\in I\setminus J$, the following diagram commutes 
\[
\xymatrix@C=150pt@R=21pt{
\D(I)(u_!u^*({\mathscr X}),{\mathscr Y})\ar[r]^-{(-)\circ\epsilon_i}\ar[dr]_-{i^*(-)}\ar[ddd]^-{\rotatebox[origin=c]{90}{$\sim$}}_-{u^*(-)\circ\eta_u}&\D(I)(i_!i^*u_!u^*{\mathscr X},{\mathscr Y})\ar[d]_-{\rotatebox[origin=c]{90}{$\sim$}}^-{i^*(-)\circ\eta_i}\\
&\D(\bbone)(i^*u_!u^*{\mathscr X},i^*{\mathscr Y})\ar[d]_-{\rotatebox[origin=c]{90}{$\sim$}}^-{(-)\circ\gamma_!}\\
&\D(\bbone)(\hocolim_{u/i}\pr_i^*u^*{\mathscr X},i^*{\mathscr Y})\ar[d]_-{\rotatebox[origin=c]{90}{$\sim$}}^-{\pt_{u/i}^*(-)\circ\eta_{u/i}}\\
\D(J)(u^*(\mathscr X),u^*(\mathscr Y))\ar[r]_{\gamma^*\circ \pr^*_{i}(-)}&\D(u/i)(\pr_i^*u^*(\mathscr X),\pt_{u/i}^*\mathscr Y_i)
}
\]
where the columns are natural isomorphisms. Furthermore, for each $f\colon u^*\mathscr X\to u^*\mathscr Y$ the corresponding map $\gamma^*\circ \pr^*_{i}(f)\colon \pr_i^*u^*(\mathscr X)\to \pt_{u/i}^*\mathscr Y_i$ acts as follows: for each object $(j,a\colon j\to i)$ in $u/i$,
\[
(j,a)^*(\gamma^*\circ \pr^*_{i}(f))=\mathscr Y_a\circ f_j\colon \mathscr X_j\to \mathscr Y_i.
\]
\end{proposition}
\begin{proof}
The three morphisms in the right column are the same as the ones in Proposition~\ref{morph_fin_length_prop_1}, so they are isomorphisms. Furthermore, the upper triangle commutes by the usual triangular equalities relative to the adjunction $i_!\dashv i^*$. On the other hand, the morphism in the left column is the natural isomorphism induced by the adjunction $u_!\dashv u^*$, whose inverse is the map $\epsilon_u\circ u_!(-)\colon \D(J)(u^*(\mathscr X),u^*(\mathscr Y))\to \D(I)(u_!u^*({\mathscr X}),{\mathscr Y})$. Hence, the diagram in the statement commutes if, and only if, the following diagram commutes:
\[
\xymatrix@C=150pt@R=22pt{
\D(I)(u_!u^*({\mathscr X}),{\mathscr Y})\ar[r]^-{i^*(-)}\ar@{<-}[dd]^-{\rotatebox[origin=c]{90}{$\sim$}}_-{\epsilon_u\circ u_!(-)}
&\D(\bbone)(i^*u_!u^*{\mathscr X},i^*{\mathscr Y})\ar[d]_-{\rotatebox[origin=c]{90}{$\sim$}}^-{(-)\circ\gamma_!}\\
&\D(\bbone)(\hocolim_{u/i}\pr_i^*u^*{\mathscr X},i^*{\mathscr Y})\ar[d]_-{\rotatebox[origin=c]{90}{$\sim$}}^-{\pt_{u/i}^*(-)\circ\eta_{u/i}}\\
\D(J)(u^*(\mathscr X),u^*(\mathscr Y))\ar[r]_{\gamma^*\circ \pr^*_{i}(-)}&\D(u/i)(\pr_i^*u^*(\mathscr X),\pt_{u/i}^*\mathscr Y_i).
}
\]
Finally, the above diagram commutes by the following series of equalities:
\begin{align*}
&\pt^*_{u/i}(i^*(\epsilon_u\circ u_!(-))\circ\gamma_!)\circ \eta_{u/i}=\\
&=\pt^*_{u/i}(i^*(\epsilon_u)\circ i^*u_!(-)\circ\gamma_!)\circ \eta_{u/i}\\
&=\pt^*_{u/i}(i^*(\epsilon_u)\circ \gamma_!\circ(\pt_{u/i})_!\pr_i(-))\circ \eta_{u/i}&\text{by naturality of $\gamma_!$;}\\
&=\pt^*_{u/i}(i^*(\epsilon_u))\circ \pt^*_{u/i}(\gamma_!)\circ\pt^*_{u/i}(\pt_{u/i})_!\pr_i(-)\circ \eta_{u/i}\\
&=\pt^*_{u/i}(i^*(\epsilon_u))\circ \pt^*_{u/i}(\gamma_!)\circ\eta_{u/i}\circ \pr_i(-)&\text{by naturality of $\eta_{u/i}$;}\\
&=\gamma^*\circ\pr^*_i(-)&\text{since $(\gamma_!)_*=\gamma^*$.}&\qedhere
\end{align*}
\end{proof}

\subsection{Categories of countable length}

\begin{definition}\label{countable_length}
A small category $I$ is said to be of {\em countable length} if there is a functor
\[
d\colon I\to \N^{\op}
\]
such that, for any pair of objects $i,\, j\in I$, if there is a non-identity morphism in $I(i,j)$, then $d(i)>d(j)$ in~$\N$. 
\end{definition}

As for the finite length case, if $I$ has countable length, then $I(i,i)=\{\id_i\}$ for any $i\in I$. Furthermore, given $i\neq j\in I$, at most one of $I(i,j)$ and $I(j,i)$ is not empty. Throughout this subsection, we fix the following notation:

\begin{notation}\label{notation_countable_cat}
We let $\D\colon \cat^\op\to \Cat$ be an additive derivator and $I$ a category of countable length with a fixed functor $d\colon I\to \N^{\op}$ like in Definition \ref{countable_length}. For each $n\in\N$, we let 
\[
I_{\leq n}\subseteq I\quad \text{be the (full) subcategory of all }\{ i\in I: d(i)\leq n-1\}.
\] 
Furthermore, we let $\iota_{m,n}\colon I_{\leq n}\to I_{\leq m}$ and $\iota_{n}\colon I_{\leq n}\to I$ be the canonical cosieve inclusions, for all $m\geq n$ in $\N$. Finally, we consider the adjunctions:
\[
(\iota_n)_!\adjunct{\eta_n}{\epsilon_n}\iota_n^*\qquad\text{and}\qquad(\iota_{m,n})_!\adjunct{\eta_{m,n}}{\epsilon_{m,n}}\iota_{m,n}^*
\]
and we define
$\partial^n:=(\epsilon_n\circledast (\iota_{n+1})_!\iota_{n+1}^*)\circ ((\iota_n)_!\iota_{n+1,n}^*\circledast \eta_{n+1}\circledast \iota_{n+1}^*)\colon (\iota_n)_!\iota_n^*\Rightarrow(\iota_{n+1})_!\iota_{n+1}^*$.
\end{notation}

Let us remark that, in the above notation, the restriction $d_n\colon I_{\leq n}\to \mathbf{n}^{\op}$ of the functor $d\colon I\to \N^{\op}$ shows that each $I_{\leq n}$ is a category of finite length. Furthermore, for each $m\geq n\in \N$, we have $\iota_m\circ \iota_{m,n}=\iota_n$.
In particular, the equality $\iota_{n+1,n}^*\circ \iota_{n+1}^*=\iota_n^*$ and Lemma~\ref{lem.adjoints} provide us with a natural isomorphism $\psi_n\colon (\iota_n)_! \stackrel{\sim}\Longrightarrow(\iota_{n+1})_!(\iota_{n+1,n})_!$ given by
\begin{equation}\label{eq.composition-of-left-adjoints}
\psi_n = (\epsilon_n\circledast(\iota_{n+1})_!(\iota_{n+1,n})_!) \circ
((\iota_n)_!\iota_{n+1,n}^*\circledast\eta_{n+1}\circledast(\iota_{n+1,n})_!) \circ
((\iota_n)_!\circledast\eta_{n+1,n}).
\end{equation}
Let now $\mathscr X\in \D(I)$ and $n\in \N$, the following commutative diagram helps visualize $\partial^n$:
\[
\xymatrix@C=120pt@R=5pt{
(\iota_n)_!\iota_{n+1,n}^*\iota_{n+1}^*{\mathscr X}\ar[r]^-{(\iota_n)_!\iota_{n+1,n}^*(\eta_{n+1})}\ar@{=}[d]&(\iota_n)_!\iota_{n+1,n}^*  \iota_{n+1}^*(  \iota_{n+1})_!  \iota_{n+1}^*{\mathscr X}\ar@{=}[d]\\
(\iota_n)_!\iota_n^*{\mathscr X}\ar[rddd]_-{\partial^n}&(\iota_n)_!\iota_{n}^*(  \iota_{n+1})_!  \iota_{n+1}^*{\mathscr X}\ar[ddd]^-{\epsilon_n}\\
\\
\\
&(\iota_{n+1})_!  \iota_{n+1}^*{\mathscr X}.}
\]

It is also instructive to recall Lemma~\ref{easy_relations_for_u_and_k} and Remark~\ref{rem.cosieve} in the context
of the above definitions. For example, given $\mathscr X\in\D(I)$, the component of $((\iota_n)_!\iota_n^*\mathscr X)_k$ is isomorphic to $\mathscr X_k$ if $k\in I_{\le n}$ and $((\iota_n)_!\iota_n^*\mathscr X)_k=0$ otherwise. The components of the counit $(\iota_n)_!\iota_n^*\mathscr X\to\mathscr X$ are isomorphisms or maps from the zero object of $\D(\bbone)$, respectively. The next lemma shows that the map $\partial^n$ is of a similar nature (in fact, a direct but tedious computation also reveals that $\partial^n$ coincides, up to the identification given by $\psi_n$ from \eqref{eq.composition-of-left-adjoints}, with $(\iota_{n+1})_!\circledast\epsilon_{n+1,n}\circledast\iota_{n+1}^*\colon (\iota_{n+1})_!(\iota_{n+1,n})_!\iota_n^*\Longrightarrow(\iota_{n+1})_!\iota_{n+1}^*$).

\begin{lemma}\label{ed=e_lemma}
In the above notation, $\epsilon_{n+1}\circ \partial^n=\epsilon_n$, for each $n\in\N$.
\end{lemma}
\begin{proof}
Let $n\in \N$ and note that, by the naturality of $\epsilon_n$, we have:
\[
\epsilon_{n+1}\circ (\epsilon_n\circledast (\iota_{n+1})_!\iota_{n+1}^*)=\epsilon_{n}\circ ( (\iota_{n})_!\iota_{n}^*\circledast\epsilon_{n+1}).
\]
In other words, the following diagram commutes for each $\mathscr X\in \D(I)$:
\begin{equation}\label{nat_of_en_diagram}
\xymatrix@C=70pt{
(\iota_n)_!\iota_n^*{\mathscr X}\ar@{<-}[r]^-{(\iota_n)_!\iota_n^*(\epsilon_{n+1})}\ar[d]_{\epsilon_n}&(\iota_n)_!\iota_n^*(\iota_{n+1})_!\iota_{n+1}^*{\mathscr X}\ar[d]^{\epsilon_n}\\
{\mathscr X}\ar@{<-}[r]_-{\epsilon_{n+1}}&(\iota_{n+1})_!\iota_{n+1}^*{\mathscr X}.
}
\end{equation}
Therefore,  we have that 
\begin{align*}
\epsilon_{n+1}\circ \partial^n&=\epsilon_{n+1}\circ (\epsilon_n\circledast (\iota_{n+1})_!\iota_{n+1}^*)\circ ((\iota_n)_!\iota_{n+1,n}^*\circledast \eta_{n+1}\circledast \iota_{n+1}^*)&\text{}\\
&=\epsilon_{n}\circ ( (\iota_{n})_!\iota_{n}^*\circledast\epsilon_{n+1})\circ ((\iota_n)_!\iota_{n+1,n}^*\circledast \eta_{n+1}\circledast \iota_{n+1}^*)&\text{}\\
&=\epsilon_n\circ ((\iota_n)_!\iota_{n+1,n}^*\circledast ((\iota_{n+1}^*\circledast\epsilon_{n+1})\circ (\eta_{n+1}\circledast \iota_{n+1}^*)))=\epsilon_n,
\end{align*}
where we have applied the triangular equality $(\iota_{n+1}^*\circledast\epsilon_{n+1})\circ (\eta_{n+1}\circledast \iota_{n+1}^*)=\id_{\iota_{n+1}^*}$ relative to the adjunction $(\iota_{n+1})_!\dashv \iota_{n+1}^*$.
\end{proof}

We are now ready to introduce the standard presentation of a coherent diagram of countable length:

\begin{lemma}\label{dist_triangle_countable_length}
For any ${\mathscr X}\in \D(I)$, there is a pointwise split-exact sequence
\[
\xymatrix@C=130pt{
\coprod_{\N}(\iota_n)_!\iota_n^*{\mathscr X}\ar[r]^{
\footnotesize
\alpha=
\left[\begin{smallarray}{cccc}
1&&&\\
-\partial^0&1&&{\text{\Large $0$}}\\
&-\partial^1&1&\\
{\text{\Large $0$}}&&\ddots&\ddots
\end{smallarray}\right]
}&\coprod_{\N}(\iota_n)_!\iota_n^*{\mathscr X}\ar[r]^-{
\footnotesize
\beta=
\left[\begin{smallarray}{cccc}
\epsilon_0 &\epsilon_1 &\epsilon_2 &\cdots
\end{smallarray}\right]
}&{\mathscr X}.
}
\] 
%
\end{lemma}
\begin{proof}
It is easily seen that, by Lemma \ref{ed=e_lemma}, $\beta\circ\alpha=0$. Let now $k\in I_{\leq n}\setminus I_{\leq n-1}\subseteq I$ (for some integer $n\geq1$) and let us show that, after applying $k^*$ to the sequence in the statement, we get a split-exact sequence in $\D(\bbone)$. By \cite[Prop.\,3.6]{Moritz} we get that
\[
k^*(\iota_m)_!\iota_m^*=0,\qquad \text{whenever $m<n$},
\]
while $k^*(\epsilon_m)$ is invertible for all $m\geq n$, by Lemma \ref{easy_relations_for_u_and_k}, and $k^*(\epsilon_{m+1})\circ k^*( \partial^m)=k^*(\epsilon_m)$, by Lemma \ref{ed=e_lemma}. In particular, we obtain the following commutative diagram in $\D(\bbone)$, where the three columns are isomorphisms
\[
\xymatrix@C=100pt{
\coprod_{m\geq n}k^*(\iota_m)_!\iota_m^*{\mathscr X}\ar[r]^{\scriptsize k^*(\alpha)}\ar[d]^-{\rotatebox[origin=c]{90}{$\sim$}}_{(k^*(\epsilon_m))_{m}}&\coprod_{m\geq n}k^*(\iota_m)_!\iota_m^*{\mathscr X}\ar[r]^{\scriptsize k^*(\beta)}\ar[d]_-{\rotatebox[origin=c]{90}{$\sim$}}^{(k^*(\epsilon_m))_{m}}&{\mathscr X}_k\ar@{=}[d]\\
{\mathscr X}_k^{(\N)}\ar[r]|-{
\footnotesize
\left[\begin{smallarray}{cccc}
1&&&\\
-1&1&&{\text{\Large $0$}}\\
&-1&1&\\
{\text{\Large $0$}}&&\ddots&\ddots
\end{smallarray}\right]
}&{\mathscr X}_k^{(\N)}\ar[r]_{
\footnotesize
\left[\begin{smallarray}{cccc}
1&1&1&\cdots
\end{smallarray}\right]
}&{\mathscr X}_k
}
\]
showing that our sequence is (isomorphic to) the one in Example \ref{countable_splitexact_ex}.
\end{proof}

The following technical proposition will be essential in the following section:

\begin{proposition}\label{morph_count_length_prop}
Given ${\mathscr X},{\mathscr Y}\in\D(I)$ and $n\in \N$, the following diagram commutes 
\[
\xymatrix@C=100pt@R=30pt{
\D(I)((\iota_n)_!\iota_n^*({\mathscr X}),{\mathscr Y})\ar@{<-}[r]^-{(-)\circ\partial^n}\ar[d]^-{\rotatebox[origin=c]{90}{$\sim$}}_-{\iota_n^*(-)\circ\eta_n}&\D(I)((\iota_{n+1})_!\iota_{n+1}^*({\mathscr X}),{\mathscr Y})\ar[d]_-{\rotatebox[origin=c]{90}{$\sim$}}^-{\iota_{n+1}^*(-)\circ\eta_{n+1}}\\
\D(I_{\leq n})(\iota_n^*({\mathscr X}),\iota_n^*({\mathscr Y}))\ar@{<-}[r]^{\iota_{n+1,n}^*(-)}&\D(I_{\leq n+1})(\iota_{n+1}^*({\mathscr X}),\iota_{n+1}^*({\mathscr Y}))}
\]
where the columns are natural isomorphisms.
\end{proposition}
\begin{proof}
The two columns are isomorphisms because they are the maps given by the adjunctions $(\iota_n)_!\dashv \iota_n^*$ and $(\iota_{n+1})_!\dashv \iota_{n+1}^*$, respectively. In particular, the inverse of the map in the leftmost column is $\epsilon_{n}\circ(\iota_n)_!(-)$ and, therefore, the commutativity of the original square is equivalent to the commutativity of the following diagram:
\[
\xymatrix@C=100pt@R=30pt{
\D(I)((\iota_n)_!\iota_n^*({\mathscr X}),{\mathscr Y})\ar@{<-}[r]^-{(-)\circ\partial^n}\ar@{<-}[d]^-{\rotatebox[origin=c]{90}{$\sim$}}_-{\epsilon_{n}\circ(\iota_n)_!(-)}&\D(I)((\iota_{n+1})_!\iota_{n+1}^*{\mathscr X},{\mathscr Y})\ar[d]_-{\rotatebox[origin=c]{90}{$\sim$}}^-{\iota_{n+1}^*(-)\circ\eta_{n+1}}\\
\D(I_{\leq n})(\iota_n^*({\mathscr X}),\iota_n^*({\mathscr Y}))\ar@{<-}[r]^{\iota_{n+1,n}^*(-)}&\D(I_{\leq n+1})(\iota_{n+1}^*({\mathscr X}),\iota_{n+1}^*({\mathscr Y})).}
\]
That is, we have to prove that $\epsilon_{n}\circ(\iota_n)_!(\iota_{n+1,n}^*(\iota_{n+1}^*(-)\circ\eta_{n+1}))=(-)\circ\partial^n$. But, in fact,
\begin{align*}
\epsilon_{n}\circ(\iota_n)_!(\iota_{n+1,n}^*(\iota_{n+1}^*(-)\circ\eta_{n+1}))&=\epsilon_{n}\circ(\iota_n)_!\iota_{n}^*(-)\circ(\iota_n)_!\iota_{n+1,n}^*(\eta_{n+1})\\
&=(-)\circ \epsilon_{n}\circ(\iota_n)_!\iota_{n+1,n}^*(\eta_{n+1})=(-)\circ\partial^n,
\end{align*}
where the first equality holds since $\iota_{n+1}\circ\iota_{n+1,n}=\iota_n$, the second one is true by the naturality of $\epsilon_n$ (analogously to the commutativity of the diagram \eqref{nat_of_en_diagram} but with an arbitrary map instead of $\epsilon_{n+1}$), and the last one follows just by definition of $\partial^n$.
\end{proof}

\subsection{Homotopical epimorphic images of categories of countable length}\label{homotopical_epi_subs}

Let $I$ be a small category. In this section we associate with $I$ a new category of countable length $\Delta(I)$ and a homotopical epimorphism $u\colon \Delta(I) \to I$. The idea of using the category $\Delta(I)$ in this way was suggested to us by Fritz H{\"o}rmann. A similar idea appears in his paper \cite{hoermann-enlargement}, where he proves much more in this context---he uses the categories $\Delta(I)$ to extend the domain of definition of fibred (multi)derivators from directed categories of countable length to arbitrary shapes. Let us start with the following definition:

\begin{definition}
Given a small category $I$, the category $\Delta(I)$ is defined as follows:
\begin{itemize}
\item the objects of $\Delta(I)$ are the functors $\mathbf{n}\to I$, for $n\in\N_{>0}$;
\item given two objects $a\colon \mathbf{n}\to I$ and $b\colon \mathbf{m}\to I$, we define
\[
\Delta(I)(a,b):=\{\text{$\varphi\colon{\mathbf{m}\to\mathbf{n}}$ strictly increasing and s.t. $a\circ\varphi=b$}\}.
\]
\end{itemize}
\end{definition}
There is a canonical functor
\[
u\colon\Delta(I)\longrightarrow I\qquad\text{such that}\qquad (a\colon \mathbf{n}\to I)\mapsto a(0).
\]
In what follows, our aim is to prove that the above functor $u\colon \Delta(I)\to I$ is a homotopical epimorphism. Given $i\in I$, define the {\em fiber} of $u$ at $i$ to be the category $u^{-1}(i)$, given by the following pullback diagram in $\cat$:
\begin{equation}\label{pb_defining_fibers}
\xymatrix{
u^{-1}(i)\ar@{.>}[r]^{\iota_i}\ar@{}[dr]|{P.B.}\ar@{.>}[d]_{\pt_{u^{-1}(i)}}&\Delta(I)\ar[d]^u\\
\bbone\ar[r]_{i}&I
}
\end{equation}
We can describe $u^{-1}(i)$ as the subcategory of the slice category $i/u$ of those objects $(a,\varphi\colon i\to u(a))$ such that $u(a)=i$ and $\varphi=\id_i$. The following lemma is taken from the discussion in~\cite[Appendix A]{GPS-additivity}.

\begin{lemma}\label{basic_properties_delta_construction}
In the above notation, the following statements hold true:
\begin{enumerate}[\textup(1\textup)]
\renewcommand{\theenumi}{\textup{\arabic{enumi}}}
\item $\Delta(I)$ has countable length;
\item the inclusion $u^{-1}(i)\to i/u$ has a right adjoint;
\item $u^{-1}(i)$ has a terminal object.
\end{enumerate}
\end{lemma}
\begin{proof}
(1). We define a functor $d\colon \Delta(I)\to \N^{\op}$
as follows: $d(\mathbf{n}\to I):=n-1$ and, given a morphism $(\mathbf{n}\to I)\overset{\varphi}{\longrightarrow} (\mathbf{m}\to I)$, we have that $n\geq m$, so we can map $\varphi$ to the unique map $(n-1)\to (m-1)$ in $\N^{\op}$. It is now easy to verify that this functor satisfies the conditions of Definition \ref{countable_length}.

\smallskip\noindent
(2). Let $x=(a_0\to a_1\to \ldots\to a_n, \varphi\colon i\to a_0)\in i/u$ and suppose that $\varphi\ne\id_x$ (i.e., that $x\not\in u^{-1}(i)$). Let $c(x)$ be the following object of $u^{-1}(i)$:
\[
c(x)=(i\overset{\varphi}{\longrightarrow}a_0\to a_1\to \ldots\to a_n, \id_i\colon i\to i).
\]
Of course the morphism $\mathbf{n+1}\to\mathbf{n+2}$ such that $k\mapsto k+1$ gives a well-defined morphism $c(x)\to x$ in $i/u$. It is easy to verify that this morphism is a coreflection of $x$ onto $u^{-1}(i)$.

\smallskip\noindent
(3). One checks directly that $(\bbone \overset{i}{\longrightarrow}I, \id_i\colon i\to i)$ is a terminal object in $u^{-1}(i)$.
\end{proof}

Before introducing the following definition, let us recall that an object $\mathscr X\in \D(J)$ is said to be constant if, and only if, $\mathscr X\cong \pt_J^*\mathscr X_j$ for some (i.e., all) $j\in J$. Note also that, if $J$ has a terminal object $t$, then there is a unique natural transformation $\alpha\colon \id_J\Rightarrow t\circ\pt_J$ and $\mathscr X$ is constant if, and only if, $\alpha^*_{\mathscr X}\colon \mathscr X\to \pt_J^*\mathscr X_t$ is an isomorphism. By (Der.2) this is equivalent to ask that $j^*(\alpha^*_{\mathscr X})$ is an isomorphism for each $j\in J$. Furthermore,  $j^*(\alpha^*_{\mathscr X})={\mathscr X}_a\colon \mathscr X_j\to \mathscr X_t$, where $a\colon j\to t$ is the unique morphism in $J(j,t)$. As a consequence, whenever $J$ has a terminal object, a coherent diagram $\mathscr X\in \D(J)$ is constant if, and only if, $\mathscr X_a\colon \mathscr X_{j}\to \mathscr X_k$ is an isomorphism for each morphism $a\colon j\to k$ in $J$. This observation is important because it reduces the property of being constant for the coherent diagram $\mathscr X$, to a condition that we can check on the incoherent diagram $\dia_J(\mathscr X)$ (that is, that all its transitions maps are invertible).

\begin{definition}\label{constant_fiber_def}
Let $\D$ be a derivator and $\mathscr X\in \D(\Delta(I))$. We say that $\mathscr X$ is {\em constant on fibers} if, for any $i\in I$, $\iota_i^*{\mathscr X}$ is a constant diagram in $\D(u^{-1}(i))$ (where $\iota_i\colon u^{-1}(i)\to \Delta(I)$ is the functor introduced in \eqref{pb_defining_fibers}).
\end{definition}

We are now ready to state and prove the main result of this subsection:

\begin{proposition}\label{homotopical_epi}
In the above notation, $u\colon \Delta(I)\to I$ is a homotopical epimorphism. Furthermore, given a derivator $\D$, the essential image of $u^*\colon \D^I\to \D^{\Delta(I)}$ is the sub-derivator $\mathbb E\subseteq \D^{\Delta(I)}$, where $\mathbb E(J)$ is spanned by those ${\mathscr X}\in \D^{\Delta(I)}(J)$ that are constant on fibers, for each $J\in \cat$. 
\end{proposition}
\begin{proof}
Fix a derivator $\D$ and let us verify that the sub-prederivator $\mathbb E\subseteq \D^{\Delta (I)}$ defined in the statement satisfies the conditions of Lemma \ref{criterion_homo_epi}. Indeed, the first condition of the lemma can be verified as follows: let $J\in \cat$ and ${\mathscr X}\in \D^I(J)$, let us prove that $u^*{\mathscr X}\in \mathbb E(J)(\subseteq \D^{\Delta (I)}(J))$. Indeed, choose $i\in I$ and $a\in u^{-1}(i)$. 
Applying our derivator to the pullback diagram in \eqref{pb_defining_fibers}, we get the following commutative square
\[
\xymatrix@C=40pt{
\D({\Delta (I)})\ar[d]_{\iota^*_i}&\D(I)\ar[l]_{u^*}\ar[d]^{u(a)^*}\\
\D(u^{-1}(i))&\D(\bbone)\ar[l]^(.35){\pt^*_{u^{-1}(i)}}
}
\]
showing that $\iota_i^*(u^*{\mathscr X})= \pt^*_{u^{-1}(i)}u(a)^*{\mathscr X}= \pt^*_{u^{-1}(i)}(u^*{\mathscr X})_a$, as desired.
As for the second condition, choose $J\in \cat$ and ${\mathscr Y}\in \mathbb E(J)$, and let us study the component of the counit 
\[
\epsilon_{\mathscr Y}\colon u^*u_*{\mathscr Y}\longrightarrow {\mathscr Y}.
\] 
By (Der.2), isomorphisms in $\D({\Delta (I)})$ can be detected pointwise, so it is enough to show that, for each $a\in {\Delta (I)}$,  the following map  is an isomorphism
\[
a^*\epsilon_{\mathscr Y}\colon a^*u^*u_*{\mathscr Y}\longrightarrow {\mathscr Y}_a.
\] 
Let $a\in {\Delta (I)}$ and $i:=u(a)$; by (Der.4), there is an isomorphism 
\[
a^*u^*u_*{\mathscr Y}=i^*u_*{\mathscr Y}\overset{\cong}\longrightarrow  \holim_{i/u}\pr_i^*{\mathscr Y},
\] 
where $\pr_i\colon i/u\to {\Delta (I)}$ is the obvious projection. Hence, it suffices to prove that the canonical map $\holim_{i/u}\pr_i^*{\mathscr Y}\to {\mathscr Y}_a$ is an isomorphism (see also~\cite[Lem.\,8.7]{GS2017}).
Now let $\alpha_i\colon u^{-1}(i)\to i/u$ be the obvious inclusion; by Lemma \ref{basic_properties_delta_construction}, $\alpha_i$ is a left adjoint and so, by the dual of \cite[Prop.\,1.24]{Moritz}, the following canonical map is an isomorphism
\[
\holim_{i/u}\pr_i^*{\mathscr Y}\overset \cong \longrightarrow \holim_{u^{-1}(i)}\alpha_i^*\pr_i^*{\mathscr Y}=\holim_{u^{-1}(i)}\iota_i^*{\mathscr Y}\cong \holim_{u^{-1}(i)}\pt_{u^{-1}(i)}^*{\mathscr Y}_{a}
\]
where the last isomorphism holds since we have started with ${\mathscr Y}\in \mathbb E(J)$. To conclude recall that, by Lemma \ref{basic_properties_delta_construction}, there is a terminal object $t_i\colon \bbone \to u^{-1}(i)$, that is, $t_i$ is right adjoint to $\pt_{u^{-1}(i)}$. Hence, $\pt_{u^{-1}(i)}^*\cong (t_i)_*$, showing that 
\[
\holim_{u^{-1}(i)}\pt_{u^{-1}(i)}^*{\mathscr Y}_a\cong \holim_{u^{-1}(i)}(t_i)_*{\mathscr Y}_{a}\cong (\pt_{u^{-1}(i)}\circ t_i)_*{\mathscr Y}_{a}\cong {\mathscr Y}_{a}.\qedhere
\]
\end{proof}

\section{Lifting incoherent diagrams along diagram functors}\label{lift_sec}

Let $\D\colon \cat^{\op}\to \Cat$ be a strong and stable derivator. In this section we study the following problem: 
\begin{quotation}\noindent
Given an (incoherent) diagram $X\colon I\to \D(\bbone)$, under which conditions is it possible to lift it to a coherent diagram $\mathscr X$ of shape $I$? That is, can we find an object $\mathscr X\in \D(I)$ such that $\dia_I(\mathscr X)\cong X$?
\end{quotation}
The property of $\D$ being strong implies that we can always lift diagrams of shape $\bbtwo$. The main result of this section is the following theorem that gives sufficient conditions for a diagram of arbitrary shape to lift to a coherent diagram. These sufficient conditions consist in assuming that there are no ``negative extensions" in $\D(\bbone)$ between the components of our diagram $X\colon I\to \D(\bbone)$. Similar conditions are given to identify pairs of coherent diagrams $\mathscr X, \mathscr Y\in \D(I)$ for which $\dia_I(-)\colon \D(I)(\mathscr X, \mathscr Y)\to \D(\bbone)^I(\dia_I\mathscr X, \dia_I\mathscr Y)$ is bijective.

\begin{theorem}\label{main_lifting_theorem}
Given a strong and stable derivator $\D\colon \cat^{\op}\to \Cat$, the following statements hold true for any small category $I$:
\begin{enumerate}[\textup(1\textup)]
\renewcommand{\theenumi}{\textup{\arabic{enumi}}}
\item given ${\mathscr X},\ {\mathscr Y}\in \D(I)$, the canonical map $\D(I)({\mathscr X},{\mathscr Y})\to \D(\bbone)^I(\dia_I{\mathscr X},\dia_I{\mathscr Y})$ is an isomorphism provided $\D(\bbone)(\Sigma^n {\mathscr X}_i,{\mathscr Y}_j)=0$ for all $i,\, j\in I$ and $n>0$;
\item given $X\in \D(\bbone)^{I}$, there is an object ${\mathscr  X}\in \D(I)$ such that $\dia_I({\mathscr  X})\cong X$, provided $\D(\bbone)(\Sigma^n X_i,X_j)=0$ for all $i,\, j\in I$ and $n>0$.
\end{enumerate}
\end{theorem}

The proof of the above theorem is quite involved and it will occupy the rest of this section, which is divided in three subsections to reflect the main steps of the argument.

\subsection{Lifting diagrams of finite length}\label{finite_lift_subs}

In this subsection we are going to verify the statement of Theorem \ref{main_lifting_theorem} for diagrams of shape $I$, with $I$ a category of finite length. These results are very close to some of the main results in \cite{Porta}. We offer here a different and self-contained argument.

\begin{proposition}\label{prop_lift_I}
Let $\D\colon \cat^{\op}\to \Cat$ be a strong and stable derivator, $I$ a small category of finite length, and  ${\mathscr X},\ {\mathscr Y}\in \D(I)$. The canonical map 
\[
\dia_I(-)\colon\D(I)({\mathscr X},{\mathscr Y})\longrightarrow \D(\bbone)^I(\dia_I{\mathscr X},\dia_I{\mathscr Y}),
\] 
is an isomorphism provided $\D(\bbone)(\Sigma^n {\mathscr X}_i,{\mathscr Y}_j)=0$ for all $i,\, j\in I$ and $n>0$.
%
\end{proposition}
\begin{proof}
We proceed by induction on $\ell(I)$. If $\ell(I)=1$, then $I$ is a disjoint union of copies of $\bbone$  and, by (Der.1), there is nothing to prove. For $\ell(I)>1$, consider the setting of Notation \ref{notation_fl} and let $\mathscr X,\, \mathscr Y\in \D(I)$. By Lemmas \ref{split_induces_triangle_lemma} and \ref{lifting_inductive_step}, there is a triangle in $\D(I)$:
\[
\xymatrix{
\coprod_{I\setminus J}i_!i^*u_!u^*\mathscr X\ar[r]^-{\alpha}&\coprod_{I\setminus J}i_!i^*\mathscr X\sqcup u_!u^*\mathscr X\ar[r]^-{\beta}&\mathscr X\ar[r]&\Sigma\left(\coprod_{I\setminus J}i_!i^*u_!u^*\mathscr X\right).
}
\]
where the maps $\alpha$ and $\beta$ are described explicitly in Lemma \ref{lifting_inductive_step}. We are going now to apply the functor $(-,{\mathscr Y}):=\D(I)(-,{\mathscr Y})$ to the above triangle and study the resulting long exact sequence in $\Ab$. We start noting that, by Propositions \ref{morph_fin_length_prop_1} and \ref{morph_fin_length_prop_2}, for each $i\in I$, there is the following commutative diagram:
\[
\xymatrix@C=52pt@R=4pt{
(i_!i^*\mathscr X\sqcup u_!u^*\mathscr X,\mathscr Y)\ar@{=}[d]\\
(i_!i^*\mathscr X,\mathscr Y)\times (u_!u^*\mathscr X,\mathscr Y)\ar[rr]^-{\left[\begin{smallarray}{ccc}(-)\circ i_!i^*\epsilon_u&&
-(-)\circ \epsilon_i
\end{smallarray}\right]}%
\ar[dddddd]^-{\rotatebox[origin=c]{90}{$\sim$}}_{\left[\begin{smallarray}{cc}
i^*(-)\circ\eta_i&0\\
\\
0&u^*(-)\circ\eta_u\\
\end{smallarray}\right]}&&(i_!i^*u_!u^*\mathscr X,\mathscr Y)\ar[dddddd]^{(*)}_-{\rotatebox[origin=c]{90}{$\sim$}}\\
\\
\\
\\
\\
\\
\D(\bbone)(i^*\mathscr X,i^*\mathscr Y)\times \D(J)(u^*\mathscr X,u^*\mathscr Y)\ar[rr]_-{\left[\begin{smallarray}{ccc}
\pt^*_{u/i}(-)\circ\gamma^*&&
-\gamma^*\circ \pr^*_{i}(-)
\end{smallarray}\right]}&&\D(u/i)(\pr_i^*u^*(\mathscr X),\pt_{u/i}^*\mathscr Y_i)}
\]
where both columns are isomorphisms and the map $(*)$ is the composition of the three vertical maps on the right-hand-side of the diagrams in both Propositions   \ref{morph_fin_length_prop_1} and \ref{morph_fin_length_prop_2}. We also have the following series of isomorphisms: 
\begin{align*}
\left(\Sigma\left(\coprod{}_{I\setminus J}i_!i^*u_!u^*\mathscr X\right),\mathscr Y\right)&\cong \prod{}_{I\setminus J}\D(\bbone)(\Sigma i^*u_!u^*\mathscr X,\mathscr Y_i)\\
&\cong \prod{}_{I\setminus J}\D(\bbone)(\Sigma\hocolim_{u/i}\pr_i^*u^*\mathscr X,\mathscr Y_i)\\
&\cong \prod{}_{I\setminus J}\D(u/i)(\Sigma\pr_i^*u^*\mathscr X,\pt_{u/i}^*\mathscr Y_i),
\end{align*}
showing that $\left(\Sigma\left(\coprod{}_{I\setminus J}i_!i^*u_!u^*\mathscr X\right),\mathscr Y\right)=0$ by inductive hypothesis. Therefore, after all these observations, we have obtained the following commutative diagram in $\Ab$
\[
\xymatrix@C=10pt@R=14pt{
0\ar[d]&&\\
(\mathscr X,\mathscr Y)\ar[d]&&\\
{\begin{split}
\prod{}_{I\setminus J}\D(\bbone)&(i^*\mathscr X,i^*\mathscr Y)\\
&\times \D(J)(u^*\mathscr X,u^*\mathscr Y)\end{split}}\ar[dddd]|
-{
\footnotesize
\left[\begin{smallarray}{ccc|c}
\ddots &&{\text{\Large $0$}}&\vdots\\
&\pt_{u/i}^*(-)\circ\gamma^*&&\ -\gamma^*\circ \pr_i^*(-)\\
{\text{\Large $0$}}&&\ddots&\vdots\end{smallarray}\right]
}%
\ar[rr]^-{\dia}_-{\sim}&&
{\begin{split}\prod{}_{I\setminus J}&\D(\bbone)(i^*\mathscr X,i^*\mathscr Y)\\ 
&\times\D(\bbone)^J(\dia_{J}(u^*\mathscr X),\dia_J(u^*\mathscr Y))\end{split}} \ar[dddd]^{\Phi}\\
\\
\\
\\
\prod_{I\setminus J}\D(u/i)(\pr_i^*u^*(\mathscr X),\pt_{u/i}^*\mathscr Y_i)\ar[rr]^-{\dia}_-{\sim}&& \prod_{I\setminus J}\D(\bbone)^{u/i}(\dia_{u/i}(\pr_i^*u^*(\mathscr X)),\dia_{u/i}(\pt_{u/i}^*\mathscr Y_i))
}
\]
where the left column is exact, showing that $(\mathscr X,\mathscr Y)$ is the kernel of the following map in the diagram, while the rows consist in an application of the corresponding diagram functors and, therefore, they are isomorphisms by inductive hypothesis (as we are working with categories of shorter length than $\ell(I)$). To conclude, it is enough to verify that $\ker(\Phi)=\D(\bbone)^I(\dia_I\mathscr X,\dia_I\mathscr Y)$. Hence, take an element 
\[
\xymatrix{
f:=(f_k)_{k\in I}\in \prod_{I\setminus J}\D(\bbone)(i^*\mathscr X,i^*\mathscr Y)\times\D(\bbone)^J(\dia_{J}(u^*\mathscr X),\dia_J(u^*\mathscr Y))
}
\]
where $f_{\restriction J}:=(f_j)_{j\in J}$ satisfies the required compatibilities to be a morphism of diagrams $\dia_{J}(u^*\mathscr X)\to\dia_J(u^*\mathscr Y)$. By Propositions \ref{morph_fin_length_prop_1} and \ref{morph_fin_length_prop_2}, this $f$ is sent to
\[
\Phi(f)=((f_i\circ \mathscr X_a-\mathscr Y_a\circ f_j)_{(j,a)\in u/i})_{i\in I\setminus J}.
\]
That is, $\Phi((f_k)_{k\in I})=0$ if, and only if, for each $i\in I\setminus J$ and each $(j,a\colon j\to i)\in u/i$,  we have that $f_i\circ \mathscr X_a-\mathscr Y_a\circ f_j=0$, that is, each of the following diagrams commutes:
\[\xymatrix@R=25pt@C=50pt{
{\mathscr X}_j\ar[r]^{f_j}\ar[d]_{{\mathscr X}_a}&{\mathscr Y}_j\ar[d]^{{\mathscr Y}_a}\\
{\mathscr X}_i\ar[r]_{f_i}&{\mathscr Y}_i.
}\]
Equivalently, $(f_k)_{k\in I}$ represents a morphism in $\D(\bbone)^I$ from $\dia_I{\mathscr X}$ to $\dia_I{\mathscr Y}$.
\end{proof}

\begin{proposition}\label{prop_lift_II}
Let $\D\colon \cat^{\op}\to \Cat$ be a strong and stable derivator, $I$ a small category of finite length, and  $X\in \D(\bbone)^I$. Then, there is an object ${\mathscr  X}\in \D(I)$ such that $\dia_I({\mathscr  X})\cong X$, provided $\D(\bbone)(\Sigma^n X(i),X(j))=0$ for all $i,\, j\in I$ and $n>0$.
\end{proposition}
\begin{proof}
We proceed by induction on $\ell(I)$. If $\ell(I)=1$, then $I$ is a disjoint union of copies of $\bbone$  and, by (Der.1), there is nothing to prove. If $\ell(I)>1$, consider the setting of Notation~\ref{notation_fl}.
By inductive hypothesis we have  an object ${\mathscr X}_J\in \D(J)$ such that there exists an isomorphism
\[
\xymatrix{
\xi:=(\xi_j)_{j\in J}\colon \dia_J({\mathscr X}_J)\ar[r]^-{\sim}& X_{\restriction J}&&\text{in $\D(1)^J$.}
}
\] 
For each $i\in I$, let ${\mathscr X}_{u/i}:=\pr_i^*{\mathscr X}_J\in \D(u/i)$ and note  that 
\[
\xymatrix{
\dia_{u/i}({\mathscr X}_{u/i})=\dia_{J}({\mathscr X}_J)\circ \pr_i\ar[rrr]^-{\sim}_-{(\xi_{\pr_i(a)})_{a\in u/i}}&&& X_{\restriction J}\circ \pr_i.
}
\]
We also need to consider $\pt_{u/i}^*X(i)$, for which $\dia_{u/i}(\pt_{u/i}^*X(i))=\kappa_{u/i}(X(i))$ is constant, and $C_i:=\hocolim_{u/i}{\mathscr X}_{u/i}\in \D(\bbone)$, for which we have the usual isomorphism
\[
\xymatrix{
C_i=\hocolim_{u/i}{\mathscr X}_{u/i}\ar[rr]^-{\sim}_-{\gamma_!}&&i^*u_!{\mathscr X_J}.
}
\]
Consider now the following morphism in $\D(\bbone)^{u/i}$:
\[
\xymatrix{
(X_a)_{a\in u/i}\colon X_{\restriction J}\circ \pr_i\ar[rr]&&\kappa_{u/i}(X(i))
}
\]
where, given $a=(j,a\colon j\to i)\in u/i$, the morphism $X_a\colon X(j)\to X(i)$ is the obvious connecting map in $X\in \D(\bbone)^I$. By Proposition \ref{prop_lift_I}, the diagram functor $\dia_{u/i}$ induces the isomorphism $\D(u/i)({\mathscr X}_{u/i}, \pt_{u/i}^*X(i))\cong\D(\bbone)^{u/i}(\dia_{u/i}({\mathscr X}_{u/i}), \dia_{u/i}(\pt_{u/i}^*X(i)))$, hence there is a unique isomorphism
\[
\bar\varphi_i\colon{\mathscr X}_{u/i}\longrightarrow \pt_{u/i}^*X(i)
\]
such that $\dia_{u/i}(\bar\varphi_i)=(X_a\circ\xi_{\pr_i(a)})_{a\in u/i}$. Furthermore,  the adjunction $(\pt_{u/i})_!\dashv \pt_{u/i}^*$ and the axiom (Der.4) induce the following isomorphisms:
\begin{equation}\label{def_phi_i_eq}
\vcenter{
\xymatrix@C=22pt@R=2pt{
\D(\bbone)(i^*u_!{\mathscr X}_J,X(i))\ar[rr]^-{\sim}_-{(-)\circ\gamma_!}&&\D(\bbone)(C_i, X(i))\ar[rr]^-{\sim}_-{\pt_{u/i}^*(-)\circ\eta_{u/i}}&&\D(u/i)({\mathscr X}_{u/i},\pt_{u/i}^*X(i))\\
\varphi_i\ar@{|->}[rrrr]&&&& \bar \varphi_i
}
}
\end{equation}
so we can define $\varphi_i\colon i^*u_!{\mathscr X}_J\to X(i)$ as the unique map such that $\pt_{u/i}^*(\varphi_i\circ \gamma_!)\circ \eta_{u/i}=\bar\varphi_i$. Now define $\mathscr X$ as the cone in $\D(I)$ of the following map $\alpha$:
\[
\xymatrix@C=32pt@R=9pt{
\coprod_{I\setminus J}i_!i^*u_!{\mathscr X}_J\ar[rr]^-{
\footnotesize
\alpha:=\left[\begin{smallarray}{ccc}
\ddots &&{\text{\Large $0$}}\\
&i_!\varphi_i&\\
{\text{\Large $0$}}&&\ddots\\
\hline
\cdots&-\epsilon_i&\cdots
\end{smallarray}\right]
}%
&& \coprod_{I\setminus J}i_!X(i)      \sqcup    u_!{\mathscr X}_J   \ar[r]^-\beta
&{\mathscr X}\ar[r]&\Sigma{\left(\coprod_{I\setminus J}i_!i^*u_!{\mathscr X}_J\right)}.
}
\]
Note also that the above triangle is pointwise split-exact: it is enough to verify that $k^*(\alpha)$ is a split monomorphism in $\D(\bbone)$ for all $k\in I$, but this is trivially true for $k\in J$, as in this case $k^*(\coprod_{I\setminus J}i_!i^*u_!{\mathscr X}_J)=0$, while for $k=i\in I\setminus J$, so the morphism $i^*(\alpha)$ becomes:
\[
\xymatrix@C=32pt@R=9pt{
i^*i_!i^*u_!{\mathscr X}_J\ar[rr]^-{
\footnotesize
i^*(\alpha)=\left[\begin{smallarray}{ccc}
i^*i_!\varphi_i\\
 \\
\hline
\\
-i^*\epsilon_i
\end{smallarray}\right]
}%
&& i^*i_!X(i)      \sqcup    i^*u_!{\mathscr X}_J,  
}
\]
which is a split monomorphism because $-i^*\epsilon_i\colon i^*i_!i^*u_!{\mathscr X}_J\to i^*u_!{\mathscr X}_J  $ is an isomorphism. As a consequence, one sees that $\dia_{I}(\mathscr X)$ is the cokernel of the map $\dia_I(\alpha)$. Finally, define the following map of (incoherent) diagrams
\[
\varphi:=(\varphi_i)_{i\in I}\colon \dia_{I}(u_!\mathscr X_J)\longrightarrow X\qquad\text{in $\D(\bbone)^I$},
\] 
where $\varphi_i$ is the map defined in \eqref{def_phi_i_eq},
and consider the following commutative diagram
\[
\xymatrix@C=28pt@R=30pt{
\coprod_{I\setminus J}\dia_I(i_!i^*u_!{\mathscr X}_J)\ar[d]^-{\rotatebox[origin=c]{90}{$\sim$}}_{\chi}\ar[rr]^-{\dia_I(\alpha)}%
&& \coprod_{I\setminus J}\dia_I(i_!X(i))\sqcup   \dia_I( u_!{\mathscr X}_J)\ar[d]^-{\rotatebox[origin=c]{90}{$\sim$}}_{\chi}   \ar[rr]^-{\dia_I(\beta)}&&{\dia_I(\mathscr X)}
\\
\coprod_{I\setminus J}(i^*u_!{\mathscr X}_J)\otimes i\ar[rr]|-{
\footnotesize
\left[\begin{smallarray}{ccc}
\ddots &&{\text{\Large $0$}}\\
&\varphi_i\otimes i&\\
{\text{\Large $0$}}&&\ddots\\
\hline
\cdots&-e_i&\cdots
\end{smallarray}\right]
}%
&& \coprod_{I\setminus J}X(i)\otimes i\sqcup   \dia_I( u_!{\mathscr X}_J)   \ar[rr]_-{
\footnotesize
\left[\begin{array}{ccc|c}
\cdots &e_i&\cdots&\varphi
\end{array}\right]
}&&{X}}
\]
where the vertical maps are induced by the natural isomorphisms $\chi\colon \dia_I\circ i_!\tilde{\Longrightarrow}-\otimes i$ (with $i$ varying in $I\setminus J$) and $e_i$ is the counit of the adjunction $(-\otimes i)\dashv (-)_{\restriction i}$ (see Example~\ref{coh_incoh_rest_ext_example}). One now checks easily that the second row is point-wise split, so $X$ is a cokernel of the first map, which is isomorphic to $\dia_I(\alpha)$. Therefore, $\dia_I(\mathscr X)\cong X$, as desired.
\end{proof}

\subsection{Lifting diagrams of countable length}\label{countable_lift_subs}

\begin{proposition}\label{cor_lift_countable_I}
Let $\D\colon \cat^{\op}\to \Cat$ be a strong and stable derivator, $I$ a small category of countable length and ${\mathscr X},\, \ {\mathscr Y}\in \D(I)$. Then, the canonical map 
\[
\dia_I(-)\colon \D(I)({\mathscr X},{\mathscr Y})\longrightarrow \D(\bbone)^I(\dia_I({\mathscr X}),\dia_I({\mathscr Y}))
\] 
is an isomorphism provided $\D(\bbone)(\Sigma^n {\mathscr X}_i,{\mathscr Y}_j)=0$ for all $i,\, j\in I$ and $n>0$.
\end{proposition}
\begin{proof}
Fix Notation \ref{notation_countable_cat} and consider the following triangle given by the Lemmas \ref{split_induces_triangle_lemma} and~\ref{dist_triangle_countable_length}, where $\alpha$ and $\beta$ are the explicit matrices given in the lemma:
\[
\xymatrix@C=20pt{
\coprod_{\N}(\iota_n)_!\iota_n^*({\mathscr X})\ar[rr]^{\alpha}&&\coprod_{\N}(\iota_n)_!\iota_n^*({\mathscr X})\ar[rr]^-{\beta}&&{\mathscr X}\ar[rr]&&\Sigma\left(\coprod_{\N}(\iota_n)_!\iota_n^*({\mathscr X})\right).
}
\] 
Apply $(-,{\mathscr Y}):=\D(I)(-,{\mathscr Y})$ to the above triangle to get the following long exact sequence:
\[
\xymatrix@C=8.5pt{
\cdots \ar[r]&0\ar[r]&({\mathscr X},{\mathscr Y})\ar[r]&(\coprod_\N (\iota_n)_!\iota_n^*({\mathscr X}),{\mathscr Y})\ar[rrr]^{(-)\circ\alpha}&&&(\coprod_{\N}(\iota_n)_!\iota_n^*({\mathscr X}),{\mathscr Y})\ar[r]&\cdots
}
\]
where we can write the $0$ on the left by the following series of isomorphisms: 
\begin{align*}
\textstyle
\left(\Sigma(\coprod_{\N}(\iota_n)_!\iota_n^*({\mathscr X})),{\mathscr Y}\right)&\cong \textstyle\prod_{\N}\left((\iota_n)_!\Sigma\iota_n^*({\mathscr X}),{\mathscr Y}\right)\\
&\cong \textstyle\prod_{\N}\D(I_{\leq n})\left(\Sigma\iota_n^*({\mathscr X}),\iota_n^*({\mathscr Y})\right)\\
&\cong \textstyle\prod_{\N}\D(\bbone)^{I_{\leq n}}\left(\dia_{I_{\leq n}}(\Sigma\iota_n^*({\mathscr X})),\dia_{I_{\leq n}}(\iota_n^*({\mathscr Y}))\right)=0.
\end{align*}
where the first isomorphism holds because $\Sigma$ is an equivalence (and, as such, it commutes with coproducts) and $(\iota_n)_!$ is triangulated, the second isomorphism follows by the adjunction $(\iota_n)_!\dashv \iota_n^*$, the third one is a consequence of Proposition~\ref{prop_lift_I} (as $I_{\leq n}$ is of finite length) and the last equality to $0$ follows by our orthogonality hypotheses. A consequence of the  long exact sequence above is that $(\mathscr X,\mathscr Y)$ is the kernel of the map $(-)\circ \alpha$ and, therefore, we will have concluded if we could prove that also $\D(\bbone)^I(\dia_I(\mathscr X),\dia_I(\mathscr Y))$ is a kernel for this map. We start considering the following commutative diagram:
\[
\xymatrix@C=48pt@R=30pt{
\D(I)((\iota_n)_!\iota_n^*({\mathscr X}),{\mathscr Y})\ar@{<-}[r]^-{(-)\circ\partial^n}\ar[d]^-{\rotatebox[origin=c]{90}{$\sim$}}_-{\iota_n^*(-)\circ\eta_n}&\D(I)((\iota_{n+1})_!\iota_{n+1}^*({\mathscr X}),{\mathscr Y})\ar[d]_-{\rotatebox[origin=c]{90}{$\sim$}}^-{\iota_{n+1}^*(-)\circ\eta_{n+1}}\\
\D(I_{\leq n})(\iota_n^*({\mathscr X}),\iota_n^*({\mathscr Y}))\ar@{<-}[r]^{\iota_{n+1,n}^*(-)}\ar[d]^-{\rotatebox[origin=c]{90}{$\sim$}}_-{\dia_{I_{\leq n}}(-)}&\D(I_{\leq n+1})(\iota_{n+1}^*({\mathscr X}),\iota_{n+1}^*({\mathscr Y}))\ar[d]_-{\rotatebox[origin=c]{90}{$\sim$}}^-{\dia_{I_{\leq n+1}}(-)}\\
\D(\bbone)^{I_{\leq n}}(\dia_I({\mathscr X})_{\restriction I_{\leq n}},\dia_I({\mathscr Y})_{\restriction I_{\leq n}})\ar@{<-}[r]^{(-)_{\restriction I_{\leq n}}}&\D(\bbone)^{I_{\leq n+1}}(\dia_I({\mathscr X})_{\restriction I_{\leq n+1}},\dia_I({\mathscr Y})_{\restriction I_{\leq n+1}})
}
\]
where the upper square commutes by Proposition \ref{morph_count_length_prop}, and the lower vertical maps are isomorphisms because of Proposition \ref{prop_lift_I} (and the fact that $I_{\leq n}$ has finite length, for all \mbox{$n\in \N$}).
%
%
Hence, identifying $\D(I_{\leq n})(\coprod_{\N}(\iota_n)_!\iota_n^*({\mathscr X}),{\mathscr Y})$ with $\prod_{\N}\D(I_{\leq n})((\iota_n)_!\iota_n^*({\mathscr X}),{\mathscr Y})$, and also identifying each $\D(I_{\leq n})((\iota_n)_!\iota_n^*({\mathscr X}),{\mathscr Y})$ with $\D(\bbone)^{I_{\leq n}}(\dia_I({\mathscr X})_{\restriction I_{\leq n}},\dia_I({\mathscr Y})_{\restriction I_{\leq n}})$ via the isomorphism $\dia_{I_{\leq n}}(\iota_n^*(-)\circ \eta_n)$ (as in the columns of the above diagram), we can describe the action of $(-)\circ \alpha$ as follows:
\begin{align*}
\textstyle\prod_n\D(\bbone)^{I_{\leq n}}(\dia_I({\mathscr X})_{\restriction{I_{\leq n}}},\dia_I({\mathscr Y})_{\restriction{I_{\leq n}}})&\textstyle\to \prod_n\D(\bbone)^{I_{\leq n}}(\dia_I({\mathscr X})_{\restriction{I_{\leq n}}},\dia_I({\mathscr Y})_{\restriction{I_{\leq n}}})
\\
(\phi_n)_{n\in \N}&\mapsto (\phi_n-(\phi_{n+1})_{\restriction{I_{\leq n}}})_{n\in\N}.
\end{align*}
Thus, a sequence $\phi:=(\phi_n)_{n\in\N}\in \prod_n\D(\bbone)^{I_{\leq n}}(\dia_I({\mathscr X})_{\restriction{I_{\leq n}}},\dia_I({\mathscr Y})_{\restriction{I_{\leq n}}})$ is in the kernel of $(-)\circ \alpha$ if, and only if, $(\phi_{n+1})_{\restriction{I_{\leq n}}}=\phi_n$ for all $n\in\N$, that is if, and only if, $\phi$ represents (the sequence of successive truncations of) a morphism $\dia_I({\mathscr X})\to \dia_I({\mathscr Y})$. 
\end{proof}

\begin{proposition}\label{cor_lift_countable_II}
Let $\D\colon \cat^{\op}\to \Cat$ be a strong and stable derivator, $I$ a small category of countable length and $X\in \D(\bbone)^{I}$. Then, there is an object ${\mathscr  X}\in \D(I)$ such that $\dia_I({\mathscr  X})\cong X$, provided $\D(\bbone)(\Sigma^n X_i,X_j)=0$ for all $i,\, j\in I$ and $n>0$.
\end{proposition}
\begin{proof}
For each $n\in\N$, the inclusions $\iota_n\colon I_{\leq n}\to I$ and $\iota_{n+1,n}\colon I_{\leq n}\to I_{\leq n+1}$ are cosieves and, therefore, the associated restrictions 
\[
(-)_{\restriction I_{\leq n}}\colon \D(\bbone)^I\longrightarrow \D(\bbone)^{I_{\leq n}}\qquad\text{and}\qquad(-)_{\restriction I_{\leq n+1}}\colon \D(\bbone)^I\longrightarrow \D(\bbone)^{I_{\leq n+1}}
\] 
 have left adjoints that act as extension by zero functors (recall Remark~\ref{rem.cosieve}). We denote these adjoints by
 \[
(-)\otimes \iota_n\colon \D(\bbone)^{I_{\leq n}}\longrightarrow \D(\bbone)^I\qquad\text{and}\qquad(-)\otimes \iota_{n+1,n}\colon \D(\bbone)^{I_{\leq n}}\longrightarrow \D(\bbone)^{I_{\leq n+1}},
\]
respectively. We denote the counits of the above adjunctions by 
\[
e^n=(e^n_i)_I\colon(-)_{\restriction I_{\leq n}}\otimes \iota_n\Rightarrow \id_{\D(\bbone)^I}\quad\text{and}\quad d^n=(d^n_i)_{I_{\leq n+1}}\colon(-)_{\restriction I_{\leq n}}\otimes \iota_{n+1,n}\Rightarrow \id_{\D(\bbone)^{I_{\leq n+1}}}
\]
where $e^n_i$ and $d^n_i$ are identities if $i\in I_{\leq n}$ and $0$ otherwise.

Consider now the diagram $X\in \D(\bbone)^I$ of the statement, let $X^n:=(X_{\restriction I_{\leq n}})\otimes \iota_n$, for all $n\in \N$, and consider the counit 
\[
d^n\colon (X_{\restriction I_{\leq n}})\otimes \iota_{n+1,n}=(X_{\restriction I_{\leq n+1}})_{\restriction I_{\leq n}}\otimes \iota_{n+1,n}\longrightarrow X_{\restriction I_{\leq n+1}}.
\]
Extending this counit we get a morphism $d^n\otimes\iota_{n+1}\colon X^n\to X^{n+1}$ (whose components relative to $I_{\leq n}$ are  identities and the remaining components are trivial). Then there is a pointwise split-exact sequence in $\D(\bbone)^I$ of the form 
\[
\xymatrix@C=35pt{
\coprod_{\N}X^n\ar[rrr]^(.47){
\footnotesize
\bar\alpha:=
\left[\begin{smallarray}{cccc}
1&&&\\
-d^0\otimes{\iota_1}&1&&{\text{\Large $0$}}\\
&-d^1\otimes{\iota_2}&1&\\
{\text{\Large $0$}}&&\ddots&\ddots
\end{smallarray}\right]
}&&&\coprod_{\N}X^n\ar[rrr]^-{
\footnotesize
\bar\beta:=
\left[\begin{smallarray}{cccc}
e^0 &e^1 &e^2 &\cdots
\end{smallarray}\right]
}&&&X.
}
\] 
By Proposition \ref{prop_lift_II}, for any $n\in\N_{>0}$ we can find an object ${\mathscr X^n}\in\D(I_{\leq n})$ such that $\dia_{I_{\leq n}}({\mathscr X^n})\cong X_{\restriction I_{\leq n}}$. Furthermore, by Proposition \ref{cor_lift_countable_I}, for each $n\in \N$ there is a unique morphism $\delta^n\colon(\iota_{n+1,n})_!\mathscr X^n\to \mathscr X^{n+1}$ such that $\dia_{I_{\leq n+1}}(\delta^n)=d^n$. Consider also a natural isomorphism $\xi^n\colon (\iota_n)_!\Rightarrow (\iota_{n+1})_!\circ (\iota_{n+1,n})_!$  (which exists because these functors are both left adjoints to $\iota_n^*=\iota_{n+1,n}^*\circ \iota_{n}^*$) and construct an object $\mathscr X\in \D(I)$ as the cone of the map $\alpha$ in the following triangle:
\[
\xymatrix@C=14pt{
\coprod_{\N}(\iota_n)_!(\mathscr X^n)\ar[rrrrrrrr]^(.46){
\footnotesize
\alpha:=
\left[\begin{smallarray}{ccccc}
1&&&\\
-(\iota_{1})_!\delta^0\circ\xi^0&1&&{\text{\Large $0$}}\\
&-(\iota_{2})_!\delta^1\circ\xi^1&&1&\\
{\text{\Large $0$}}&&\ddots&&\ddots
\end{smallarray}\right]
}&&&&&&&&\coprod_{\N}(\iota_n)_!(\mathscr X^n)\ar[r]&\mathscr X\ar[r]&\left(\coprod_{\N}(\iota_n)_!(\mathscr X^n)\right).
}
\]
Just by construction, $\dia_I(\alpha)$ is isomorphic to $\bar\alpha$, so the above triangle is pointwise split-exact and $\dia_I(\mathscr X)$ is the cokernel of  $\dia_I(\alpha)$. Since $X$ is the cokernel of $\bar \alpha$, we deduce that $\dia_I(\mathscr X)\cong X$.
\end{proof}

\subsection{Lifting in general}\label{general_lifting_subs}

We finally have all the ingredients to complete the proof of the main result of this section:

\begin{proof}[Proof of Theorem \ref{main_lifting_theorem}]
As in Subsection \ref{homotopical_epi_subs}, consider the category of countable length $\Delta(I)$ and the homotopical epimorphism $u\colon \Delta(I)\to I$. By the results of Subsection \ref{countable_lift_subs}, we know that the statements (1) and (2) hold in $\D(\Delta(I))$ and we should try to restrict them along the fully faithful functor $u^*\colon \D(I)\to \D(\Delta(I))$.  

\smallskip\noindent
(1) Let ${\mathscr X}\!,\, {\mathscr Y}\in \D(I)$ and consider the following commutative diagram:
\[
\xymatrix{
\D(I)({\mathscr X},{\mathscr Y})\ar[d]|\cong\ar[rr]&&\D(\bbone)^I(\dia_I{\mathscr X},\dia_I{\mathscr Y})\ar[d]|{(*)}\\
\D(\Delta(I))(u^*{\mathscr X},u^*{\mathscr Y})\ar[rr]|(.4){(**)}&&\D(\bbone)^{\Delta(I)}((\dia_I{\mathscr X})\circ u,(\dia_I{\mathscr Y})\circ u)
}
\]
where the leftmost vertical map is an isomorphism since $u^*$ is fully faithful, $(**)$ is an isomorphism by Proposition \ref{cor_lift_countable_I} and the fact that $\Delta(I)$ has countable length, and $(*)$ is an isomorphism by Lemma \ref{homo_epi_implies_lax_epi_lem}. 

\smallskip\noindent
(2) Given $X$ as in the statement, by Proposition \ref{cor_lift_countable_II} there is ${\mathscr X}_{\Delta(I)}$ in $\D(\Delta(I))$ such that $\dia_{\Delta(I)}({\mathscr X}_{\Delta(I)})\cong X\circ u$. Now one should just prove that ${\mathscr X}_{\Delta(I)}$ does belong in the essential image of $u^*$, but this is true since being constant on fibers is a property that just depends on the underlying (incoherent) diagrams (see the discussion before Definition \ref{constant_fiber_def}). 
\end{proof}

\section{\texorpdfstring{$t$}{t}-Structures on strong and stable derivators}\label{sec_tstr}

Let us start by fixing a strong and stable derivator
\[
\D\colon \cat^{\op}\to \Cat.
\]
By definition, a $t$-structure in $\D$ is just a $t$-structure in the base of $\D$, so let us also fix such a $t$-structure $\t=(\U,\Sigma \V)$ on $\D(\bbone)$ whose heart we denote by $\H$.  Furthermore, for each small category $I$, we let
\[\begin{split}
&\U_I:=\{\mathscr X\in \D(I):\mathscr X_i\in \U,\ \forall i\in I\}\\
&\V_I:=\{\mathscr X\in \D(I):\mathscr X_i\in \V,\ \forall i\in I\}.
\end{split}\]
The goal of this section is to prove Theorem~\ref{mainthm.lift_tstructure_general_thm}, i.e., we will verify that each $\t_I:=(\U_I,\Sigma\V_I)$ is a $t$-structure in $\D(I)$ and, moreover, that the heart $\H_I$ of $\t_I$ is equivalent to the functor category $\H^I$. Our proof will rely heavily on the results of Section \ref{lift_sec}; also, the scheme of the proof will be the same: we will verify our statement for categories of finite length, then for categories of countable length and, finally, in full generality.

\smallskip
We start with the following definition:

\begin{definition} \label{def.class closed under hocolim}
A subcategory $\mathcal{C}$ of $\mathbb{D}(\bbone)$ is said to be \emph{closed under taking homotopy colimits} (resp., {\em directed homotopy colimits}) with respect to $\mathbb{D}$, when for any small category (resp., directed set) $I$  and any object $\mathscr X\in\mathbb{D}(I)$, one has that $\hocolim_I \mathscr X\in\mathcal{C}$ whenever $\mathscr X_i\in\mathcal{C}$ for all $i\in I$. Closedness under (inverse) homotopy limits is defined dually.
\end{definition}

In fact, when the class $\C$ in the above definition is either an aisle or a co-aisle, then some closure properties are automatic: 

\begin{proposition} \label{prop.aisles closed for hocolim}
In the above setting, the aisle $\mathcal{U}$ is closed under taking homotopy colimits and the co-aisle $\mathcal{V}$ is closed under taking homotopy limits. 
\end{proposition}
\begin{proof}
Using the terminology of \cite{Ponto-Shulman}, it is clear that $\mathcal{U}$ is closed under coproducts and it follows from \cite[Thm.\,6.1]{GPS-MayerVietoris} or \cite[Prop.\,2.4]{stovivcek2016compactly} that it is closed under homotopy pushouts. Then,  by \cite[Thm.\,7.13]{Ponto-Shulman}, we conclude that $\mathcal{U}$ is closed under homotopy colimits. The argument for the co-aisle is formally dual.
\end{proof}

Now we can prove the lifting property of $t$-structures for categories of finite length.

\begin{lemma}\label{lifting_tstructure_finite_lem}
In the above setting, $\t_I=(\U_I,\Sigma \V_I)$ is a $t$-structure in $\D(I)$, for any small category of finite length $I$.
\end{lemma}
\begin{proof}
We proceed by induction on $\ell(I)$. If $\ell(I)=1$, then $I$ is a disjoint union of copies of $\bbone$  and, by (Der.1), there is nothing to prove. Hence, let $I$ be a small category of finite length $\ell(I)>1$, fix Notation \ref{notation_fl}, and let us verify that $\t_I$ is a $t$-structure. Note that, by Theorem \ref{main_lifting_theorem}, the axiom ($t$-S.1) holds, while the axiom ($t$-S.2) holds  by construction. Hence, we have just to verify ($t$-S.3).
Indeed, take $\mathscr Y\in \D(I)$ and let us construct a suitable truncation triangle with respect to $\t_I$. By inductive hypothesis, we can consider the following truncation triangles in $\D(J)$ and $\D(\bbone)$, respectively, 
\[
\xymatrix@C=28pt{
{\mathscr U}_{u^*{\mathscr Y}}\ar[r]^-{\varphi_u}& u^*{\mathscr Y}\ar[r]^-{\psi_u}& {\mathscr V}_{u^*{\mathscr Y}}\ar[r]& \Sigma {\mathscr U}_{u^*{\mathscr Y}}&
{ U}_{{\mathscr Y}_i}\ar[r]^-{\varphi_i}& {\mathscr Y}_i\ar[r]^-{\psi_i}& { V}_{{\mathscr Y}_i}\ar[r]& \Sigma { U}_{{\mathscr Y}_i}
}
\]
with $i$ ranging in $I\setminus J$. Furthermore, since $\varphi_i\colon { U}_{{\mathscr Y}_i}\to {\mathscr Y}_i$ is a coreflection of ${\mathscr Y}_i$ onto $\U$, and $i^*u_!\mathscr{U}_{u^*\mathscr{Y}}\cong \hocolim_{u/i}\pr_i^*\mathscr{U}_{u^*\mathscr{Y}}\in \U$ (by Proposition \ref{prop.aisles closed for hocolim}), there exists a unique morphism $e_i\colon i^*u_!\mathscr{U}_{u^*\mathscr{Y}}\to U_{\mathscr{Y}_i}$ that makes the following square commute:
\[
\xymatrix@C=70pt{
i^*u_!\mathscr{U}_{u^*\mathscr{Y}}\ar@{.>}[r]^-{\exists!\,e_i}\ar[d]_-{i^*u_!\varphi_u}&U_{\mathscr{Y}_i}\ar[d]^-{\varphi_i}\\
i^*u_!u^*\mathscr{Y}\ar[r]_-{i^*\epsilon_u}&\mathscr{Y}_i.
}
\]
Consider now the following commutative square in $\D(I)$:
\[
\xymatrix@C=50pt@R=75pt{
\coprod_{I\setminus J}i_!i^*u_!{\mathscr U}_{u^*{\mathscr Y}}\ar[rr]|-{
\footnotesize
\,\,\alpha_{\mathscr U}:=\left[\begin{smallarray}{ccc}
\ddots &&{\text{\Large $0$}}\\
&i_!e_i&\\
{\text{\Large $0$}}&&\ddots\\
\hline
\cdots&-\epsilon_i&\cdots
\end{smallarray}\right]
}\ar[d]|-{
\footnotesize
\Phi^{(1)}:=\left[\begin{smallarray}{ccc}
\ddots &&{\text{\Large $0$}}\\
&i_!i^*u_!\varphi_u&\\
{\text{\Large $0$}}&&\ddots
\end{smallarray}\right]
}&&\coprod_{I\setminus J}i_!U_{{\mathscr Y}_i}\sqcup u_!{\mathscr U}_{u^*{\mathscr Y}}\ar[d]|-{
\footnotesize
\ \ \ \ \left[\begin{smallarray}{ccc|c}
\ddots &&{\text{\Large $0$}}&\vdots\\
&i_!\varphi_i&&0\\
{\text{\Large $0$}}&&\ddots&\vdots\\
\hline
\cdots&0&\cdots&u_!\varphi_u
\end{smallarray}\right]=:\Phi^{(2)}
}\\ 
\coprod_{I\setminus J}i_!i^*u_!u^*{\mathscr Y}\ar[rr]|-{
\footnotesize
\, \, \alpha:=\left[\begin{smallarray}{ccc}
\ddots &&{\text{\Large $0$}}\\
&i_!i^*\epsilon_u&\\
{\text{\Large $0$}}&&\ddots\\
\hline
\cdots&-\epsilon_i&\cdots
\end{smallarray}\right]
}&&\coprod_{I\setminus J}i_!{\mathscr Y}_i\sqcup u_!u^*{\mathscr Y}
}
\]
where the cone of $\alpha$ is just $\mathscr Y$, by Lemma \ref{lifting_inductive_step}. 
By the $3\times 3$ Lemma in triangulated categories we can complete the above square to a commutative diagram in $\D(I)$, where all the rows and columns are distinguished triangles
\[
\xymatrix@R=15pt{
\coprod_{I\setminus J}i_!i^*u_!{\mathscr U}_{u^*{\mathscr Y}}\ar[rr]^-{\alpha_{\mathscr U}}\ar[dd]_-{\Phi^{(1)}}&&\coprod_{I\setminus J}i_!U_{{\mathscr Y}_i}\sqcup u_!{\mathscr U}_{u^*{\mathscr Y}}\ar[dd]^{\Phi^{(2)}}\ar[rr]^-{\beta_{\mathscr U}}&&{\mathscr U}\ar[dd]\ar[r]^(.65)+&\\ 
\\
\coprod_{I\setminus J}i_!i^*u_!u^*{\mathscr Y}\ar[rr]^-{\alpha}\ar[dd]_-{\Psi^{(1)}}&&\coprod_{I\setminus J}i_!{\mathscr Y}_i\sqcup u_!u^*{\mathscr Y}\ar[rr]^-{\beta}\ar[dd]^-{\Psi^{(2)}}&&{\mathscr Y}\ar[dd]\ar[r]^(.65)+&\\
\\
{\mathscr V}^{(1)}\ar[d]^(.65)+\ar[rr]^-{\alpha_{\mathscr V}}&&{\mathscr V}^{(2)}\ar[d]^(.65)+\ar[rr]^-{\beta_{\mathscr V}}&&{\mathscr V}\ar[d]^(.65)+\ar[r]^(.65)+&\\
&&&&&
}
\]
and $\beta={
\footnotesize
\left[\begin{array}{ccc|c}
\cdots &\epsilon_i&\cdots&\epsilon_u
\end{array}\right]
}$.
We will have concluded if we can prove that ${\mathscr U}\in \U_I$ and ${\mathscr V}\in \V_I$, that is, $k^*{\mathscr U}\in \U$ and $k^*{\mathscr V}\in \V$ for any $k\in I$. We start from the case when $k\in J$, and we apply $k^*$  to the above $3\times 3$ diagram, obtaining the following commutative diagram in $\D(\bbone)$, where all the rows and columns are distinguished triangles:
\[
\xymatrix@C=30pt@R=10pt{
0\ar[rr]\ar[dd]&&0\sqcup k^*u_!{\mathscr U}_{u^*{\mathscr Y}}\ar[dd]\ar[rr]^-{\sim}&&{\mathscr U}_k\ar[dd]\ar[r]^(.65)+&\\ 
\\
0\ar[rr]\ar[dd]&&0\sqcup k^*u_!u^*{\mathscr Y}\ar[rr]^-{\sim}\ar[dd]&&{\mathscr Y}_k\ar[dd]\ar[r]^(.65)+&\\
\\
({\mathscr V}^{(1)})_k\ar[d]^(.65)+\ar[rr]&&({\mathscr V}^{(2)})_k\ar[d]^(.65)+\ar[rr]^-{\sim}&&{\mathscr V}_k\ar[d]^(.65)+\ar[r]^(.65)+&\\
&&&&&
}
\]
Since $k\in J$, the category $u/k$ has a terminal object and, therefore, using (Der.4) and \cite[Lem.\,1.19]{Moritz}, one shows that there is a natural isomorphism $k^*u_!\cong k^*$. Using this and the fact that $k^*$ sends triangles to triangles, we can observe that the central column is a truncation triangle of ${\mathscr Y}_k\cong k^*u_!u^*{\mathscr Y}$ with respect to the $t$-structure $(\mathcal{U}, \Sigma\mathcal{V})$ on $\D(\bbone)$. Since also $k^*{\mathscr V}^{(1)}=0$ (see the first column), we have ${\mathscr U}_k\cong k^*u_!{\mathscr U}_{u^*{\mathscr Y}}\cong U_{{\mathscr Y}_k}\in \U$ and ${\mathscr V}_k\cong {\mathscr V}^{(2)}_k\cong V_{{\mathscr Y}_k}\in \V$.

On the other hand, if $k\in I\setminus J$, applying $k^*$ to the above diagram we get the following commutative diagram in $\D(\bbone)$, where rows and columns are distinguished triangles:
\[
\xymatrix@R=10pt{
k^*k_!k^*u_!{\mathscr U}_{u^*{\mathscr Y}}\ar[rr]\ar[dd]&&k^*k_!U_{{\mathscr Y}_k}\sqcup k^*u_!{\mathscr U}_{u^*{\mathscr Y}}\ar[dd]\ar[rr]&&{\mathscr U}_k\ar[dd]\ar[r]^(.65)+&\\ 
\\
k^*k_!k^*u_!u^*{\mathscr Y}\ar[rr]\ar[dd]&&k^*k_!{\mathscr Y}_k\sqcup k^*u_!u^*{\mathscr Y}\ar[rr]\ar[dd]&&{\mathscr Y}_k\ar[dd]\ar[r]^(.65)+&\\
\\
({\mathscr V}^{(1)})_k\ar[d]^(.65)+\ar[rr]&&({\mathscr V}^{(2)})_k\ar[d]^(.65)+\ar[rr]&&{\mathscr V}_k\ar[d]^(.65)+\ar[r]^(.65)+&\\
&&&&&
}
\]
The components $k^*k_!k^*u_!{\mathscr U}_{u^*{\mathscr Y}}\to k^*u_!{\mathscr U}_{u^*{\mathscr Y}}$ and $k^*k_!k^*u_!u^*{\mathscr Y}\to k^*u_!u^*{\mathscr Y}$ are isomorphisms by Lemma \ref{easy_relations_for_u_and_k} (since $k\colon \bbone \to I$ is fully faithful), so the  first maps in each of the first two rows is a split monomorphism. Hence, the first two rows are split triangles and, therefore, the map ${\mathscr U}_k\to {\mathscr Y}_k$ is isomorphic to $k^*k_!U_{{\mathscr Y}_k}\to k^*k_!{\mathscr Y}_k$. Furthermore, by Example \ref{coh_incoh_rest_ext_example} (and the fact that $|I(k,k)|=1$), there is a natural isomorphism $k^*k_!\Rightarrow \id_{\D(\bbone)}$, therefore $k^*k_!U_{{\mathscr Y}_k}\cong U_{{\mathscr Y}_k}\in \U$ and $k^*k_!V_{{\mathscr Y}_k}\cong V_{{\mathscr Y}_k}\in \V$, showing that $k^*k_!U_{{\mathscr Y}_k}\to k^*k_!{\mathscr Y}_k$ (and, therefore, also ${\mathscr U}_k\to {\mathscr Y}_k$) is a coreflection onto $\U$. Hence, ${\mathscr V}_k\in \V$ as desired.
\end{proof}

\begin{lemma}\label{lift_tstructure_countable_lem}
In the same setting as above, $\t_I=(\U_I,\Sigma \V_I)$ is a $t$-structure in $\D(I)$, for any small category of countable length $I$.
\end{lemma}
\begin{proof}
As in the proof of Lemma \ref{lifting_tstructure_finite_lem}, we have just to verify ($t$-S.3).
Indeed, take $\mathscr X\in \D(I)$ and let us construct a suitable truncation triangle with respect to $\t_I$. Fix Notation \ref{notation_countable_cat}, then,
by Lemma \ref{lifting_tstructure_finite_lem}, for any $n\in\N_{>0}$, there is a triangle
\[
\xymatrix{
{\mathscr U}^n\ar[r]^{\varphi_n}&\iota_n^*{\mathscr X}\ar[r]& {\mathscr V}^n\ar[r]& \Sigma{\mathscr  U}^n
}
\]
with ${\mathscr U}^n\in\U_{I_{\leq n}}$ and ${\mathscr V}^n\in \V_{I_{\leq n}}$ 
Applying $(\iota_n)_!$ we obtain a triangle in $\D(I)$:
\begin{equation}\label{trun_tria_eq_lem_count}
\xymatrix@C=45pt{
(\iota_n)_!{\mathscr U}^n\ar[r]^-{(\iota_n)_!(\varphi_n)}& (\iota_n)_!\iota_n^*{\mathscr X}\ar[r]& (\iota_n)_!{\mathscr V}^n\ar[r]& \Sigma (\iota_n)_!{\mathscr U}^n.
}
\end{equation}
Since $(\iota_n)_!$ is an extension by $0$, we have that $(\iota_n)_!{\mathscr U}^n\in \U_I$ and $(\iota_n)_!{\mathscr V}^n\in \V_I$, so  the triangle in \eqref{trun_tria_eq_lem_count} is a truncation triangle with respect to $\t_I$ and, as a consequence, $(\iota_n)_!(\varphi_n)$ is a coreflection of $ (\iota_n)_!\iota_n^*{\mathscr X}$ onto $\U_I$. Consider the following solid diagram:
\[
\xymatrix@C=50pt@R=35pt{
(\iota_1)_!{\mathscr U}^1\ar[d]|{(\iota_1)_!(\varphi_1)}\ar@{.>}[r]^{\exists!\ \partial^1_{\mathscr U}}&(\iota_2)_!{\mathscr U}^2\ar[d]|{(\iota_2)_!(\varphi_2)}\ar@{.>}[r]^{\exists!\ \partial^2_{\mathscr U}}&\ldots\ar@{.>}[r]&(\iota_n)_!{\mathscr U}^n\ar[d]|{(\iota_n)_!(\varphi_n)}\ar@{.>}[r]^{\exists!\ \partial^n_{\mathscr U}}&\ldots\\
(\iota_1)_!\iota_1^*{\mathscr X}\ar[r]^{\partial^1}&(\iota_2)_!\iota_2^*{\mathscr X}\ar[r]^{\partial^2}&\ldots\ar[r]&(\iota_n)_!\iota_n^*{\mathscr X}\ar[r]^{\partial^n}&\ldots
}
\]
Given $n\in \N_{>0}$, since $(\iota_{n+1})_!(\varphi_{n+1})$ is a coreflection onto $\U_I$ and $(\iota_n)_!{\mathscr U}^n\in \U_I$, there is a unique map $\partial^n_{\mathscr U}\colon (\iota_n)_!{\mathscr U}^n\to (\iota_{n+1})_!{\mathscr U}^{n+1}$ such that $(\iota_{n+1})_!(\varphi_{n+1})\circ \partial^n_{\mathscr U}=\partial^n\circ (\iota_{n})_!(\varphi_{n})$.

Take now $k\in I$ and note that, by Lemma \ref{easy_relations_for_u_and_k}, $k^*(\epsilon_n)$ is invertible for all $n\geq d(k)$. Furthermore, $k^*(\epsilon_{n+1})\circ k^*( \partial^n)=k^*(\epsilon_n)$ by Lemma \ref{ed=e_lemma} and, therefore, $k^{*}(\partial^n)$ is an isomorphism for all $n\geq d(k)$. Similarly, $k^*(\partial^n_{\mathscr U})$ is an isomorphism for all $n\geq d(k)$, since $k^*(\partial^n_{\mathscr U})$ is a coreflection onto $\U$ of the isomorphism $k^{*}(\partial^n)$. Taking the cones of the vertical maps in the above diagram, we obtain a sequence like the following one:
\[
\xymatrix@C=40pt{
(\iota_1)_!{\mathscr V}^1\ar[r]^{\partial_{\mathscr V}^1}&(\iota_2)_!{\mathscr V}^2\ar[r]^{\partial_{\mathscr V}^2}&\ldots\ar[r]&(\iota_n)_!{\mathscr V}^3\ar[r]^{\partial_{\mathscr V}^n}&\ldots.
}
\]
Given $k\in I$, also $k^*(\partial_{\mathscr V}^n)$ is an isomorphism for all $n\geq d(k)$. To see this just consider the morphism of triangles $(k^*(\partial_{\mathscr U}^n),k^*(\partial^n),k^*(\partial_{\mathscr V}^n))$: as the first two components are invertible for  $n\geq d(k)$, then so has to be the third. We can now conclude by taking the Milnor colimits of these three sequences
(i.e.\ the rightmost terms in triangles as in Lemma~\ref{dist_triangle_countable_length})
to get the following triangle:
\[
{\mathscr U}\to {\mathscr X}\to {\mathscr V}\to \Sigma {\mathscr U}
\]
with ${\mathscr U}\in \U_I$, as ${\mathscr U}_k\cong ({\mathscr U}^{d(k)})_k\in \U$, and ${\mathscr V}\in \V_I$, as ${\mathscr V}_k\cong ({\mathscr V}^{d(k)})_k\in \V$, for all $k\in I$, by Example~\ref{countable_splitexact_ex}(2).
\end{proof}

Now we can prove Theorem~\ref{mainthm.lift_tstructure_general_thm} as stated in the Introduction. That is, $\t_I=(\U_I,\Sigma \V_I)$ is a $t$-structure in $\D(I)$ for any small category $I$, and the functor $\dia_I\colon \D(I)\to \D(\bbone)^I$ induces an equivalence of categories $\H_{I}\cong \H^I$.

\begin{proof}[Proof of Theorem~\ref{mainthm.lift_tstructure_general_thm}]
As in the proof of Lemmas \ref{lifting_tstructure_finite_lem} and \ref{lift_tstructure_countable_lem}, we have just to verify ($t$-S.3). Indeed, take $\mathscr X\in \D(I)$ and let us construct a suitable truncation triangle with respect to $\t_I$. 
As in Subsection~\ref{homotopical_epi_subs}, we consider the homotopical epimorphism $u\colon \Delta(I)\to I$. By Lemma  \ref{lift_tstructure_countable_lem},  $\t_{\Delta(I)}$ is a \mbox{$t$-structure} on $\D(\Delta(I))$, so we can take the following truncation triangle in $\D(\Delta(I))$:
\begin{equation}\label{truncation_in_delta_eq}
{\mathscr U}_{u^*{\mathscr X}}\to u^*{\mathscr X}\to {\mathscr V}_{u^*{\mathscr X}}\to \Sigma {\mathscr U}_{u^*{\mathscr X}}.
\end{equation}
We have to verify that ${\mathscr U}_{u^*{\mathscr X}}$ and ${\mathscr V}_{u^*{\mathscr X}}$ belong to the essential image of $u^*$. For that, we need to prove that ${\mathscr U}_{u^*{\mathscr X}}$ and ${\mathscr V}_{u^*{\mathscr X}}$ are constant on fibers. Indeed, let $i\in I$ and consider the following diagram in $\D(u^{-1}(i))$:
\[
\xymatrix{
\iota_i^*{\mathscr U}_{u^*{\mathscr X}}\ar[r]& \iota_i^*u^*{\mathscr X}\ar[r]\ar@{=}[d] &\iota_i^*{\mathscr V}_{u^*{\mathscr X}}\ar[r] &\Sigma \iota_i^*{\mathscr U}_{u^*{\mathscr X}}.\\
\pt_{u^{-1}(i)}^*{\mathscr U}_{{\mathscr X}_i}\ar[r]&\pt_{u^{-1}(i)}^*{\mathscr X}_i\ar[r]&\pt_{u^{-1}(i)}^*{\mathscr V}_{{\mathscr X}_i}\ar[r]&\pt_{u^{-1}(i)}^*\Sigma {\mathscr U}_{{\mathscr X}_i}
}
\]
Both the first and the second line are truncation triangles for $\iota_i^*u^*{\mathscr X}= \pt_{u^{-1}(i)}^*{\mathscr X}_i$ in $\D(u^{-1}(i))$ with respect to the $t$-structure $\t_{u^{-1}(i)}$. By the uniqueness of  truncations, $\iota_i^*{\mathscr U}_{u^*{\mathscr X}}\cong \pt_{u^{-1}(i)}^*{\mathscr U}_{{\mathscr X}_i}$ and $\iota_i^*{\mathscr V}_{u^*{\mathscr X}}\cong \pt_{u^{-1}(i)}^*{\mathscr V}_{{\mathscr X}_i}$, so ${\mathscr U}_{u^*{\mathscr X}}$ and ${\mathscr V}_{u^*{\mathscr X}}$ are constant on fibers. Thus, the triangle \eqref{truncation_in_delta_eq} lies in the essential image of $u^*$ and, therefore, we can lift it back along $u^*$ to obtain the desired truncation triangle of $\mathscr X$ in $\D(I)$.

Finally, note that the functor $\H_I\to \H^I$ induced by $\dia_I$ is fully faithful and essentially surjective by Theorem \ref{main_lifting_theorem}.
\end{proof}

\section{Finiteness conditions for objects in the base}

Let us start by fixing a strong and stable derivator 
\[
\D\colon \cat^{\op}\to \Cat.
\] 
In the first part of the section we recall what a closed monoidal derivator $\mathbb V$ is, and what it means for $\D$ to have a closed action by such a $\mathbb V$. We then introduce the class of homotopically finitely presented objects in $\D(\bbone)$, relative to a specific closed action on $\D$. In the second part of the section we introduce an intrinsic version of homotopic finite presentability for objects in $\D(\bbone)$, where by ``intrinsic'' we mean that this notion does not depend on the choice of a specific action, as it just uses data from $\D$ itself. We conclude the section by proving that a given object in $\D(\bbone)$ is compact (in the usual sense of triangulated categories with coproducts) if, and only if, it is intrinsically homotopically finitely presented if, and only if, it is homotopically finitely presented with respect to any closed action on $\D$ by a closed monoidal derivator $\mathbb V$, whose tensor unit is a compact generator. 

%
%

\subsection{Closed actions by a closed monoidal derivator}

Following \cite[Def.\,3.1]{GPS-additivity}, a {\em monoidal structure} on a prederivator $\mathbb{V}$ is given by
\begin{itemize}
\item a tensor product $\otimes\colon \mathbb{V}\times\mathbb{V}\to\mathbb{V}$ and
\item a tensor unit $\mathbb{S}\colon\mathbb{Y}_\bbone\to\mathbb{V}$ (here $\mathbb{Y}_\bbone$ stands for the terminal derivator, i.e.\ $\mathbb{Y}_\bbone(I)=\bbone$ for each $I\in\Cat$, cf.\ Examples~\ref{expl.Yoneda_prederivator} and~\ref{expl.Yoneda_derivator}),
\end{itemize}
together with the usual unitality, associativity and symmetry constraints in the $2$-category of prederivators. One can restrict this definition to derivators and, after some tedious work, arrive at the definition of a {\em closed monoidal derivator} $\mathbb{V}$, see \cite[Def.\,8.5]{GPS-additivity}. We will further say that $\mathbb{V}$ is \emph{monogenic} (cf.\ \cite[Def.\ 1.1.4]{hps-axiomatic-stable-htpy-theory}) when it is further a strong and stable derivator, and the image of the unique object by the functor $\mathbb{S}:\mathbb{Y}_1(\bbone )=\bbone\to\mathbb{V}(\bbone)$, that we still denote by $\mathbb{S}$, compactly generates $\mathbb{V}(\bbone )$.

Given a closed monoidal derivator $\mathbb{V}$ we have, for each $I,\, J\in\cat$, the functor of {\em internal hom-objects}
\begin{align*}
\rhd\colon \mathbb{V}(I)^{\op}\times\mathbb{V}(J) &\longrightarrow \mathbb{V}(I^{\op}\times J)\\
(\mathscr X,\mathscr Y)&\longmapsto \mathscr X\rhd\mathscr Y.
\end{align*}
%
The situation further generalizes as follows: if $\mathbb{V}$ is a closed monoidal derivator, we say that $\D$ {\em has a closed $\mathbb{V}$-action} if there is a morphism $\otimes\colon \mathbb{V}\times\D\to\D$ together with associativity and unitality constraints in the $2$-category of derivators which is at the same time a \mbox{2-variable} adjoint in the sense of~\cite[Def.\,8.1]{GPS-additivity}. In that case, \cite[Thm.\,9.1]{GPS-additivity} provides us again with an {\em internal hom-functor} for each $I,\, J\in\cat$,
\[ \rhd\colon \D(I)^{\op}\times\D(J) \longrightarrow \mathbb{V}(I^{\op}\times J). \]
Moreover, if we specialize to $I=\bbone$, these bifunctors, for any fixed  $C\in\D(\bbone)$, naturally assemble to give a morphism of derivators
\begin{equation} \label{eq.V-enrichment}
C\rhd-\colon \D \longrightarrow \mathbb{V}.
\end{equation}
Crucial for our purposes is the fact that, in this last situation, for all $V\in\mathbb{V}(\bbone)$ and $X,\,Y\in\D(\bbone)$, we have an adjunction isomorphism,  which is functorial in $V$, $X$ and $Y$: 
\[
\D(\bbone)(V\otimes X,Y)\cong\mathbb{V}(\bbone)(V,X\rhd Y).
\]
Our first example of this situation 
is aimed at algebraically minded readers:

\begin{example} \label{ex.DG acted by DAb}
Let $\G$ be a Grothendieck Abelian category. Then, the derived category $\Der(\G)$ is enriched over $\Der(\Ab)$ (we have $\mathbf{R}\Hom_\G\colon \Der(\G)^{\op} \times \Der(\G) \to \Der(\Ab)$). Similarly, the canonical derivator $\D_\G$ enhancing $\Der(\G)$ (see Example~\ref{expl.derived categories}) is enriched over $\D_\Ab$ and from this one gets an internal hom bifunctor %
\[ 
-\rhd-=\mathbf{R}\Hom_\G\colon \D_\G(I)^{\op}\times\D_\G(J) \longrightarrow \D_\Ab(I^{\op}\times J), 
\]
Note that $\D_\Ab$ is strong and stable, and it is also closed monoidal with respect to usual derived tensor product. Furthermore, the  unit object $\mathbb{S}\in\D_\Ab(\bbone)=\Der (\Ab)$ is precisely $\mathbb{Z}$, viewed as a stalk complex at zero, so that $\D_\Ab$ is monogenic. For $I=\bbone$ and  $C\in\D_\G(\bbone)=\Der(\G)$, we have the obvious morphism of derivators
\begin{equation} \label{eq.Ab-enrichment}
C\rhd -=\mathbf{R}\Hom_\G(C,-)\colon \D_\G \longrightarrow \D_\Ab.
\end{equation}
\end{example}

 We next remind the reader that the situation  applies to any strong and stable derivator:
 
 \begin{example} \label{expl.cisinski-tabuada enrichment}
Let $\Sp$ be a closed monoidal model category of spectra and $\D_\Sp$ be the corresponding strong stable derivator of spectra. More concretely, there are various mutually Quillen equivalent stable closed monoidal model categories of spectra which we can choose for $\Sp$, see \cite{EKMM-model-for-spectra,HSS-symmetric-spectra,MMSS-monoidal-spectra-comparison,schwede2007untitled}, and we pass to the corresponding homotopy strong stable derivator (see Example~\ref{expl.derived categories}).
This will be a closed monoidal derivator by \cite[Thm.\,9.13]{GPS-additivity}.  Moreover, its unit object $\mathbb{S}$ is the sphere spectrum (see \cite[Thm.\,7.13]{schwede2007untitled}), which is  compact in $\D_\Sp (\bbone)=\text{Ho}(\Sp)$ (see our Appendix~\hyperref[sec_spectra]{A}). 
For the sake of completeness, we also remark that the  Bousfield-Friedlander stable model category of spectra studied in Appendix~\hyperref[sec_spectra]{A} is not suitable to derive the monoidal structure in $\D_\Sp$---although it is also Quillen equivalent to the others, it does \emph{not} have the required structure of a monoidal model category.

Now, given any strong and stable derivator, the discussion in \cite[Appendix A.3]{CisTab11} shows that there is a canonical closed action $\otimes\colon \D_\Sp \times \D \to \D$. This in particular implies  that for each $C \in \D(\bbone)$, we have a morphism of derivators
\[ C\rhd -= F(C,-)\colon \D \longrightarrow \D_\Sp, \]
called the {\em function spectrum}, which is right adjoint (internal to the $2$-category of derivators) to $-\otimes C\colon \D_\Sp \to \D$.
\end{example}

\subsection{Homotopically finite presentability relative to a closed action}\label{subs_homofpextr}

All through this subsection, we assume that $\D$ is a strong and stable derivator,  $\mathbb{V}$ a closed monoidal, strong and stable derivator, and we fix a closed action:
\[
\otimes\colon\mathbb{V}\times\D\longrightarrow\D,
\] 
Given $C\in\D(\bbone)$, we have corresponding internal hom-functor $C\rhd -\colon\D\to\mathbb{V}$. By construction, $C\rhd -$  is a morphisms of derivators and so it comes equipped with a natural isomorphism
\[
\gamma_u^{C}:=\gamma_u^{C\rhd -}\colon u^*\circ (C\rhd -)\ \tilde\Longrightarrow\ (C\rhd -)\circ u^*,
\]
for each $u\colon J\to I$ in $\cat$. As in Subsection \ref{subs_pred}, one can then construct the following natural transformation:
\[
(\gamma_{u}^C)_!\colon u_!\circ (C\rhd -)\ \Longrightarrow\ (C\rhd -)\circ u_!.
\]  
In particular, when $I=\bbone$ and $u:=\pt_J\colon J\to\bbone$, we get a canonical natural transformation 
\[
(\gamma_{\pt_J}^C)_!\colon \hocolim_J\circ (C\rhd -)\Longrightarrow (C\rhd-)\circ \hocolim_J .
\]
Note that the  homotopy colimit  on the left hand side is taken in $\mathbb{V}$, while that on the right hand side is taken in $\D$. This leads to the following crucial definition:

\begin{definition}
Given a small category $I$, an object $C \in \mathbb{D}(\bbone)$ is said to be \emph{$I$-homotopically finitely presented in $\mathbb{D}$, relative to $\mathbb{V}$}, when the last given natural transformation is an isomorphism, i.e., when the  
canonical map
\[
\hocolim_I (C\rhd\mathscr X)\longrightarrow C\rhd(\hocolim_I \mathscr X)
\]
is an isomorphism for all $\mathscr X\in\mathbb{D}(I)$. 
We say that $C$ is \emph{homotopically finitely presented in $\D$, relative to $\mathbb V$},  when it is $I$-homotopically finitely presented, for any directed set $I$.
\end{definition}

In the algebraic case, we have the following very intuitive interpretation.

\begin{example} \label{ex.homotopicallyfp-algebraic-case}
In the situation of Example \ref{ex.DG acted by DAb}, for each $C\in\D_\G(\bbone)=\Der(\G)$,  the morphism $C\rhd -$ is given by $\mathbf{R}\Hom_\G(C,-)\colon\D_\G(I)=\Der (\G^I)\to\Der (\Ab^I)=\D_{\Ab}(I)$. Saying that $C$ is $I$-homotopically finitely presented in $\D_\G$, relative to $\D_\Ab$, amounts to say that the canonical morphism 
\[
\hocolim_I\, \, \mathbf{R}\Hom_\G(C,\mathscr X)\longrightarrow\mathbf{R}\Hom_\G(C,\, \hocolim_I\mathscr X)
\] 
is an isomorphism, for all $\mathscr X\in\D_\G(I)=\Der (\G^I)$.
\end{example}

In the next lemma we  show that, for certain shapes $I$, the $I$-homotopical finite presentability comes for free. We recall that $I\in\cat$ is \emph{strictly homotopy finite} if its nerve contains only finitely many non-degenerate simplices, \cite[Sec.\,7]{Ponto-Shulman}. In our terminology, this is equivalent to $I$ being of finite length and having finitely many morphisms. A small category is \emph{homotopy finite} if it is equivalent to one that is strictly homotopy finite.

\begin{lemma} \label{lem.finite-htpy-fin-pres-for-free}
Every object $C\in\D(\bbone)$ is $I$-homotopically finitely presented in $\D$, relative to $\mathbb V$, if the small category $I$ is homotopy finite.
\end{lemma}
\begin{proof}
The morphism $C\rhd-\colon\D\to\mathbb{V}$ is exact (=left exact, as we are in the stable setting) since it is right adjoint to $-\otimes C$ (see~\cite[Def.\,3.15 and Coro.\,4.17]{Moritz}). This is enough to conclude since, by \cite[Thm.\,7.1(ii)]{Ponto-Shulman}, every exact morphism between stable derivators preserves homotopy finite colimits.
\end{proof}

We can now give a handy characterization of homotopical finite presentability:

\begin{proposition} \label{prop.Corollary2.4-referee}
For an object $C\in\D(\bbone)$, the following assertions are equivalent:
\begin{enumerate}[\rm (1)]
\item the morphism of derivators $C\rhd -:\D\to\mathbb{V}$ preserves all homotopy colimits;
\item the object $C$ is homotopically finitely presented in $\D$, relative to $\mathbb{V}$;
\item the morphism of derivators $C\rhd -:\D\to\mathbb{V}$ preserves  coproducts.
\end{enumerate}
\end{proposition}
\begin{proof}
The implication ``(1)$\Rightarrow$(2)'' is clear, and ``(2)$\Rightarrow$(3)'' follows by Lemma \ref{lem.referee2.3}. Finally, the implication  ``(3)$\Rightarrow$(1)'' follows since $C\rhd -:\D\to\mathbb{V}$ is an exact morphism of derivators, and so it preserves homotopy pushouts. By \cite[Thm.\,7.13]{Ponto-Shulman}, a morphism of derivators that preserves coproducts and homotopy pushouts preserves all homotopy colimits. 
\end{proof}

\subsection{Intrinsic definition of homotopical finite presentability}\label{subs_homofpintr}

In this subsection we present an alternative definition of homotopically finitely presented object in $\D(\bbone)$, that depends only on $\D$, and not on any action by a closed monoidal derivator. 
Indeed, given a small category $I$, consider the following adjunction:
\[
\hocolim_I\adjunct{\eta_I}{\epsilon_I}\pt_I^*\colon \D(I)\longrightarrow \D(\bbone).
\] 
Recall also that, by Example \ref{description_constant_diagram}, $\dia_I\circ \pt_I^*=\kappa_I\colon \D(\bbone)\to \D(\bbone)^I$ is the constant diagram functor, and so we get the following natural transformation 
\[
\dia_I(\eta_I)\colon\dia_I\Longrightarrow \kappa_I\circ\hocolim_I\colon \D(I)\longrightarrow \D(\bbone)^I.
\] 
In turn, this yields the following canonical morphism of Abelian groups: 
\[
\mu_{C,\mathscr X}\colon\colim_I\D(\bbone)(C,\mathscr X_{i})\longrightarrow\D(\bbone)(C,\hocolim_I\mathscr X),
\] 
for any $C\in\D(\bbone)$ and $\mathscr X\in\D(I)$, which is natural in both variables. 

\begin{definition} \label{def.intrinsic-homotopically-fp}
An object $C\in\D(\bbone)$ is \emph{intrinsically homotopically finitely presented in $\D$} when, for every directed set $I$, the morphism 
\[
\xymatrix{
{\mu}_{C,\mathscr X}\colon\varinjlim_I\D(\bbone)(C,\mathscr X_i)\longrightarrow\D(\bbone)(C,\hocolim_I\mathscr X)
}
\] 
is an isomorphism, for all $\mathscr X\in\D(I)$.
\end{definition}



The following result will be crucial in the proof of Theorem \ref{thm.compactness-versus-homtopicallyfp}.

\begin{lemma} \label{lem.homopcompact-implies-compact}
If an object $C\in\D(\bbone)$ is homotopically finitely presented in $\D$, then it is also compact in $\D(\bbone)$. 
\end{lemma} 
\begin{proof}
Let $I$ be a set and adopt the notation of the proof of Lemma \ref{lem.referee2.3}, that is, $I$ is a set and $P:=\mathcal P^{<\omega}(I)$ is the directed set of finite subsets of $I$, with $u\colon I\to P$ the obvious inclusion.
Given $F\in P$, there is an equivalence $F\cong u/F$ and, by (Der.4), we get the following isomorphism
\[
\hocolim_{F}\mathscr X_{\restriction F}\cong \hocolim_{u/F}\, \pr_F^* \mathscr X\tilde\longrightarrow (u_!\mathscr X)_F
\]
for each $\mathscr X\in \D(I)$. We can conclude by the following series of isomorphisms:
\begin{align*}
\textstyle{\coprod_I}\D(\bbone)(C,X_i)&\cong\textstyle{\varinjlim_{P}\ \coprod_F}\ \D(\bbone)(C,\mathscr X_i)\\
&\cong\textstyle{\varinjlim_{P}}\ \D(\bbone)(C,(u_!\mathscr X)_F)\\
& \cong{\D(\bbone)}(C,\hocolim_P(u_!\mathscr X))&\text{as $P$ is homo.\,f.p.};\\
&\cong\D(\bbone)(C,(\pt_I)_!\mathscr X)&\text{as $\pt_I=\pt_P\circ u$};\\
&\cong\D(\bbone)\left(C,\textstyle{\coprod_I}X_i\right)&\text{as $I$ is discrete.}&\qedhere
\end{align*}
\end{proof}

\subsection{Compactness versus homotopical finite presentability}

In this final subsection we show that the two versions (extrinsic and intrinsic) of homotopical finite presentability coincide provided that, in the extrinsic case, the $\otimes$-unit of the given closed action, is a compact {generator}. Furthermore, both versions of  homotopical finite presentability are equivalent to compactness, in the usual sense of triangulated categories with coproducts.
This generalizes an analogous result obtained in~\cite[Prop.\ 6.6]{stovicek2014derived} for the special case of $\D=\D_R$, the canonical derivator enhancing $\Der(\mod R)$ for a ring $R$ (cf.\ Example~\ref{expl.derived categories}), and $\mathbb V=\D_\Ab$.

\begin{theorem} \label{thm.compactness-versus-homtopicallyfp}
Let $\mathbb{V}$ be a closed monoidal derivator which is monogenic, and $\D$ a stable derivator with a closed $\mathbb{V}$-action. The following assertions are equivalent for $C\in\D(\bbone)$:
\begin{enumerate}[\rm (1)]
\item $C$ is intrinsically homotopically finitely presented in $\D$;
\item $C$ is compact in $\D(\bbone)$, i.e.,  the functor $\D(\bbone)(C,-)\colon\D(\bbone)\to\Ab$ preserves coproducts;
\item the morphism of derivators $C\rhd -\colon\D\to\mathbb{V}$ preserves coproducts;
\item $C$ is homotopically finitely presented in $\D$, relative to $\mathbb{V}$;
\item the morphism of derivators $C\rhd -\colon\D\to\mathbb{V}$ preserves all homotopy colimits. 
\end{enumerate}
\end{theorem}
\begin{proof}
The equivalence of assertions (3--5) is Proposition \ref{prop.Corollary2.4-referee} and the implication ``(1)$\Rightarrow$(2)'' is Lemma \ref{lem.homopcompact-implies-compact}.  

\smallskip\noindent
(2)$\Leftrightarrow$(3). By (Der.1), assertion (3) holds if the following functor preserves coproducts:
\[
C\rhd -\colon\D(\bbone)\longrightarrow\mathbb{V}(\bbone).
\] 
Let then $I$ be a set, $\mathscr X=(X_i)_{I}\in \D(I)\cong\D(\bbone)^I$, and consider the canonical morphism 
\[
\xymatrix{\lambda =(\gamma_{\pt_I}^C)_! \colon\coprod_{I}(C\rhd X_i)\longrightarrow C\rhd (\coprod_{ I}X_i)}\qquad \text{in $\mathbb{V}(1)$,}
\]  
as in Subsection \ref{subs_homofpextr}. Furthermore, consider the canonical morphism 
\[
\xymatrix{\mu=\mu_{\Sigma^nC,\mathscr X}\colon\coprod_{I}\D(\bbone)(\Sigma^n C,X_i)\longrightarrow\D(\bbone)(\Sigma^n C,\coprod_{I}X_i)}\qquad\text{in $\Ab$,}
\] 
as in Subsection \ref{subs_homofpintr}. Then, assertion (3) holds if, and only if, $\lambda$ is an isomorphism for all $\mathscr X\in \D(I)$, and assertion (2) holds if, and only if, $\mu$ is an isomorphism, for all $\mathscr X\in \D(I)$ and all $n\in\mathbb{Z}$. In turn, since the unit $\mathbb{S}$ is a compact generator of $\mathbb{V}(\bbone)$, we know that $\lambda$ is an isomorphism if, and only if, the  map 
\[
\xymatrix{\lambda_*\colon\mathbb{V}(\bbone)(\Sigma^n\mathbb{S},\coprod_{I}(C\rhd X_i))\longrightarrow\mathbb{V}(\bbone)(\Sigma^n\mathbb{S},C\rhd (\coprod_{I}X_i))}
\] 
is an isomorphism, for all $n\in\mathbb{Z}$. Now, to see that $\mu$ is an isomorphism if, and only if, $\lambda_*$ is an isomorphism, we consider the following commutative diagram, where all vertical arrows are isomorphisms:
\[
\xymatrix@R=15pt{%
	\mathbb{V}(\mathbf{1})(\Sigma^n\mathbb{S},\coprod_{ I}(C\rhd X_i)) \ar[r]^-{\lambda_\ast} &
		\mathbb{V}(\mathbf{1})(\Sigma^n\mathbb{S},C\rhd (\coprod_{I}X_i)) \cr
	\coprod_{I}\mathbb{V}(\mathbf{1})(\Sigma^n\mathbb{S},C\rhd X_i) \ar[u]^-{\cong} &
		\mathbb{D}(\bbone)(\Sigma^n\mathbb{S}\otimes C,\coprod_{I}X_i) \ar[u]_-{\cong} \cr
	\coprod_{I}\mathbb{D}(\bbone)(\Sigma^n\mathbb{S}\otimes C,X_i) \ar[u]^-{\cong} &
		\mathbb{D}(\bbone)(\Sigma^nC,\coprod_{I}X_i) \ar[u]_-{\cong} \cr
	\coprod_{I}\mathbb{D}(\bbone)(\Sigma^nC,X_i) \ar[u]^-{\cong}\ar[ur]_-{\mu}
}
\]
For the invertibility of the vertical arrows, use the natural isomorphism $(\mathbb{S}\otimes -)\cong \id_\D$, the adjunction $(-\otimes C)\vdash (C\rhd -)$ and the compactness of $\mathbb{S}$ in $\mathbb{V}(\bbone)$. Therefore, $\lambda_*$ is an isomorphism if, and only if, so is $\mu$.

\smallskip\noindent
(2)$\Rightarrow$(1). As we now know that assertions (2--5) are all equivalent, for any given closed symmetric monoidal derivator $\mathbb{V}$ which is monogenic and acts on $\D$, we can assume in this implication that $\mathbb{V}=\D_\Sp$ and that $(C\rhd -)=F(C,-)=\D\to\D_\Sp$ is the function spectrum (see Example \ref{expl.cisinski-tabuada enrichment}). In particular,  $C$ is homotopically finitely presented, relative to $\D_\Sp$. Let then $I$ be any directed set and $\mathscr X\in\D(I)$. As in the proof of the implication ``(2)$\Leftrightarrow$(3)'', consider the canonical map $\lambda \colon\hocolim_I(C\rhd\mathscr X)\to C\rhd\hocolim_IX$, and the induced maps
\[
\lambda^*\colon\mathbb{V}(\bbone)(\Sigma^n\mathbb{S},\hocolim_I(C\rhd\mathscr X))\longrightarrow \mathbb{V}(\bbone)(\Sigma^n\mathbb{S}, C\rhd\hocolim_IX)
\] 
for each $n\in\mathbb{Z}$. We are going to verify that $\lambda^*$ is an isomorphism if, and only if, the following canonical map is an isomorphism for all $n\in\Z$:
\[
\xymatrix{{\mu}={\mu}_{\Sigma^n C,\mathscr X}\colon\varinjlim_I\mathcal D(\Sigma^n C,\mathscr X_i)\longrightarrow\mathcal D(\Sigma^n C,\hocolim_I\mathscr X)}.
\] 
In fact, arguing as above, we get a commutative diagram whose columns are isomorphisms:
\[
\xymatrix@R=15pt{%
	\mathbb{V}(\mathbf{1})(\Sigma^n\mathbb{S},\hocolim_I(C\rhd\mathscr{X})) \ar[r]^-{\lambda_\ast} &
		\mathbb{V}(\mathbf{1})(\Sigma^n\mathbb{S},C\rhd\hocolim_I\mathscr{X}) \cr
	\varinjlim_I\mathbb{V}(\mathbf{1})(\Sigma^n\mathbb{S},C\rhd \mathscr{X}_i) \ar[u]^-{\cong} &
		\mathbb{D}(\bbone)(\Sigma^n\mathbb{S}\otimes C,\hocolim_I \mathscr{X}) \ar[u]_-{\cong} \cr
	\varinjlim_I\mathbb{D}(\bbone)(\Sigma^n\mathbb{S}\otimes C,\mathscr{X}_i) \ar[u]^-{\cong} &
		\mathbb{D}(\bbone)(\Sigma^nC,\hocolim_I\mathscr{X}) \ar[u]_-{\cong} \cr
	\varinjlim_I\mathbb{D}(\bbone)(\Sigma^nC,\mathscr{X}_i) \ar[u]^-{\cong}\ar[ur]_-{\mu}
}
\]
where the upper left vertical arrow is an isomorphism because the sphere spectrum $\mathbb{S}$ is intrinsically homotopically finitely presented in $\D_\Sp$ (see Appendix~\hyperref[sec_spectra]{A}).
\end{proof}

As an immediate consequence (see Example  \ref{expl.cisinski-tabuada enrichment}), we get

\begin{corollary}\label{coro_compact=hfp}
Let $\D\colon \cat^{\op}\to \Cat$ be a strong and stable derivator. The following assertions are equivalent for an object $C\in\D(\bbone)$:
\begin{enumerate}[\rm (1)]
\item $C$ is intrinsically homotopically finitely presented in $\D$;
\item $C$ is a compact object of $\D(\bbone)$;
\item $C$ is homotopically finitely presented in $\D$, relative to $\D_\Sp$.
\end{enumerate}
\end{corollary}

\section{Finiteness conditions and directed colimits in the heart}\label{sec_dirlim_heart}

Let us start by fixing a strong and stable derivator 
\[
\D\colon \cat^{\op}\to \Cat.
\] 
In this section we  discuss some finiteness conditions on a $t$-structure $\mathbf{t} = (\U,\Sigma\V)$ in $\D(\bbone)$. The simplest condition which we can impose is that $\V$ is closed under coproducts---this in fact makes sense in any triangulated category. However, it will turn out later in Example~\ref{ex.HRS t-structure smash but non-htpy} that this is not sufficient for the heart to be (Ab.$5$) Abelian. A stronger condition which does imply exactness of directed colimits in the heart is that $\V$ is closed under directed homotopy colimits. This will establish Theorem~\ref{mainthm.htpy smashing has AB5 heart}.

\subsection{Homotopically smashing $t$-structures}

Let us start introducing the finiteness conditions which we are going to study:

\begin{definition} \label{def.homotcompact-homotsmashing}
A $t$-structure $\t =(\U,\Sigma\V)$ in $\D(\bbone)$ is said to be:
\begin{itemize}
\item \emph{compactly generated} when $\V$ is of the form $\V=\mathcal S^{\perp_{\leq 0}}$, where $\cal S\subseteq \D(\bbone)$ is a set of compact objects (see Subsection \ref{tria_and_t_subs});
\item \emph{homotopically smashing} (with respect to $\mathbb{D}$) when $\mathcal{V}$ is closed under taking directed homotopy colimits;
\item \emph{smashing} when $\mathcal{V}$ is closed under taking coproducts. 
\end{itemize}
\end{definition}

The three notions relate as follows.

\begin{proposition} \label{prop.homtopsmashing-implies-smashing}
Let $\t=(\U,\Sigma\V)$ be a $t$-structure in $\D(\bbone)$ and consider the following conditions:
\begin{enumerate}[\textup(1\textup)]
\renewcommand{\theenumi}{\textup{\arabic{enumi}}}
\item $\t$ compactly generated;
\item $\t$ homotopically smashing;
\item $\t$ smashing.
\end{enumerate}
Then, the implications ``(1)$\Rightarrow$(2)$\Rightarrow$(3)'' hold and none of them can be reversed in general.
\end{proposition}
\begin{proof}
The implication ``(1)$\Rightarrow$(2)'' is a direct consequence of the definition of intrinsically homotopically finitely presented object and of Corollary~\ref{coro_compact=hfp}. As for the implication ``(2)$\Rightarrow$(3)'', let us assume that $\t$ is homotopically smashing and let $\mathscr X=(X_i)_{i\in I}$ be a family of objects of $\mathcal{U}^\perp$. We have seen in the proof of Lemma \ref{lem.homopcompact-implies-compact} that, looking at $I$ as a discrete category,  we can view $\mathscr X$ as an object of $\mathbb{D}(I)$ and then there is a canonical isomorphism $\coprod_{i\in I}X_i\cong \hocolim_P u_!\mathscr X$, with the same notation as in that proof. Then we just need to check that $(u_!\mathscr X)_F\in\mathcal{U}^\perp$, for all finite subsets $F\subseteq I$. But, in fact, $(u_!\mathscr X)_F\cong\coprod_{i\in F}X_i$, which is an object of $\mathcal{U}^\perp$ since co-aisles are closed under finite coproducts.
We refer to  Examples~\ref{ex.HRS t-structure smash but non-htpy} and~\ref{ex.HRS t-structure non-cpt gen} in the next section for explicit counterexamples showing that the implications in the statement cannot be reversed in general.
\end{proof}

\subsection{Directed colimits in the heart}

Here we prove Theorem~\ref{mainthm.htpy smashing has AB5 heart}, as stated in the Introduction, i.e.\ that the heart of a homotopically smashing $t$-structure is an (Ab.$5$) Abelian category. We start with the following consequence of Theorem~\ref{mainthm.lift_tstructure_general_thm}:

\begin{lemma}\label{lem.hocolim=colim}
Let $\t=(\U,\Sigma\V)$ be a $t$-structure in $\D(\bbone)$. Given $I\in\cat$ and $\mathscr X\in \H_I$, 
\[
\colim_I\, \dia_I \mathscr X \cong \tau^{\Sigma\V}(\hocolim_I \mathscr X),
\]
where the  colimit on the left hand side is taken in the heart $\H = \U\cap\Sigma\V$.
\end{lemma}
\begin{proof}
By  Theorem \ref{mainthm.lift_tstructure_general_thm}, $\dia_I$ induces an equivalence $F_I\colon\H_I\to \H^I$, and we fix a quasi-inverse $F_I^{-1}\colon \H^{I}\to \H_I$. Now, $\colim_I$ is defined as the left adjoint to $\kappa_I\colon \H\to \H^{I}$ so, composing the two adjunctions $\colim_I\dashv \kappa_I$ and $F_I\dashv F_I^{-1}$, we obtain that $\colim_I\circ F_I$ is left adjoint to $F_I^{-1}\circ \kappa_I$. Furthermore, $F_I^{-1}\circ \kappa_I\cong \pt_I^*\restriction_{\H}$.

On the other hand, $(\tau^{\Sigma V})_{\restriction\U}$ is left adjoint to the inclusion $\H\to\U$ and $\hocolim_I$ is left adjoint to $\pt_I^*\colon \D(\bbone)\to \D(I)$. Composing the restrictions to the corresponding subcategories of the two adjunctions, we see that $\tau^{\Sigma V}\circ(\hocolim_I)_{\restriction\H_I}$ is a left adjoint to $(\pt_I^*)_{\restriction\H}\colon\H\to\H_I$.
Thus, we deduce the desired natural isomorphism:
\[
\tau^{\Sigma\V}\circ(\hocolim_I)_{\restriction\H_I}\cong\colim_I\circ (\dia_I)_{\restriction\H_I}.\qedhere
\]
\end{proof}

The above lemma has the following immediate consequence: 

\begin{corollary}\label{cor.hocolim=colim}
If $\t=(\U,\Sigma\V)$ is a homotopically smashing $t$-structure in $\D(\bbone)$, $I$ is a directed set and $\mathscr X\in \H_I$, then $\varinjlim_I \dia_I \mathscr X \cong\hocolim_I \mathscr X$.
\end{corollary}

We can now proceed with the proof of Theorem~\ref{mainthm.htpy smashing has AB5 heart}: 

\begin{proof}[Proof of Theorem~\ref{mainthm.htpy smashing has AB5 heart}]
In view of Proposition~\ref{prop.homtopsmashing-implies-smashing}, it remains to prove the (Ab.$5$) condition for the heart of a homotopically smashing $t$-structure $\t=(\U,\Sigma \V)$. That is, given three diagrams $X$, $Y$, and $Z\colon I\to \H$ for some directed set $I$, together with natural transformations $f\colon X\Rightarrow Y$ and $g\colon Y\Rightarrow Z$, such that 
\[
0\longrightarrow X_i\overset{f_i}{\longrightarrow} Y_i\overset{g_i}{\longrightarrow}Z_i\longrightarrow 0
\]
is a short exact sequence in $\H$ for any $i\in I$, then
\[
\xymatrix{
0\longrightarrow \varinjlim_{ I}X_i\overset{}{\longrightarrow} \varinjlim_{I}Y_i\overset{}{\longrightarrow}\varinjlim_{I}Z_i\longrightarrow 0}
\]
is short exact. By Theorem~\ref{mainthm.lift_tstructure_general_thm}, the short exact sequence $0\to X\to Y\to Z\to 0$ in $\H^I$, can be identified with a short exact sequence in $\H_I\subseteq \D(I)$. Remember that a sequence in the heart of a $t$-structure is short exact if and only if it represents a triangle of the ambient category whose three first terms happen to lie in the heart. Hence, there is a map $Z\to \Sigma X$ such that
\[
X\longrightarrow Y\longrightarrow Z\longrightarrow \Sigma X
\]
is a triangle in $\D(I)$. Taking homotopy colimits we get a triangle in $\D(\bbone)$:
\[
\hocolim_I X\longrightarrow \hocolim_I Y\longrightarrow \hocolim_I Z\longrightarrow \Sigma \hocolim_I X.
\]
As $\t$ is homotopically smashing,  $\hocolim_I X$, $\hocolim_I Y$ and $\hocolim_I Z$ belong to $\H$, so the following sequence in $\H$ is short exact:
\[
0\longrightarrow \hocolim_I X\longrightarrow \hocolim_I Y\longrightarrow \hocolim_I Z\longrightarrow 0.
\]
One concludes by Corollary~\ref{cor.hocolim=colim} since $\dia_I X\cong (X_i)_{i\in I}$, and similarly for $Y$ and $Z$.
\end{proof}

\section{Tilted \texorpdfstring{$t$}{t}-structures and examples}\label{sec_ex}

This section is devoted to providing examples of smashing $t$-structures which are not homotopically smashing and of homotopically smashing $t$-structures which are not compactly generated. This will complete the proof of Proposition~\ref{prop.homtopsmashing-implies-smashing}. The strategy is to start with the canonical $t$-structure of a suitable Grothendieck Abelian category, tilt it using a torsion pair (see Example~\ref{expl.tilted t-str}), and relate the properties of the resulting tilted $t$-structure to the properties of the torsion pair. Throughout the section, 
\[
\D\colon \cat^{\op}\to \Cat
\]
will be a fixed, strong and stable derivator.

\subsection{Homotopically smashing tilts of $t$-structures}

Our first result characterizes the situation when the Happel-Reiten-Smal{\o} tilt $\t_\tau$ of a homotopically smashing $t$-structure $\t$ is homotopically smashing again.

\begin{proposition}\label{prop.homo_smash_iff_colim}
Let  $\t=(\U,\Sigma\V)$ be a homotopically smashing $t$-structure in $\D(\bbone)$, and let $\t_\tau := (\U_\tau,\Sigma\V_\tau)$ be the tilt of this $t$-structure with respect to a torsion pair $\tau=(\mathcal T,\mathcal F)$ in the heart $\H:=\U\cap \Sigma\V$ of $\t$. Then $\t_\tau$ is a smashing $t$-structure. Moreover, the following conditions are equivalent:
\begin{enumerate}[\textup(1\textup)]
\renewcommand{\theenumi}{\textup{\arabic{enumi}}}
\item $\t_\tau$ is homotopically smashing;
\item $\mathcal F$ is closed under taking directed colimits in $\H$. 
\end{enumerate}
\end{proposition}
\begin{proof}
Since the heart $\H$ of $\t$ is an (Ab.$5$) Abelian category by Theorem~\ref{mainthm.htpy smashing has AB5 heart}, the canonical map $\coprod_I X_i \to \prod_I X_i$ is a monomorphism for any set $I$ and $(X_i)_I\subseteq \H$. Indeed, the latter map is a directed colimit of the split inclusions $\coprod_J X_i \cong \prod_J X_i \to \prod_I X_i$, where $J$ runs over all finite subsets of $I$. In particular, $\cal F$ is closed under coproducts both in $\H$ and $\D(\bbone)$ and, since $\V$ is closed under coproducts in $\D(\bbone)$ by Proposition~\ref{prop.homtopsmashing-implies-smashing}, so is~$\V_\tau$. Thus, $\t_\tau$ is smashing.

\smallskip\noindent
(1)$\Rightarrow$(2). Let $I$ be a directed set and let $F=(F_i)_{i\in I}$ be a direct system in $\H$, that is, an object in $\H^I$. By Theorem~\ref{mainthm.lift_tstructure_general_thm}, $\H^{I}\cong \H_I\subseteq \D(I)$, so we can identify $F$ with an object in $\H_I$ and, as such, there is an isomorphism
\[
\xymatrix{\hocolim_I F\cong \varinjlim_{I} F_i}
\]
where the colimit on the right-hand side is taken in $\H$ (see Corollary~\ref{cor.hocolim=colim}). Now, if $F_i\in \mathcal{F}$ for each $i\in I$, the fact that $\t_\tau$ is homotopically smashing  tells us that $\hocolim_I F\in \V_\tau$ and so $\hocolim_I F\in \V_\tau\cap \H=\mathcal F$. 

\smallskip\noindent
(2)$\Rightarrow$(1). Let $I$ be a directed set and let $\mathscr Y\in \D(I)$ be such that $\mathscr Y_i\in \V_\tau$ for any $i\in I$. Consider the truncation triangle of $\mathscr Y$ with respect to the lifted $t$-structure $\t_{I}$ in $\D(I)$:
\begin{equation}\label{trunc_tria_lifted_t}
\mathscr U\longrightarrow\mathscr Y\longrightarrow \mathscr V\longrightarrow \Sigma \mathscr U,
\end{equation}
where $\mathscr U_i\in \U_{}$ and $\mathscr V_i\in \V_{}$, for any $i\in I$. For any $i\in I$ we get a triangle in $\D(\bbone)$:
\[
\Sigma^{-1}\mathscr V_i\longrightarrow \mathscr U_i\longrightarrow \mathscr Y_i\longrightarrow \mathscr V_i.
\]
Since $\Sigma^{-1}\mathscr V_i\in \Sigma^{-1}\V \subseteq \Sigma \V$ and $\mathscr Y_i\in \V_\tau\subseteq \Sigma \V$, we get that $\mathscr U_i\in \U\cap \Sigma \V=\H$. On the other hand, $\Sigma^{-1}\mathscr V_i\in \Sigma^{-1}\V\subseteq \Sigma^{-1}\V_\tau\subseteq \V_\tau$, and so $\mathscr U_i\in \V_\tau$. These two observations together give us that $\mathscr U_i\in \H\cap \V_\tau=\mathcal F$. Taking now the homotopy colimit of the triangle in \eqref{trunc_tria_lifted_t}, we get the following triangle in $\D(\bbone)$:
\[
\hocolim_I \mathscr U\longrightarrow \hocolim_I \mathscr Y\longrightarrow \hocolim_I \mathscr V\longrightarrow \Sigma \hocolim_I \mathscr U.
\]
As we know that $\mathscr U_i\in \mathcal F\subseteq \H$ for any $i\in I$, we have $\hocolim_I \mathscr U\cong \varinjlim_{I}\mathscr U_i$ and, by our assumptions, this last directed colimit belongs to $\mathcal F$. We can now conclude by noting that $\hocolim_I \mathscr Y\in \mathcal F*\V=\V_\tau$.
\end{proof}

\begin{example}\label{ex.HRS t-structure smash but non-htpy}
Let $\D_\G\colon \cat^{\op}\to \Cat$ be the canonical derivator enhancing the derived category $\Der(\G)$ of a Grothendieck Abelian category $\G$ (Example~\ref{expl.derived categories}), and let $\t=(\U,\Sigma \V)$ be the canonical $t$-structure in $\Der(\G)$. If $\tau=(\mathcal{T},\mathcal{F})$ is a torsion pair in $\G$ that is not of finite type (that is, $\varinjlim\mathcal{F}\neq\mathcal{F}$), then the Happel-Reiten-Smal{\o} tilt $\t_\tau$ of $\t$ with respect to $\tau$ is smashing but not homotopically smashing.

Explicitly we can take $\G=\Ab$ and for $\mathcal T$ the class of all divisible Abelian groups. In this case, the heart $\H_\tau$ of $\t_\tau$ is not an (Ab.$5$) Abelian category. Indeed, given any prime number $p$, we have in $\H_\tau$ a chain of monomorphisms
\[ \Sigma\Z_p \hookrightarrow \Sigma\Z_{p^2} \hookrightarrow \Sigma\Z_{p^3} \hookrightarrow \ldots \]
whose directed colimit is $\tau^{\Sigma\V_\tau}(\Sigma\Z_{p^\infty}) = 0$ by Lemma~\ref{lem.hocolim=colim}.
\end{example}

\subsection{Compactly generated tilts of $t$-structures}

Now we focus on what conditions a torsion pair must satisfy in order for the corresponding tilted $t$-structure to be compactly generated.
We start with a lemma which says that compact objects in the aisle induce finitely presented objects in the heart (cp.\ \cite[Thm.\,6.7]{stovicek2014derived}).

\begin{lemma} \label{lem.cpt to fin pres in H}
Let $\t=(\U,\Sigma\V)$ be a homotopically smashing $t$-structure in $\D(\bbone)$ and  $C \in \U$. If $C$ is compact in $\D(\bbone)$, then $H_\t^0(C)$ is finitely presented in $\H$.
\end{lemma}
\begin{proof}
As $C$ is homotopically finitely presented by Corollary~\ref{coro_compact=hfp}, we have a canonical isomorphism $\varinjlim_I\D(\bbone)(C, \mathscr X_i)\cong \D(\bbone)(C,\hocolim_I \mathscr X)$ for any directed set $I$ and $\mathscr X\in\D(I)$. Furthermore, if $\mathscr X\in \H_I$, then $\mathcal{H}(H_\t^0(C), \mathscr X_i)\cong\D(\bbone)(C, \mathscr X_i)$ for each $i\in I$ and $\mathcal{H}(H_\t^0(C),\hocolim_I \mathscr X)\cong\D(\bbone)(C,\hocolim_I \mathscr X)$, so the canonical morphism $\varinjlim_I\mathcal{H}(H_\t^0(C), \mathscr X_i)\cong \mathcal{H}(H_\t^0(C),\hocolim_I \mathscr X)$ is invertible as well.
The conclusion follows since $\hocolim_I \mathscr X \cong \varinjlim_I \mathscr X_i$ by Corollary~\ref{cor.hocolim=colim}.
\end{proof}

If the $t$-structure is even compactly generated, a much better description of finitely presented objects in the heart was obtained in~\cite[Thm.\,1.6]{SaorinStovicekTstrViaFunctors}. The proof is more involved and we are not going to discuss it here.

\begin{proposition}[{\cite[Thm.\,1.6]{SaorinStovicekTstrViaFunctors}}] \label{prop.lifting_fp_to_compacts_in_the_aisle}
Let $\t=(\U,\Sigma\V)$ be a compactly generated $t$-structure in a triangulated category $\mathcal D$ with coproducts and let $X\in\H=\U\cap\Sigma\V$. Then $X$ is finitely presented in $\H$ if, and only if, there exists a compact object $C\in\U$ such that $X\cong H_\t^0(C)$.
\end{proposition}

We can now give our criterion for a Happel-Reiten-Smal\o{} tilt of a $t$-structure to be compactly generated.

\begin{proposition}\label{prop.homo_fin_pres_is_fin_pres}
Let $\t=(\U,\Sigma\V)$ be a homotopically smashing $t$-structure in $\D(\bbone)$ and let $\t_\tau =(\mathcal{U}_\tau,\Sigma\V_\tau)$ be the tilt of $\t$ with respect to a torsion pair $\tau = (\cal T,\cal F)$ in the heart $\H=\U\cap\Sigma \V$. Consider the following conditions:
\begin{enumerate}[\textup(1\textup)]
\renewcommand{\theenumi}{\textup{\arabic{enumi}}}
\item $\t_\tau$ is compactly generated,
\item there is a set $\mathcal{S}\subseteq\mathcal{T}$ of finitely presented objects of $\H$ with $\mathcal{F}=\bigcap_{S\in\mathcal{S}}\Ker(\mathcal{H}(S,-))$.
\end{enumerate}
Then, (1) implies (2). If, moreover, $\t$ is compactly generated,
then (1) and (2) are equivalent.
\end{proposition}
\begin{proof}
(1)$\Rightarrow$(2). Let $\hat{\mathcal{S}}\subseteq\mathcal{U}_\tau\subseteq\U$ be a set of compact objects such that $\Sigma\hat{\mathcal{S}}\subseteq \hat{\mathcal{S}}$ and which generates $\t_\tau$.
We put $\mathcal{S}:=\{H_\t^0(X):X\in\hat{\mathcal S}\}$,where each $H_\t^0(X)$ is finitely presented in $\H$ by Lemma~\ref{lem.cpt to fin pres in H}. As $\Ker(\mathcal{H}(H_\t^0(X),-)) = \H\cap\Ker(\D(\bbone)(H_\t^0(X),-)) = \H\cap\Ker(\D(\bbone)(X,-))$ for each $X\in\hat{\mathcal{S}}$, we have
\[\xymatrix{
\bigcap_{S\in\mathcal{S}}\Ker(\mathcal{H}(S,-)) =
\H\cap\bigcap_{X\in\hat{\mathcal{S}}}\Ker(\D(\bbone)(X,-)) =
\H\cap\V_\tau =
\mathcal{F}.}
\]

\smallskip\noindent
(2)$\Rightarrow$(1), {\em assuming that $\t$ is compactly generated}.
It is enough to adapt the proof of \cite[Thm.\,2.3]{BP}.
We start with a set of objects which are finitely presented and $\mathcal{F}=\bigcap_{S\in\mathcal{S}}\Ker(\H(S,-))$ in $\H$. For each $S\in\mathcal S$, we fix an object $X_S\in\U$ which is compact in $\D(\bbone)$ and $H_\t^0(X_S)\cong S$. Such an object exists by Proposition~\ref{prop.lifting_fp_to_compacts_in_the_aisle}. If $\mathcal{Y}$ is a set of compact generators of $\t$, we shall prove that $\Sigma\mathcal{Y}\cup\mathcal{X}$ generates $\t_\tau$, where $\mathcal{X}:=\{X_S:S\in\mathcal{S}\}$.
Bearing in mind that $\Sigma\mathcal{Y}\cup\mathcal{X}\subseteq\U_\tau$, our task reduces to prove that $\mathcal{V}':=(\Sigma\mathcal{Y})^{\perp_{\leq 0}}\cap\mathcal{X}^{\perp_{\leq 0}}=(\Sigma\mathcal{Y}\cup\mathcal{X})^{\perp_{\leq 0}}\subseteq\V_{\tau}$.
Since $(\Sigma\mathcal{Y})^{\perp_{\leq 0}} = \Sigma \V$, an object $V\in\D(\bbone)$ is in $\mathcal{V}'$ if, and only if, $V\in\Sigma\V$ and $\D(\bbone)(X_S,V)=0$, for all $S\in\mathcal{S}$.
However, since $X_S\in\U$ for all $S\in\mathcal{S}$, the following series of isomorphisms for each $V\in\Sigma\V$,
\[
\D(\bbone)(X_S,V)\cong \D(\bbone)(H_\t^0(X_S),V)\cong\H(H_\t^0(X_S),H_\t^0(V))=\H(S,H_\t^0(V)),
\]
shows that $V\in\mathcal{V}'$ if, and only if, $V\in\Sigma\V$ and $H_\t^0(V)\in\mathcal{F}$ or, equivalently, $V\in\mathcal{V}_\tau$. 

\end{proof}

\begin{example}\label{ex.HRS t-structure non-cpt gen}
Let $\D_R\colon \cat^{\op}\to \Cat$ be the canonical derivator enhancing the derived category $\Der(\mod R)$ for a ring $R$ (Example~\ref{expl.derived categories}), and let $\t=(\U,\Sigma \V)$ be the canonical \mbox{$t$-structure} on $\Der(\mod R)$.
Suppose that $R$ admits a non-trivial two-sided idempotent ideal $I$ contained in its Jacobson radical $J(R)$ (see \cite{kel94-smashing} for an explicit example and note that such an $I$ must be infinitely generated from either side and, consequently, $R$ must be non-Noetherian, by the Nakayama Lemma). Consider the torsion pair $\tau =(\mathcal{T}_I,\mathcal{F}_I)$, where 
\[
\mathcal{T}_I:=\{T\in\mod R:TI=T\}\quad\text{and}\quad\mathcal{F}_I:=\{F\in\mod R:FI=0\}.
\]
Then, the tilted $t$-structure $\t_\tau$ in $\Der(\mod R)$ is homotopically smashing but not compactly generated.
Indeed, the torsion-free class $\mathcal{F}_I$ is closed under directed colimits in $\mod R$, so $\t_\tau$ is homotopically smashing by Proposition~\ref{prop.homo_smash_iff_colim}. On the other hand, due to the Nakayama Lemma, $\mathcal{T}_I$ does not contain any non-zero finitely generated module. Then $\t_\tau$ cannot be compactly generated as a $t$-structure because of Proposition~\ref{prop.homo_fin_pres_is_fin_pres}.
\end{example}

\begin{remark}
Example~\ref{ex.HRS t-structure non-cpt gen} gives a negative answer to \cite[Quest.\,3.2]{BP}.
\end{remark}

\section{On the existence of a set of generators}\label{sec_generators}

We conclude with the discussion of when the heart of a $t$-structure in the base $\D(\bbone)$ of a strong and stable derivator $\D$ is actually is a Grothendieck Abelian category. Unlike the exactness of directed colimits, the existence of a generator is, to a large extent, a purely technical condition which is usually satisfied in examples ``from the nature''. In the world of $\infty$-categories, a condition for the heart of a $t$-structure to be Grothendieck Abelian was given by Lurie in~\cite[Rem.\,1.3.5.23]{Lurie_higher_algebra}. We give a similar (and short) discussion also in our setting, restricting to derivators of the form $\D=\D_{(\C,\W)}$, where $(\C,\W,\cal B,\cal F)$ is a combinatorial stable model category (see Example \ref{expl.derived categories}). These derivators are called ``derivators of small presentation'' and they can be characterized internally to the $2$-category of derivators, see \cite{Renaudin:aa}.

Recall that a model category $(\C,\W,\cal B,\cal F)$ is {\em combinatorial} if 
\begin{itemize}
\item $\C$ is a locally presentable category (see \cite{gabriel1971lokal, adamek1994locally});
\item the model structure $(\W,\cal B,\cal F)$ is cofibrantly generated. 
\end{itemize}
These model categories were introduced by J.\ Smith. Let us recall here the following properties whose proof essentially follows by  \cite[Prop.\,2.3 and 7.3]{dugger-comb-model} or \cite[Prop.\,2.5]{Barwick-localizations}:

\begin{lemma}\label{basic_combinatorial_lem}
Let $(\C,\W,\cal B,\cal F)$ be a combinatorial model category and $\mathcal S\subseteq \C$ a  set of objects. Then, there exists an infinite regular cardinal $\lambda$ such that
\begin{enumerate}[\textup(1\textup)]
\renewcommand{\theenumi}{\textup{\arabic{enumi}}}
\item the functors $\C(S,-)$ commute with $\lambda$-directed colimits for all $S\in \mathcal S$;
\item there are co/fibrant replacement functors which preserves $\lambda$-directed colimits;
\item $\lambda$-directed colimits of weak equivalences are again weak equivalences;
\item given a $\lambda$-directed set $I$ and a diagram $X\in \C^I$, there is an isomorphism in $\ho(\C)$
\[
\textbf{L}\varinjlim_IX\cong F(\varinjlim_I X), 
\]
where $F\colon \C\to\ho(\C)$ is the canonical functor. 
\end{enumerate}
\end{lemma}
\begin{proof}
(1) follows since $\C$ is locally presentable, while (2,3) are \cite[Prop.\,2.3 (i,ii)]{dugger-comb-model}.

\smallskip\noindent
(4). Given a $\lambda$-directed set $I$, since $\varinjlim_I$ induces a functor $\W_I \to \W$, the universal property of $F_I\colon \C^I\to\ho(\C^I)=\C^I[\W_{I}^{-1}]$ yields a unique functor completing the following solid diagram to a commutative square:
\[
\xymatrix@C=50pt{
\C^I\ar[d]_{\varinjlim_I}\ar[r]^-{F_I}&\ho(\C^I)\ar@{.>}[d]\\
\C\ar[r]_-{F}&\ho(C)
}
\]
Of course, such a functor automatically satisfies the universal property for being the total left derived functor of $\varinjlim_I$, hence we deduce the isomorphism in the statement.
\end{proof}

As remarked in Example \ref{expl.derived categories}, for derivators arising from model categories, the functor $\hocolim_I\colon \ho(\C^I)\to \ho(\C)$ is the total left derived functor of $\varinjlim_I\colon \C^I\to \C$. Using this identification, one easily deduces the following corollary:

\begin{corollary}\label{cor.first_commutativity}
Let $(\C,\W,\cal B,\cal F)$ be a combinatorial model category, $\D = \D_{(\C,\W)}$ the induced strong derivator as in Example~\ref{expl.derived categories} and $\lambda$ an infinite regular cardinal satisfying (4) in Lemma \ref{basic_combinatorial_lem}. If $I$ is a $\lambda$-directed and $X\in \C^I$, then there is an isomorphism in $\ho(\C)$
\[
\xymatrix{
\hocolim_IX\cong F(\varinjlim_I X),
} 
\]
where $F\colon \C\to\ho(\C)$ is the canonical functor. 
\end{corollary}

We can prove the existence of a generator for the following large class of $t$-structures which includes all homotopically smashing ones (they constitute a special case for $\lambda=\aleph_0$):

\begin{definition} \label{def.accessibly-embedded}
Let $(\C,\W,\cal B,\cal F)$ be a stable and combinatorial model category, and $\t=(\U,\Sigma\V)$  a $t$-structure in $\ho(\C)$. We say that $\t$ is \emph{$\lambda$-accessibly embedded} in $\ho(\C)$, with $\lambda$ an infinite regular cardinal, if $\V$ is closed under $\lambda$-directed homotopy colimits in $\D_{(\C,\W)}$.
\end{definition}

\begin{proposition}\label{prop.lambda-colim}
Let $(\C,\W,\cal B,\cal F)$ be a stable and combinatorial model category, $\lambda$ an infinite regular cardinal satisfying (3) and (4) in Lemma \ref{basic_combinatorial_lem}, and $\t=(\U,\Sigma\V)$ a $\lambda$-accessibly embedded $t$-structure with heart $\H=\U\cap\Sigma\V$ in $\ho(\C)$.
Then, the following composition functor preserves $\lambda$-directed colimits:
\[
\C\overset{F}{\longrightarrow}\ho(\C)\overset{H^0}{\longrightarrow}\H.
\]
\end{proposition}
\begin{proof}
Given a $\lambda$-directed set $I$, consider the following diagram:
\[
\xymatrix@C=50pt{
\C^I\ar[r]^-{F_I}\ar[d]_{\varinjlim_I}&\ho(\C^I)\ar[r]^-{H^0_I}\ar[d]|{\hocolim_I}&\H_I\ar[d]^{\hocolim_I}\\
\C\ar[r]_-F&\ho(\C)\ar[r]_-{H^0}&\H.
}
\]
By Theorem~\ref{mainthm.lift_tstructure_general_thm}, $\H_I\cong \H^I$ and, identifying these two categories, $(\hocolim_I)_{\restriction\H_I}$ is conjugated to $\varinjlim_I\colon \H^I\to \H$ (see the proof of Corollary~\ref{cor.hocolim=colim}). This observation tells us that it is enough to show that the external square in the above diagram commutes. We  verify instead that the smaller squares do commute. In fact, the commutativity of the square on the left hand side is given by Corollary~\ref{cor.first_commutativity}, while the commutativity of the square on the right hand side follows from the fact that both $\U$ (by Proposition~\ref{prop.aisles closed for hocolim}) and $\V$ (by assumption) are closed under $\lambda$-directed homotopy colimits.
\end{proof}

Before we prove, as a main result of the section, Theorem~\ref{mainthm.Grothendieck heart} from the introduction, we give a discussion of its assumptions:

\begin{remark} \label{rem.well-gen}
A triangulated category is called \emph{algebraic} if it is the stable category of a Frobenius exact category~\cite[Sec.\,7.5]{Krause-ChicagoNotes}. Compactly generated algebraic triangulated categories are then triangle equivalent to the derived categories of small dg categories, and such derived categories are the homotopy categories of  combinatorial model categories of dg modules. We refer to \cite[Sec.\,2.2 and 2.3]{stovivcek2016compactly} for a more detailed discussion.

More generally, one can consider algebraic triangulated categories which are well generated in the sense of~\cite{Neeman}. These are, by~\cite[Thm.\,7.2]{porta2010popescu}, simply localizations of compactly generated algebraic triangulated categories $\cal D$ with respect to localizing subcategories $\cal L$ generated by a small set $\cal S$ of objects. Such a localization is also the homotopy category of a combinatorial model category. 
As an upshot, each well generated algebraic triangulated category is the homotopy category of a combinatorial model category. 

An analogous result for well generated \emph{topological} triangulated categories (i.e.,  homotopy categories of spectral model categories) can be found in~\cite[Thm.\,4.7 and 5.11]{heider2007}.
\end{remark}

\begin{lemma}
Let $(\C,\W,\cal B,\cal F)$ be a stable combinatorial model category and $\t=(\U,\Sigma \V)$ a $t$-structure in $\ho(\C)$ generated by a  set $\cal S$ of objects of $\C$. Then there exists an infinite regular cardinal $\lambda$ such that $\t$ is $\lambda$-accessibly embedded. 
\end{lemma}
\begin{proof}
Up to replacing $\mathcal S$ by another set of the same cardinality, we can assume that any object in $\mathcal S$ is cofibrant. Let now $\lambda$ be an infinite regular cardinal with the properties (1--4) described in Lemma \ref{basic_combinatorial_lem}, and fix a fibrant replacement functor $R$ that commutes with $\lambda$-directed colimits. We have to prove that $\V=\mathcal S^{\perp}$ is closed under taking $\lambda$-directed homotopy colimits. Indeed, let $I$ be a $\lambda$-directed set and $X\in \C^I$ such that $X_i\in \V$ for all $i\in I$. For any $C\in \mathcal S$ and $i\in I$, there is a commutative diagram as follows:
\[
\xymatrix{
\C(C, RX_i)\ar@{->>}[d]_{(*)}\ar[rr]&&\C(C, \varinjlim_IRX_i)\ar@{->>}[d]^{(**)}\\
0=\ho(\C)(C, RX_i)\ar[rr]&&\ho(\C)(C, \hocolim_IRX_i)
}
\]
where $(*)$ is surjective since $C$ is cofibrant and $RX_i$ is fibrant, while $(**)$ is surjective since $ \hocolim_IRX_i\cong \varinjlim_IRX_i\cong R(\varinjlim_IX_i)$ by Lemma \ref{basic_combinatorial_lem} and the fact that $R$ commutes with $\lambda$-directed colimits. Using the universal property of $\varinjlim_I$ in $\Ab$, we obtain the following commutative diagram:
\[
\xymatrix{
\varinjlim_I\C(C, RX_i)\ar@{->>}[d]\ar[rr]^{\cong}&&\C(C, \varinjlim_IRX_i)\ar@{->>}[d]\\
0=\varinjlim_I\ho(\C)(C, RX_i)\ar[rr]&&\ho(\C)(C, \hocolim_IRX_i)
}
\]
where the top row is invertible by  Lemma \ref{basic_combinatorial_lem}, so that $\ho(\C)(C, \hocolim_IRX_i)=0$, for all $C\in \mathcal S$, and so $\hocolim_IX_i\cong \hocolim_IRX_i\in \mathcal S^{\perp}=\V$.
\end{proof}
\color{black}

\begin{proof}[Proof of Theorem~\ref{mainthm.Grothendieck heart}]
We remind the reader that we need to prove the following. If $(\C,\W,\cal B,\cal F)$ is a stable combinatorial model category and $\t=(\U,\Sigma\V)$ is a $\lambda$-accessibly embedded $t$-structure in $\ho(\C)$, then the heart $\H=\U\cap\Sigma\V$ has a generator.
We may also, without loss of generality, assume that $\lambda$ satisfies conditions (3) and (4) from Lemma \ref{basic_combinatorial_lem}.
Since the category $\C$ is locally presentable, there exists a set $\Q$ of objects of $\C$ such that every object $C\in\C$ is a $\lambda$-directed colimit $C = \varinjlim_{I_C} (Q_i)$ of a $\lambda$-direct system $(Q_i)_{i\in I_C}$ in $\Q$. Consider the following set of objects in $\H$:
\[
\overline{\mathcal Q}:=\{H^0(Q):Q\in \Q\}.
\]
Now notice that
\[
H^0(F(C))\cong H^0(F(\varinjlim_{I_C}Q_i))\overset{(*)}\cong \varinjlim_{I_C}H^0(F(Q_i)),
\]
where $(*)$ follows by Proposition~\ref{prop.lambda-colim}. 
Since any object in $\H$ is of the form $H^0(F(C))$ for some $C\in\C$, we have just verified that $\overline{ \mathcal Q}$ is a set of generators for $\H$. 
\end{proof}

\begin{proof}[Proof of Corollary~\ref{maincor.Grothendieck heart}]
If the $t$-structure in the above proof is compactly generated, then it is homotopically smashing by Proposition~\ref{prop.homtopsmashing-implies-smashing}. Since the heart is (Ab.$5$) by Theorem~\ref{mainthm.htpy smashing has AB5 heart} and it has a generator by Theorem~\ref{mainthm.Grothendieck heart}, it is a Grothendieck Abelian category.
\end{proof}

\begin{remark}\label{comparison_rem}
Corollary~\ref{maincor.Grothendieck heart} is a straightforward consequence of Theorems~\ref{mainthm.htpy smashing has AB5 heart} and \ref{mainthm.Grothendieck heart}.
On the other hand, under the same set of hypotheses, one even has that $\H=\varinjlim H^0_\t(\mathcal T)$, where $\mathcal S$ is a set of compact generators for $\t$ and $\mathcal T$ is the smallest subcategory of $\ho(\C)$ containing  $\mathcal S$ which is closed under extensions, suspension and direct summands.
Indeed, $\mathcal T$ is precisely the class of compact objects of $\ho(\C)$ which are contained in the aisle $\U$. We refer to~\cite[Thm.\,4.2.1(2)]{bondarko} or~\cite[Thm.\,4.5]{stovivcek2016compactly} for this fact. The conclusion then follows from~\cite[Thm.\,1.6]{SaorinStovicekTstrViaFunctors}, where it was shown, while this paper was under review, that the heart $\H$ is locally finitely presentable and $H^0_\t(\mathcal T)$ is precisely the class of finitely presented objects (recall Proposition~\ref{prop.lifting_fp_to_compacts_in_the_aisle}).
%
\end{remark}

\appendix
\section*{Appendix A: Directed homotopy colimits of spectra}
\label{sec_spectra}
\addcontentsline{toc}{section}{Appendix A: Directed homotopy colimits of spectra}
\setcounter{section}{1}

The main point of this appendix is to establish that the stable homotopy groups of spectra (in the topological sense) commute with (homotopy) directed colimits.
In the terminology of Subsection~\ref{subs_homofpintr}, one can equivalently say that the sphere spectrum is intrinsically homotopically finitely presented in $\D_\Sp$.
This can be viewed as a topological analogue of the fact that the cohomology groups of complexes of Abelian groups commute with directed colimits. Although this result seems to be well-known to experts in homotopy theory (see, e.g., \cite[Thm.\,1.7]{Adams71} or~\cite[Rem.\,1, p.~331]{Switzer}), we are lacking an adequate reference.

\medskip\noindent
\textbf{Model structures on categories of diagrams.} 
When discussing homotopy colimits in detail, we are confronted with the following question: Given a model category $\C$ with the class of weak equivalences $\W$ and a small category $I$, is there a suitable model structure on the diagram category $\C^I$ with pointwise weak equivalences? Although the existence of such model structures is not clear in full generality and for many purposes one can work around this problem with other techniques (cf.~\cite{Cis03}), they nevertheless exist in many situations and make the discussion easier there. Most notably, one usually considers two ``extremal'' model structures, provided that they exist:

\begin{definition} \label{def.diagram model structures}
	Given a model category $(\C, \W,\cal B,\cal F)$ and a small category $I$,  a model structure on $\C^I$ is called
	\begin{enumerate}[\textup(1\textup)]
		\item the \emph{projective model structure} if the weak equivalences and fibrations are defined pointwise (i.e., a morphism $f\colon X \to Y$ in $\C^I$ is a weak equivalence or fibration if $f_i\colon X_i \to Y_i$ is a weak equivalence or fibration, respectively, in $\C$ for each $i\in I$);
		\item the \emph{injective model structure} if the weak equivalences and cofibrations are defined pointwise.
	\end{enumerate}
\end{definition}

Note that both the projective and the injective model structures, when they exist, are unique. It is also a standard fact that the class of injective fibrations is included in the class of projective fibrations and dually for cofibrations. To see this, note that for each $i\in I$, the evaluation functor $(-)_{\restriction{i}}\colon \C^I \to \C$, $X \mapsto X(i)$ has a left adjoint $-\otimes i\colon \C \to \C^I$ given by $(X\otimes i)(k)\cong \coprod_{I(i,k)}X$, for all $k\in I$ (see Example \ref{coh_incoh_rest_ext_example}). One readily checks that $-\otimes i$ preserves cofibrations and trivial cofibrations if $\C^I$ is equipped with the injective model structure. Thus, $(-\otimes i, (-)_{\restriction{i}})$ is a Quillen adjunction and $(-)_{\restriction{i}}$ sends injective fibrations to fibrations in $\C$ for each $i\in I$. Another fact which we will need is the following observation regarding the structure of fibrant and cofibrant objects.

\begin{lemma} \label{lem.structure fibrant cofibrant}
	Let $(\C,\W,\cal B,\cal F)$ be a model category and $(I,\le)$  a partially ordered set (viewed as a small category). If a projective model structure exists, $X\in\C^I$ is a projectively cofibrant object and $i\le j$ are elements of $I$, then $X(i) \to X(j)$ is a cofibration in $\C$. Dually, if an injective model structure exists and $X$ is an injectively fibrant object, then $X(i) \to X(j)$ is a fibration in $\C$. 
\end{lemma}

\begin{proof}
	We  only prove the part regarding the injective model structure, the other part is dual. In view of \cite[Lem.\,1.1.10]{Hovey_libro}, we need to prove that, given any commutative square
	\begin{equation}\label{eq.commutative square injectively fibrant}
	\vcenter{
		\xymatrix{
			U \ar[d]_-{c} \ar[r]^-{u} & X(i) \ar[d] \\
			V \ar[r]_-{v} \ar@{.>}[ur]|\hole|h & X(j)
		}
	}
	\end{equation}
	in $\C$, where $c\colon U \to V$ is a trivial cofibration, the dotted arrow can be filled in so that both triangles commute.	To construct $h$, we start with the fact that $X$ is injectively fibrant, i.e., $\C(V',X) \to \C(U',X)$ is surjective for any pointwise trivial cofibration $f'\colon U' \to V'$ in $\C^I$. We apply this property to a specially crafted $f'$, based on the morphism $f$ above. We let $V' = V \otimes i$, i.e., $V'(k) = V$ if $k \ge i$ and is the initial object $0\in\C$ otherwise. Let also:
	\[
	U'(k) = 
	\begin{cases}
	$V$ & \textrm{if } k \ge j;\\
	$U$ & \textrm{if } k \ge i \textrm{ and } k \not\ge j;\\
	0  & \textrm{otherwise.}
	\end{cases}
	\]
	The morphisms $U'(k) \to U'(\ell)$ are copies of $c$ if $k\ge i$, $k\not\ge j$ and $\ell\ge j$, and the identity morphisms or the morphisms from the initial object, otherwise. There is an obvious morphism $c'\colon U' \to V'$ whose components are just identity morphisms and copies of $c$. Furthermore, the commutative square~\eqref{eq.commutative square injectively fibrant} allows us to define a morphism $u'\colon U' \to X$ such that
\begin{enumerate}
\item	$u'(k)$ is the composition of $v$ with $X(j)\to X(k)$ if $k\ge j$,
\item $u'(k)$ is the composition of $u$ with $X(i)\to X(k)$ if $k\ge i$, but $k\not\ge j$ and
\item $u'(k)$ is the morphism from the initial object otherwise.
\end{enumerate}

As mentioned above, our assumption dictates that $u'\colon U' \to X$ factors through $c'\colon U' \to V'$ via a map $h'\colon V' \to X$. The restriction of the morphism $c'$ and $h'$ to the components $i,\, j\in I$ yields the following commutative diagram in $\C$,
	\[
	\xymatrix@C=50pt{
		U \ar[d]_-{c} \ar[r]^-{c} & V \ar@{=}[d] \ar[r]^-{h'(i)} & X(i) \ar[d] \\
		V \ar@{=}[r] & V \ar[r]_-{h'(j)} & X(j),
	}
	\]
	where the compositions in the rows are $u$ and $v$ respectively. It follows that $v=h'(j)$ and $h = h'(i)$ fits into~\eqref{eq.commutative square injectively fibrant}.
\end{proof}

Our main motivation for considering the projective model structure is that, if it exists, the constant diagram functor $\kappa_I\colon \C\to \C^I$  preserves fibrations and trivial fibrations, so that $(\colim_I, \kappa_I)$ is a Quillen adjunction. In other words, the  total left derived functor $\mathbf{L}\colim_I\colon \ho(\C^I) \to \ho(\C)$ exists and it can be computed using projectively cofibrant resolutions of diagrams. We will denote $\mathbf{L}\colim_I$ by $\hocolim_I$ and call it the \emph{homotopy colimit} functor. Of course, formally dual statements hold for limits and the injective model structure.

Finally, we touch the problem of the existence of projective an injective model structures in the case of combinatorial model structures (see Section \ref{sec_generators}):

\begin{proposition}[{\cite[Prop.\,A.2.8.2]{HTT}}] \label{prop.diagram model structures}
	Let $(\C,\W,\cal B,\cal F)$ be a combinatorial model category and $I$ a small category. Then the diagram category $\C^I$ admits both the projective and the injective model structures.
\end{proposition}

\medskip\noindent
\textbf{Model categories of simplicial sets and spectra.}
Let $\Delta$ be the category with all finite ordinals $\mathbf{1},\, \mathbf{2},\, \mathbf{3}, \ldots$, where $\mathbf{n}=\{0\to 1\to\ldots\to (n-1)\}$, as objects and order-preserving maps as morphisms. Here we keep our convention for ordinals from set-theory rather than from topology (there one often denotes by $[n]$ the ordinal $\mathbf{n+1}$ as it indexes the $n$-dimensional simplices). The category $\sSet$ of \emph{simplicial sets} is defined as the category of functors $\Delta^{\op}\to\Set$. As customary, we denote by $\Delta[n] := \Delta(-,(\mathbf{n+1}))$ the representable functors in $\sSet$.

A simplicial set can be viewed as a combinatorial model for a topological space. More precisely, there is a \emph{geometric realization functor} $\lvert-\rvert\colon \sSet \to \Top$, $X \in \sSet \mapsto \lvert X\rvert \in \Top$. This is a left adjoint functor which preserves finite limits up to a weak homotopy equivalence. More in detail, recall that a continuous map if topological spaces $f\colon X \to Y$ is a \emph{weak homotopy equivalence} if $\pi_n(f)\colon \pi_n(X,x) \to \pi_n(Y,f(x))$ is a bijection for all $n\ge 0$ and all base points $x\in X(\mathbf{1})$ (here, $\pi_0(X,x)$ is just the set of all path components of $X$, while $\pi_n(X,x)$ is a group for $n\ge 1$ and an Abelian group for $n\ge 2$). Now if $I\in\cat$ is a finite category and $X\in\sSet^I$, the canonical map $\lvert \lim_{i\in I} X_i\rvert\to\lim_{i\in I}\lvert X_i\rvert$ is a bijection of sets, but the topologies may not agree. Nevertheless, it is always a weak homotopy equivalence---see~\cite[Sec.\,III.3]{GZ} for details.
 The topological spaces of the form $\lvert X\rvert$ have very nice properties as they naturally have a structure of CW-complexes \cite[Def.\,5.3]{Switzer}. In fact, if $X\subseteq Y$ is an inclusion of simplicial sets, then $\lvert X\rvert$ is naturally a CW-subcomplex of $\lvert Y\rvert$ in the sense of \cite[Def.\,5.8]{Switzer}. This follows, e.g., directly from the proof of~\cite[Prop.\,3.2.2]{Hovey_libro}.

Given $X \in \sSet$, $x\in X(\mathbf{1})$ and $n\ge 0$, one defines $\pi_n(X,x):=\pi_n(\lvert X\rvert,\lvert x\rvert )$ with the base point corresponding to $x$. It is also possible to define $\pi_n(X,x)$ combinatorially directly on the simplicial set $X$, see~\cite[Sec.\,I.7 and III.4]{GoerssJardine}.
The category $\sSet$ comes equipped with a standard model structure such that cofibrations are inclusions of simplicial sets, fibrations are the so-called Kan fibrations, and $f\colon X \to Y$ is a weak equivalence if $\pi_n(f)\colon \pi_n(X,x) \to \pi_n(Y,f(x))$ is a bijection for all $n\ge 0$ and all base points $x\in X(\mathbf{1})$ (or, equivalently, by \cite[Thm.\,6.32]{Switzer}, if $\lvert f\rvert$ is a homotopy equivalence of topological spaces).

Since homotopy groups need a choice of a base point, it will be helpful to work with the category $\sSet_*$ of {pointed simplicial sets}, which is simply the slice category $\Delta[0]/\sSet$. In pedestrian terms, its objects are pairs $(X,x)$, where $X\in\sSet$ and $x\in X(\mathbf{1})$, and the morphisms are base point-preserving morphisms of simplicial sets. This category also carries a model structure where a morphism $f\colon (X,x) \to (Y,y)$ is a weak equivalence, cofibration or fibration if the underlying map $X\to Y$ is such in $\sSet$ (see~\cite[Prop.\,1.1.8]{Hovey_libro}).

If $I$ is a directed set and $(X_i,x_i)_{i\in I} \in \sSet_*^I$, then there is a canonical homomorphism $\varinjlim_I\pi_n(X_i,x_i) \to \pi_n(\hocolim_I(X_i,x_i))$ for each $n\ge 0$. More explicitly, we first replace $(X_i,x_i)_{i\in I}$ by a pointwise weakly equivalent projectively cofibrant diagram $(X'_i,x'_i)$, then apply $\pi_n$ to the adjunction unit $\eta\colon (X'_i,x'_i) \to \kappa_I(\varinjlim_I (X'_i,x'_i)) = \kappa_I(\hocolim_I (X_i,x_i))$, and finally take the colimit map of the resulting cocone of sets or groups.

\begin{lemma} \label{lem.htpy groups and sSet_*}
	In the above setting, the map $\varinjlim_I\pi_n(X_i,x_i) \to \pi_n(\hocolim_I(X_i,x_i))$ is a bijection.
\end{lemma}

\begin{proof}
	The morphisms $X_i \to X_j$ in the projectively cofibrant diagram $(X'_i,x'_i)\in\sSet_*$ are inclusions by Lemma~\ref{lem.structure fibrant cofibrant}, so we can assume that $X'_i$ are simplicial subsets of $\varinjlim_I X'_i$ and that $\varinjlim_I X'_i$ is the directed union of the $X'_i$.
	As any compact subset of $\lvert \varinjlim_I X'_i \rvert$ is contained in $\lvert X_i\rvert$, for some $i\in I$, by \cite[Prop.\,5.7]{Switzer}, we can conclude by~\cite[Prop.\,7.52]{Switzer}.
\end{proof}

Since our main object of interest are stable Grothendieck derivators and any stable derivator is enriched over spectra by \cite[Appendix A.3]{CisTab11},  our main goal is the analogue of Lemma~\ref{lem.htpy groups and sSet_*} for spectra. To this end, let us quickly recall one model structure for the category of spectra from~\cite{BF-spectra}.

First note that the category $\sSet$ carries a Cartesian (closed) symmetric monoidal structure $(\sSet,\times,\Delta[0])$. Now the forgetful functor $\sSet^*\to \sSet$ has a left adjoint which sends $X\in\sSet$ to the disjoint union $X_+ := X \cup \Delta[0]$, and this adjunction allows to define a unique monoidal structure $(\sSet,\wedge,\Delta[0]_+)$ such that $\wedge$ preserves colimits in each component and $X \mapsto X_+$ is a monoidal functor. The functor $\wedge$ is called the \emph{smash product}.

The key object for the definition of a spectrum is a combinatorial model for the topological circle.
We can define $\mathbb{S}^1\in\sSet$ as the coequalizer of $\Delta[0] \rightrightarrows \Delta[1]$, where the two maps come from the two morphisms $\mathbf{1}\to\mathbf{2}$ in $\Delta$. As this construction gives a canonical map $\Delta[0]\to\mathbb{S}^1$, we can view $\mathbb{S}^1$ as an object of $\sSet_*$.

\begin{definition} \label{def.spectrum}
A \emph{spectrum} is a sequence $X = (X^n,\sigma^n)$ indexed by the natural numbers such that $X^n\in\sSet_*$ and $\sigma^n\colon \mathbb{S}^1\wedge X^n \to X^{n+1}$ is a map of pointed simplicial sets for each $n\in\N$. Maps of spectra $f\colon (X^n,\sigma^n) \to (Y^n,\sigma'^n)$ are defined in the obvious way as collections of maps $f^n\colon X^n\to Y^n$ such that $f^{n+1}\sigma^n = \sigma'^n (\mathbb{S}^1\wedge f^n)$ for all $n$. The category of spectra will be denoted by $\Sp$.
\end{definition}

For a spectrum $X = (X^n,\sigma^n)$ and $m\in\Z$, we can define the \emph{$m$-th stable homotopy group}
\[ \pi^s_m(X) := \varinjlim_{n\gg 0} \pi_{m+n}(X^n) \]
(see~\cite[8.21]{Switzer} for details). In fact, $\pi_m^s$ is a functor $\Sp \to \Ab$ and one can define the class $\W$ of weak equivalences on $\Sp$ as those morphisms $f\colon X \to Y$ for which $\pi_m^s(f)$ is an isomorphism for all $m\in\Z$.

The class $\W$ is a part of the \emph{Bousfield-Friedlander model structure} $(\Sp,\W,\cal B,\cal F)$, \cite[Sec.\,2]{BF-spectra}. A morphism $f\colon X \to Y$ in $\Sp$ a cofibration in this model structure if $f^0\colon X^0 \to Y^0$ and the pushout morphisms
\[ X^{n+1} \sqcup_{\mathbb{S}^1\wedge X^n} (\mathbb{S}^1\wedge Y^n) \longrightarrow Y^{n+1} \]
are cofibrations (i.e., inclusions) in $\sSet_*$ for all $n\ge 0$. It follows by induction on $n$ (e.g., using~\cite[Coro.\,1.1.11]{Hovey_libro}) that all $f^n\colon X^n\to Y^n$ are cofibrations in $\sSet_*$. Fibrant spectra $X$ are those for which each $X^n$ is fibrant in $\sSet_*$ (i.e., $X^n$ is a Kan complex) and the maps of topological spaces $\lvert X^n\rvert \to \lvert X^{n+1}\rvert^{\lvert\mathbb{S}^1\rvert}$ adjoint to $\vert\sigma^n\rvert\colon \lvert\mathbb{S}^1\rvert \wedge \lvert X^n\rvert \to \lvert X^{n+1}\rvert$ are weak homotopy equivalences (i.e., they induce bijections for all $\pi_n(-,x)$ with $x \in \lvert X^n\rvert$ and $n\ge 0$).
This model structure is known to be combinatorial (to see  this, either combine \cite[Thm.\,4.7]{Barwick-localizations} with \cite[Sec.\,3]{Hovey-spectra}, or directly apply \cite[Thm.\,10.5]{Jardine-local-htpy}). Now we can prove the desired result (cp.\ \cite[Lem.\,8.34]{Switzer}):

\begin{proposition} \label{prop.htpy groups and spectra}
Let $I$ be a directed set and $X \in \Sp^I$. Then the canonical map $\varinjlim_I \pi_m^s(X(i)) \to \pi_m^s(\hocolim_I X)$ is an isomorphism for each $m\in\Z$.
\end{proposition}

\begin{proof}
The projective model structure on $\Sp^I$ exists by Proposition~\ref{prop.diagram model structures}. If $X'$ is a projectively cofibrant replacement of $X$, the maps $X'(i)^n \to X'(j)^n$ of pointed simplicial sets are inclusions (by Lemma~\ref{lem.structure fibrant cofibrant}) for each $i\le j$ in $I$ and each $n \ge 0$. Using Lemma~\ref{lem.htpy groups and sSet_*},
\[
\xymatrix{
\pi_m^s(\varinjlim_I X') =
\varinjlim_{n\gg 0}\pi_{m+n}(\varinjlim_I X'(i)^n) \cong
\varinjlim_I\varinjlim_{n\gg 0} \pi_{m+n}(X'(i)^n) =
\varinjlim_I\pi_m^s(X'(i)).}
\]
We conclude by combining this with $\hocolim_I X \cong \varinjlim_I X'$.
\end{proof}

\begin{remark} \label{rem.sphere spectrum htpy fp}
The category $\Sp$ is a stable model category, so $\D_\Sp$ is a strong, stable derivator and $\ho(\Sp) = \D_\Sp(\bbone)$ is a triangulated category. 
There is a privileged object in $\ho(\Sp)$ which plays the role of $\mathbb{Z}$ in $\Der(\Z)$---the \emph{sphere spectrum} $\mathbb{S} = (\mathbb{S}^n, \id_{\mathbb{S}^{n+1}})$, where for each $n\ge0$ we define
\[ \mathbb{S}^n := \underbrace{\mathbb{S}^1\wedge\ldots\wedge\mathbb{S}^1}_{n\;\mathrm{times}} \]
(we formally put $\mathbb{S}^0 = \Delta[0]_+$, the monoidal unit for the smash product).
\end{remark}

\begin{corollary}
The sphere spectrum $\mathbb{S}$ is intrinsically homotopically finitely presented in  $\D_\Sp$.
\end{corollary}

\begin{proof}
 Since there is a natural equivalence $\ho(\Sp)(\mathbb{S},-) \cong \pi_0^s$ of functors $\ho(\Sp)\to\Ab$, the result immediately follows by Proposition~\ref{prop.htpy groups and spectra}.
\end{proof}

\addcontentsline{toc}{section}{References}
\bibliographystyle{alpha}
\bibliography{refs}

\medskip
\noindent\rule{5cm}{0.4pt}

\medskip
Manuel Saor\'{\i}n -- \texttt{msaorinc@um.es}\\
{Departamento de Matem\'{a}ticas,
Universidad de Murcia, Aptdo. 4021,
30100 Espinardo, Murcia,
SPAIN}

\medskip
{Jan \v{S}\v{t}ov\'{\i}\v{c}ek -- \texttt{stovicek@karlin.mff.cuni.cz}\\
Department of Algebra,
Faculty of Mathematics and Physics,
Charles University in Prague,
Sokolovsk\'{a} 83, 18675 Praha 8,
CZECH REPUBLIC}

\medskip
Simone Virili -- \texttt{s.virili@um.es} or \texttt{virili.simone@gmail.com}\\
{Departamento de Matem\'{a}ticas,
Universidad de Murcia,  Aptdo. 4021,
30100 Espinardo, Murcia,
SPAIN}

\end{document}